\documentclass[12pt]{amsart}

\usepackage{graphicx} 
\usepackage{amsmath}
\usepackage{amsthm}
\usepackage{amssymb}
\usepackage{amsfonts}
\usepackage{amsxtra}
\usepackage{enumerate}
\usepackage{mathrsfs}

\usepackage{hyperref}

\usepackage{ifthen}

\input xy
\xyoption{all}

\newtheorem{theorem}[subsubsection]{Theorem}
\newtheorem{introtheorem}{Theorem}

\newtheorem{lemma}[subsubsection]{Lemma}
\newtheorem{proposition}[subsubsection]{Proposition}
\newtheorem{corollary}[subsubsection]{Corollary}
\newtheorem{deprop}[subsubsection]{Definition/Proposition}
\newtheorem{conjecture}[subsubsection]{Conjecture}
\newtheorem{introconjecture}{Conjecture}

\newtheorem{definition}[subsubsection]{Definition}
\newtheorem{remark}[subsubsection]{Remark}

\newcommand{\nc}[1]{\newcommand{#1}}

\DeclareMathOperator{\codim}{codim}

\DeclareMathOperator{\Quot}{Quot}
\DeclareMathOperator{\supp}{supp}

\DeclareMathOperator{\ch}{ch}
 
\DeclareMathOperator{\quo}{qu} 
\DeclareMathOperator{\Lie}{Lie}
\DeclareMathOperator{\Rep}{Rep}

\nc{\xox}[2]{\textsf{#2}} 
\nc{\R}{\mathbb{R}}
\nc{\C}{\mathbb{C}}
\nc{\Z}{\mathbb{Z}}
\nc{\Q}{\mathbb{Q}}
\nc{\N}{\mathbb{N}}
\nc{\id}{id}
\nc{\W}[2]{W_{#1}\!\setminus \! W/W_{#2}} 
\nc{\Gr}[3]{{}^{#1}\! \operatorname{Gr}_{#2}^{#3}} 
\nc{\SO}{\mathcal{O}} 
\nc{\on}[2]{{}^{#1}\vartheta^{#2}} 
\nc{\out}[2]{{}^{#1}\vartheta^{#2}} 
\nc{\add}{\operatorname{add}} 
\nc{\expe}[1]{1^{#1}}
\nc{\lMod}[1]{\text{${#1}$-Mod}}
\nc{\rMod}[1]{\text{Mod-${#1}$}}
\nc{\bMod}[2]{\text{${#1}$-Mod-${#2}$}}
\nc{\lmod}[1]{\text{${#1}$-mod}}
\nc{\bimod}[2]{\text{${#1}$-mod-${#2}$}}
\nc{\Rloc}[1]{R^{(#1)}}
\nc{\Dmax}[2]{{}^{#1}D^{#2}} 
\nc{\Dmin}[2]{{}_{#1}D_{#2}} 

\nc{\RR}[2]{{}^{#1}\mathcal{R}^{#2}} 

\nc{\delflag}[2]{{}^{#1}\mathcal{F}_{\Delta}^{#2}} 
\nc{\nabflag}[2]{{}^{#1}\mathcal{F}_{\nabla}^{#2}} 

\nc{\nabchar}{\ch_{\nabla}} 
\nc{\delchar}{\ch_{\Delta}} 

\nc{\BB}[2]{{}^{#1}\mathcal{B}^{#2}}
\nc{\preBB}[2]{{}^{#1}\mathcal{B}_{BS}^{#2}}

\nc{\PP}[2]{{}^{#1}\mathcal{P}^{#2}} 

\nc{\RS}[3]{{}^{#1}\!R_{#2}^{#3}} 
\nc{\del}[3]{{}^{#1}\Delta_{#2}^{#3}} 
\nc{\nab}[3]{{}^{#1}\nabla_{#2}^{#3}} 

\nc{\Bi}[3]{{}^{#1}\!B_{#2}^{#3}} 

\nc{\wo}[1]{w_{#1}} 

\nc{\Pad}{\mathcal{P}_{\Delta}} 

\nc{\HH}{\mathbb{H}^{\bullet}} 
\nc{\pt}{pt} 

\nc{\He}[2]{{}^{#1}\mathcal{H}^{#2}} 
\nc{\h}[1]{\underline{H}_{#1}} 
\nc{\hcat}{\mathcal{C}(\He{}{})} 

\nc{\mkl}[3]{{}^{#1}\underline{H}_{#2}^{#3}} 
\nc{\mst}[3]{{\hspace{1mm}}^{#1}\!H_{#2}^{#3}} 

\nc{\dL}{\mathcal{L}} 
\nc{\dR}{\mathcal{R}}

\nc{\ver}{\vartheta}

\nc{\sF}{\mathcal{F}}
\nc{\sG}{\mathcal{G}}

\DeclareMathOperator{\Hom}{Hom}

\DeclareMathOperator{\Ext}{Ext}

\nc{\Poinc}{\pi}
\nc{\tPoinc}{\widetilde{\pi}}

\newcommand{\sur}{\twoheadrightarrow}

\newcommand{\excise}[1]{}

\def\HC{{\mathcal{H}}}

\DeclareMathOperator{\oldch}{ch_{\mathrm{old}}}
\DeclareMathOperator{\newch}{ch_{\mathrm{new}}}

\def\BC{{\mathcal{B}}}

\begin{document}

\setcounter{tocdepth}{3}

\begin{abstract}
We define and study categories of singular Soergel bimodules, which
are certain natural generalisations of Soergel bimodules. Indecomposable singular Soergel
bimodules are classified, and we conclude that the
split Grothendieck group of the 2-category of singular Soergel bimodules is
isomorphic to the Schur algebroid. Soergel's
conjecture on the characters of indecomposable Soergel bimodules
in characteristic zero is shown to imply a similar conjecture for the characters of
singular Soergel bimodules.
\end{abstract}

 \title{Singular Soergel bimodules}
\author{Geordie Williamson} \address{University of Sydney Mathematical Research Institute}
\email{g.williamson@sydney.edu.au}
\urladdr{https://www.maths.usyd.edu.au/u/geordie/}

\maketitle

\section{Introduction}

In this paper we define and study a 2-category of \emph{singular Soergel bimodules}. Singular Soergel bimodules are ubiquitous in Lie theory and geometric representation theory, and yet have an elementary definition. In this paper we give a complete algebraic treatment of their classification.

Before we come to a description of these bimodules, we give a brief description of the Schur algebroid, for which singular Soergel bimodules provide a categorification. Let $(W,S)$ be a Coxeter system and let $\mathcal{H}$ denote its Hecke algebra. For any subset $I \subset S$ of the simple reflections one has a subalgebra $\mathcal{H}_I$ (itself a Hecke algebra) and one obtains a natural module $\He{I}{}$ for $\mathcal{H}$ by inducing the ``trivial'' (right) module from $\mathcal{H}_I$ to $\mathcal{H}$.  The Schur algebroid is defined as the category with objects the modules $\He{I}{}$ for finitary subsets $I \subset S$ and morphisms given by morphisms of right $\mathcal{H}$-modules. (A subset $I \subset S$ is \emph{finitary} if the corresponding parabolic subgroup $W_I$ is finite.) For example $\He{\emptyset}{}$ is the right regular representation of $\mathcal{H}$ and its endomorphism algebra is $\mathcal{H}$ itself. (If $W$ is the symmetric group then the Schur algebroid is an idempotented version of the $q$-Schur algebra, which explains its name.)

We now explain the definition of singular Soergel bimodules, and their relation to the Schur algebroid. A finite dimensional representation $V$ of $W$ is \emph{reflection faithful} if it is faithful and the reflections in $W$ are exactly those elements which fix a codimension one subspace of $V$.

We fix a reflection faithful representation $V$ of $W$ over an infinite field of characteristic $\ne 2$, and let $R$ denote the graded ring of regular functions on $V$. Given a finitary subset $I \subset S$ denote by $R^I$ the invariants in $R$ under $W_I$. Furthermore, if $I, J \subset S$ are finitary denote by $\bMod{R^I}{R^J}$ the category of graded $(R^I, R^J)$-bimodules.

Consider the 2-category with:
\begin{enumerate}
\item objects consisting of finitary subsets $I \subset S$,
\item 1-morphisms from $I$ to $J$ given by bimodules in $\bMod{R^I}{R^J}$ (with composition of 1-morphisms given by tensor product of bimodules), and 
\item $2$-morphisms bimodule homomorphisms.
\end{enumerate}
The 2-category of singular Soergel bimodules is the full idempotent complete strict sub-2-category of the above 2-category generated by the bimodules $R^K \in \bMod{R^I}{R^J}$ whenever $I \supset K \subset J$ are finitary subsets. We write $\BB{I}{J}$ for the homomorphisms from $I$ to $J$ in this 2-category.

More concretely, given two finitary subset $I, J \subset S$ one may define $\BB{I}{J}$ to be the smallest full additive subcategory of $\bMod{R^I}{R^J}$ which contains all objects isomorphic to direct summands of shifts of objects of the form
\begin{equation*}
  R^{I_1} \otimes_{R^{J_1}} R^{I_2} \otimes_{R^{J_2}} \cdots \otimes_{R^{J_{n-1}}} R^{I_n}
\end{equation*}
where $I = I_1 \subset J_1 \supset I_2 \subset J_2 \supset \dots \subset J_{n-1} \supset I_n = J$ are finitary subsets of $S$.

Given $x \in W$ consider its twisted diagonal $\Gr{}{x}{} := \{ (x\lambda, \lambda) \; | \; \lambda \in V \}$. If $p \in W_I \setminus W / W_J$ define
\begin{equation*}
\Gr{I}{p}{J} \subset V/W_I \times V/W_J
\end{equation*}
as the image of $\Gr{}{x}{}$ under the quotient map $V \times V \to V/W_I \times V/W_J$.  We write $\Gr{I}{\le p}{J}$ (resp. $\Gr{I}{< p}{J}$) for the union of all $\Gr{I}{q}{J}$ with $q \le p$ (resp.  $q < p$) in the induced Bruhat order on double cosets. We may regard any $M \in \bMod{R^I}{R^J}$ as an $R^I\otimes R^J$-module and hence as a quasi-coherent sheaf on $V/W_I \times V/W_J$, which allows us to speak of support of $M$ or $m \in M$. We denote by $\Gamma_{\le p}M$ (resp. $\Gamma_{< p}M$) the submodule of sections supported on $\Gr{I}{\le p}{J}$ (resp. $\Gr{I}{<p}{J})$.

A classification of singular Soergel bimodules is provided by the following:

\begin{introtheorem} \label{thm:intro1}
There is a natural bijection:
\begin{align*}
\W{I}{J} \stackrel{\sim}{\longrightarrow}
\left \{ \begin{array}{c} 
\text{isomorphism classes of} \\
\text{indecomposable bimodules in $\BB{I}{J}$} \\
\text{(up to shifts in the grading).} \end{array} \right \}
\end{align*}
More precisely, for every $p \in \W{I}{J}$ there exists a unique isomorphism class (up to shifts) of indecomposable bimodules $M \in \BB{I}{J}$ whose support is $\Gr{I}{\le p}{J}$.
\end{introtheorem}

\excise{
\begin{introtheorem} \label{thm:intro1}
For every $p \in \W{I}{J}$ there exist a unique isomorphism class (up to shifts in the grading) of indecomposable bimodules $M \in \BB{I}{J}$ such that $\Gamma_{\le p} M = M$ and $\Gamma_p M \ne 0$. Moreover, any indecomposable bimodule in $\BB{I}{J}$ is of this form for some $p \in \W{I}{J}$. Therefore we have a natural bijection:
\begin{align*}
\W{I}{J} \stackrel{\sim}{\longrightarrow}
\left \{ \begin{array}{c} 
\text{isomorphism classes of} \\
\text{indecomposable bimodules in $\BB{I}{J}$} \\
\text{(up to shifts in the grading).} \end{array} \right \}
\end{align*}
\end{introtheorem}}

\excise{
\begin{introtheorem} \label{thm:intro1}
We have a bijection
\begin{align*}
\left \{ \begin{array}{c} 
\text{isomorphism classes of} \\
\text{indecomposable objects in $\BB{I}{J}$} \\
\text{(up to shifts in the grading)} \end{array} \right \}
&\stackrel{\sim}{\longrightarrow} \W{I}{J} 
\end{align*}
obtained by assigning to an indecomposable bimodule $M \in \BB{I}{J}$ the maximal $p \in \W{I}{J}$ such that $\Gamma_p M \ne 0$. If $M \in \BB{I}{J}$ corresponds to $p$ under the above bijection then $M = \Gamma_{\le p} M$.
\end{introtheorem}
}

We now explain how the 2-category of singular Soergel bimodules gives a categorification of the Schur algebroid. Let us write $\He{I}{J}$ for $\Hom(\He{J}{}, \He{I}{})$ in the Schur algebroid. Then, just like the Hecke algebra, each $\He{I}{J}$ posesses a standard basis $\{ \hspace{-2pt} \mst{I}{p}{J} \}$ parametrised by the double cosets $\W{I}{J}$ and composition gives us a morphism
\begin{align*}
\He{I}{J} \times \He{J}{K} & \to \He{I}{K} \\
(f, g) & \mapsto f *_J g
\end{align*}
which may be expressed as a renormalisation of the product in the Hecke algebra.

For any bimodule $M \in \BB{I}{J}$ and double coset $p \in \W{I}{J}$, the subquotient $\Gamma_{\le p}M / \Gamma_{< p} M$ is isomorphic to a finite direct sum of shifts of certain ``standard modules'' which may be described explicitly. It is therefore natural to define a character
\begin{align*}
\ch : \BB{I}{J} & \to \He{I}{J} \\
M & \mapsto \sum h_p \mst{I}{p}{J}
\end{align*}
where $h_p \in \N[v,v^{-1}]$ counts the graded multiplicity of the standard module in the subquotient $\Gamma_{\le p} M / \Gamma_{< p} M$.
 
Our main theorem is that the 2-category of singular Soergel bimodules categorifies the Schur algebroid:

\begin{introtheorem} \label{thm:intro2}
If $I, J, K \subset S$ are finitary we have a commutative diagram
\begin{equation*}
\xymatrix@C=2cm{ \BB{I}{J} \times \BB{J}{K} \ar[r]^(0.55){- \otimes_{R^J} - } \ar[d]^{\ch \times \ch} & \BB{I}{K} \ar[d]^{\ch} \\
\He{I}{J} \times \He{J}{K} \ar[r]^(0.55){-*_J-} & \He{I}{K}}
\end{equation*}
Hence the split Grothendieck group of the 2-category of singular Soergel bimodules is isomorphic to the Schur algebroid. Moreover, one may choose representatives $\{ \Bi{I}{p}{J} | p \in \W{I}{J} \}$ for each isomorphism class of indecomposable bimodules (up to shifts) such that $\{ \ch(\Bi{I}{p}{J}) \}$ gives a self-dual basis of $\He{I}{J}$ and 
\begin{equation*}
\ch(\Bi{I}{p}{J}) = \mst{I}{p}{J} + \sum_{q \le p} g_{q,p} \mst{I}{q}{J}
\quad \text{ for some $g_{q,p} \in \N[v,v^{-1}].$} \let\qedsymbol\openbox\qedhere
\end{equation*}
\end{introtheorem}

If $I = J = \emptyset$ then $\He{I}{J}$ is the Hecke algebra and we write $\BB{}{}$ instead of $\BB{I}{J}$. 
In this case Theorem \ref{thm:intro1} tells us that the isomorphism classes of indecomposable objects in $\BB{}{}$ are parametrised, up the shifts, by $W$ and Theorem \ref{thm:intro2} tells us that their characters yield a self-dual basis for the Hecke algebra having certain positivity properties.

The special case of the above result when $I = J = \emptyset$ of the above was obtained by Soergel in \cite{SoBimodules} (using a slightly different definition of $\BB{}{}$) and formed the principal motivation for this work. Similar ideas have also been pursued by Dyer in \cite{Dy1} and \cite{Dy2}, and by Fiebig in \cite{Fie1}, \cite{Fie2} and \cite{Fie3}.


Of course, the most famous basis for the Hecke algebra with properties similar to the above is the Kazhdan-Lusztig basis $\{ \mkl{}{w}{} \}$ of $\He{}{}$. Let us write $B_x$ for a representative of the isomorphism class of indecomposable objects parametrised by $x \in W$, normalised as in Theorem \ref{thm:intro2}. Soergel has proposed the following:

\begin{introconjecture}[\cite{SoBimodules}, Vermutung 1.13] \label{conj:main}
Suppose that $k$ is of characteristic 0. Then, for all $x \in W$ we have
\begin{equation*}
\ch(B_x) = \h{x}. \let\qedsymbol\openbox\qedhere
\end{equation*}
\end{introconjecture}
This conjecture is known in all cases where one may interpret
Kazhdan-Lusztig polynomials geometrically, for example for finite and
affine Weyl groups. Its importance is that it provides a conjectural
generalisation of this theory to arbitrary Coxeter groups. For
example, a positive solution to this conjecture would resolve the long-standing conjecture as to the positivity of Kazhdan-Lusztig polynomials, as the above character is manifestly positive.

For arbitrary finitary subsets $I, J \subset S$ there exists a Kazhdan-Lusztig basis $\{ \mkl{I}{p}{J} \}$ for $\He{I}{J}$. The following relates the objects in the categories $\BB{I}{J}$ and $\BB{}{}$ and shows that Soergel's conjecture implies character formulae for all indecomposable singular bimodules.

\begin{introtheorem} \label{thm:intro4} Let $I, J \subset S$ be finitary, $p \in \W{I}{J}$ and denote by $p_+$ the unique element of $p$ of maximal length. Then we have an isomorphism:
\begin{equation*}
R \otimes_{R^I} \Bi{I}{p}{J} \otimes_{R^J} R \cong B_{p_+} \quad \text{ in $\bMod{R}{R}$.}
\end{equation*}
In particular, if Soergel's conjecture is true then 
\begin{equation*}
\ch(\Bi{I}{p}{J}) = \mkl{I}{p}{J}. \let\qedsymbol\openbox\qedhere
\end{equation*}
 \end{introtheorem}

\subsection{Applications of singular Soergel bimodules}
Before going into more detail about the contents of this paper, we briefly discuss some other applications of singular Soergel bimodules.
\begin{enumerate}
\item 
Soergel bimodules arose out of Soergel's attempts to understand
category $\mathcal{O}$, Harish-Chandra bimodules and the
Kazhdan-Lusztig conjecture. He showed that both a regular block of
category $\mathcal{O}$ and certain equivariant perverse sheaves on the
flag variety can be described in terms of (finite dimensional versions
of) Soergel bimodules \cite{SoJAMS}. Hence it is natural to expect
that singular Soergel bimodules govern ``singular'' situations (indeed
this is the origin of their name). For Harish-Chandra bimodules such
an equivalence was established by Stroppel \cite{Str}. For an
explanation (without proof) of the relation between singular Soergel
bimodules and equivariant sheaves on the flag variety see the
introduction to \cite{SSB} (see \cite{SoLanglands} for a treatment,
with proofs, of the non-singular case). Once one has established a
connection to representation theory or geometry the classification
theorem (Theorem \ref{thm:intro1}) usually follows in a
straightforward way, and it this fact that led Soergel to suspect
that Theorem \ref{thm:intro1} might be true for a general Coxeter
system. (It was also hoped that Soergel bimodules might provide a
means of avoiding the use of the decomposition theorem. This hope has
not yet be realised.)

\item In \cite{KR} Khovanov and Rozansky constructed a
  categorification of the HOMFLYPT polyonomial and in \cite{Kh}
  Khovanov gave another construction of this invariant by taking the
  Hochschild homology of a complex of Soergel bimodules constructed by
  Rouquier \cite{Ro}. Mackaay, Stosic and Vaz conjectured that one
  could extend this construction to produce a categorification of the
  colored HOMFLYPT polynomial by instead considering a certain complex
  of singular Soergel bimodules. This was proven by Webster and the author using geometric
  methods in \cite{WW}. In part motivated to give an algebraic proof
  of this construction, Mackaay, Stosic and Vaz recently constructed a
  diagrammatic categorification of the Schur algebra $S(n,d)$
  \cite{MSV}. When $n = d$ it is natural to expect that their
  categorification agrees with the categorification in this paper
  using singular Soergel bimodules (with $W = S_n$) but this has yet
  to be understood. (An analogous construction  in the context of
  category $\mathcal{O}$ is given by Mazorchuk
  and Stroppel in \cite{MS}).

\item So far, all applications of Soergel bimodules in representation
  theory have been by using Soergel bimodules as an intermediary
  between more complicated categories. This is usually achieved with
  the help of a fully faithful functor (the achetypal example being
  Soergel's functor $\mathbb{V}$). Sometimes it is difficult to
  construct such a functor but one still expects Soergel bimodules, or
  some variant, to control a given representation theoretic
  category. It has been suggested by Rouquier that if one had
  presentations of the category of Soergel bimodules by generators and
  relations, then giving an action of the category would be much more
  straightforward (in the same way that it is difficult to explicitly specify a
  homomorphism from a group, unless one has a presentation).
  Progress in this direction has been recently made by
  Libedinsky \cite{Li3} (for right-angled Coxeter groups) and
  Elias-Khovanov \cite{EK} (for the symmetric group). It is hoped that
  a similar ``generators and relations'' description might be possible
  for singular Soergel bimodules.

\item
Let $W$ be a Weyl group with root system $\Phi$ and simple reflections $S$ and let
$W \subset \widetilde{W}$ be the corresponding affine Weyl
group. After choosing a
reflection faithful representation $V$ of $W$ we may consider
$\BB{S}{S}$, which is a full tensor subcategory of
$R^S$-bimodules. The above results show that $\BB{S}{S}$ categorifies
$\He{S}{S}$. The algebra $\He{S}{S}$ is
known in the literature as the ``spherical Hecke algebra''. It is a fact known
  as the Satake isomorphism (see \cite{Lu1}) that the spherical Hecke
  algebra is commutative and isomorphic to the representation ring of
  the adjoint semi-simple group $G^{\vee}_a$ with root system
  $\Phi^{\vee}$ dual to $\Phi$.
Using this fact, one may show that if one normalises the
representatives $\{ \Bi{I}{p}{I} \; | \; p \in \W{I}{I} \}$ as in
Theorem \ref{thm:intro2} then any tensor product $\Bi{I}{p}{I}
\otimes_{R^I} \Bi{I}{q}{I}$ is isomorphic to a direct sum of
$\Bi{I}{r}{I}$ for $r \in \W{I}{I}$ without shifts. If we only allow
degree zero morphisms, we obtain a tensor subcategory $\BB{I}{I}_0$
containing all $\Bi{I}{p}{I}$ for $p \in \W{I}{I}$. In view of work of
Mirkovic and Vilonen \cite{MV} it is natural to expect an equivalence
of tensor categories
\begin{equation*}
\BB{I}{I}_0 {\cong} \Rep G^{\vee}_a
\end{equation*}
Such an equivalence has been constructed by Florian Klein for $G =
PGL_2$ \cite{Kl}. He
also conjectures a general procedure as to how one might enlarge
$\BB{I}{I}_0$ to recover the representation ring of the simply
connected cover of $G^{\vee}_a$, and proves it for $G = SL_2$.
\end{enumerate}

\subsection{An overview of the classification}

The proof of our classification and categorification theorems (Theorems \ref{thm:intro1} and \ref{thm:intro2}) essentially follows techniques developed by Soergel in \cite{SoBimodules}. Because the argument is quite subtle, we give here a brief summary of the key points.

As has already been alluded to in the introduction, any Soergel
bimodule $B \in \BB{I}{J}$ is an $(R^I, R^J)$-bimodule and hence can
be regarded as a quasi-coherent sheaf on $V/ W_I \times V/W_J$. The
first key observation is that the quasi-coherent sheaves on $V/ W_I
\times V/ W_J$ that one obtains from singular Soergel bimodules have a special form.

Given a double coset $p \in \W{I}{J}$ we have defined a subvariety $
\Gr{I}{p}{J} \subset V/W_I \times V/W_J$. Choose an enumeration $p_1, p_2 \dots $ of the elements of $\W{I}{J}$ compatible with the Bruhat order. Then given any $M \in \bMod{R^I}{R^J}$ one obtains filtrations
\begin{align*}
\cdots \subset \Gamma_{C(i-1)} M \subset \Gamma_{C(i)} M \subset
\Gamma_{C(i+1)} M \subset \cdots \\
\cdots \supset \Gamma_{\check{C}(i-1)} M \supset \Gamma_{\check{C}(i)}
M \supset \Gamma_{\check{C}(i+1)} M \supset \cdots
\end{align*}
where
\begin{align*}
\Gamma_{C(j)} M = \text{sections supported on the union of } \Gr{I}{p_j}{J}, \Gr{I}{p_{j-1}}{J}, \dots \\
\Gamma_{\check{C}(j)} M = \text{sections supported on the union of } \Gr{I}{p_j}{J}, \Gr{I}{p_{j+1}}{J}, \dots
\end{align*}
The crucial fact is that, if $M \in \BB{I}{J}$ is a singular Soergel
bimodule, then both filtrations are finite and exhaustive and the subquotients
are isomorphic to direct sums of shifted \emph{standard modules},
which are certain $(R^I, R^J)$-bimodules which may be described
explicitly. (In particular any singular Soergel bimodule
is supported on finitely many subvarieties of the form $\Gr{I}{p}{J}$.)

In order to prove this fact we define \emph{objects with nabla flags}
and \emph{objects with delta flags} as those objects for which the
subquotients in the first or second filtration respectively are
isomorphic to direct sums of shifts of standard modules. We then show
that these subcategories are preserved by the functors of restriction
and extension of scalars, which we renormalise and rename
\emph{translation functors}. (The choice of language is intended to
emphasise the analogy with
category $\mathcal{O}$, where it is very important (and
well-known) that translation functors preserve modules with delta and
nabla flags.) Given an object with a nabla or delta
flag it is natural to define its \emph{nabla} or \emph{delta
character} in the Hecke category by counting the graded multiplicities
of standard modules in the subquotients of the above filtrations. It
turns out that one may describe the effect of translation functors on
the character in terms of multiplication with a standard generator in
the Hecke category (this is the first step towards Theorem
\ref{thm:intro2}).

By the inductive definition of the objects in $\BB{I}{J}$ it
follows that they have both nabla and delta flags. This may be
exploited to describe $\Hom(M,B)$ and $\Hom(B,M)$ when $B$ is a Soergel bimodule and $M$ has a delta or nabla flag respectively. The classification of the indecomposable objects
in $\BB{I}{J}$ is then straightforward (essentially by an idempotent lifting argument).

\subsection{Structure of this paper}

In Section \ref{sec:coxgroups} we recall basic facts about Coxeter groups
and their Hecke algebras, and introduce the Schur algebroid. In
Section \ref{sec:not} we cover some bimodule basics.
In Section \ref{sec:stand} we begin the study of
so-called singular standard modules. After giving their definition, we
turn to an analysis of the effect of restriction and extension of
scalars. To prove the existence of certain filtrations on induced
standard modules we need
equivariant Schubert calculus which is developed in
Section \ref{subsec:demRX}. In Section \ref{sec:flags} we turn to
the study of modules filtered by singular standard modules, and show
how their characters may be understood in the Schur
algebroid. Finally, in Section \ref{sec:classification} we turn to
singular Soergel bimodules, and prove the two main theorems.

\numberwithin{equation}{subsection}
 

\section{Coxeter groups and the Schur algebroid}  \label{sec:coxgroups}

\subsection{Coxeter groups} \label{subsec:coxgroups}

In this section we recall standard facts about Coxeter groups,
standard parabolic subgroups, Poincar\'e polynomials and double
cosets that will be needed in the sequel. Standard references are \cite{Hu} and
\cite{Bo}.

Throughout we fix a Coxeter system $(W,S)$ with reflections $T$, length
function $\ell : W \to \mathbb{N}$ and Bruhat order $\le$. We always
assume the set $S$ is finite.
Given a subset $I \subset S$ we denote by $W_I$ the standard parabolic
subgroup generated by $I$. We call a subset $I \subset S$
\emph{finitary} if $W_I$ is finite. Given $I \subset W_I$ finitary we
denote by $\wo{I}$ the longest element of $W_I$. We define
\[
\tPoinc(I) = \sum_{w \in W_I} v^{-2\ell(w)} \quad \text{and} \quad \Poinc(I) = v^{\ell(\wo{I})}\tPoinc(I).
\]
We call $\Poinc(I)$ the Poincar\'e polynomial of $W_I$.

Let $f \mapsto \overline{f}$ be the involution of $\Z[v,v^{-1}]$ which fixes $\Z$ and sends $v$ to $v^{-1}$. We will call elements $f \in \Z[v,v^{-1}]$ satisfying $f = \overline{f}$ \emph{self-dual}. Because $\ell(\wo{I}x) = \ell(\wo{I}) - \ell(x)$ for all $x \in W_I$ it follows that $\Poinc(I)$ is self-dual.

 Given $I \subset S$ we define
\begin{equation*}
D_I = \{ w \in W \; | \; ws > w \text{ for all }s \in I\} \; \text{ and }
\; _ID = (D_I)^{-1}.
\end{equation*}
If $I \subset S$ is finitary we define
\begin{equation*}
D^I = \{ w \in W \; | \; ws < w \text{ for all } s \in I \} \; \text{ and }
\; ^ID = (D^I)^{-1}.
\end{equation*}
The elements of $D_I$ and $D^I$ (resp. $_ID$ and $^ID$) are called the \emph{minimal} and \emph{maximal left} (resp. \emph{right}) \emph{coset representatives}.

 Given two subsets $I, J \subset S$ we define
\begin{equation*}
\Dmin{I}{J}= \Dmin{I}{} \cap \Dmin{}{J}.
\end{equation*}
If $I$ and $J$ are finitary we define
\begin{equation*}
\Dmax{I}{J}= \Dmax{I}{} \cap \Dmax{}{J}.
\end{equation*}

We have (see \cite{Ca}, Proposition 2.7.3):

\begin{proposition} Let $I, J \subset S$. Every double coset $p = W_IxW_J$ contains a unique element of $\Dmin{I}{J}$ and this is the element of smallest length in $p$. If $I$ and $J$ are finitary then $p$ also contains a unique element of $\Dmax{I}{J}$, and this is the unique element of maximal length. \end{proposition}

Let $I, J \subset S$. Given $p \in \W{I}{J}$ we denote by $p_-$ the
unique element of minimal length in $p$. If $I$ and $J$ are finitary,
we denote by $p_+$ the unique element of maximal length in $p$. We
call $p_-$ and $p_+$ the \emph{minimal} and \emph{maximal double coset
  representatives}. Define
\begin{equation*}
\tPoinc(p) = \sum_{x \in p}
v^{-2\ell(x)} \quad
\text{and} \quad
\Poinc(p) = v^{\ell(p_+)- \ell(p_-)} v^{2\ell(p_-)} \tPoinc(p).
\end{equation*}
We call $\pi(p)$ the Poincar\'e polynomial of $p$. We will see below that $\Poinc(p)$ is self-dual.

The following theorem describes intersections of (not necessarily
standard) parabolic subgroups (see \cite{Ca}, Theorem 2.7.4):

\begin{theorem}[Kilmoyer] \label{Kilmoyer}
Let $I, J \subset S$ and $p \in \W{I}{J}$. Then
\begin{displaymath}
W_I \cap p_-W_Jp_-^{-1} = W_{I \cap p_-Jp_-^{-1}}.
\end{displaymath} \end{theorem}

The following gives us canonical representatives for elements of double cosets (see \cite{Ca}, Theorem 2.7.5):

\begin{theorem}[Howlett] \label{Howlett}
Let $I, J \subset S$ and $p \in \W{I}{J}$. Setting $K = I \cap p_-Jp_-^{-1}$ the map
  \begin{eqnarray*}
    (D_K \cap W_I) \times W_J & \to & p \\
(u, v) & \mapsto & up_-v
  \end{eqnarray*}
is a bijection satisfying $\ell(up_-v) = \ell(p_-) + \ell(u) + \ell(v)$.
\end{theorem}

The intersection $I \cap p_-Jp_-^{-1}$ emerges often enough to warrent
special notation. Let $I, J \subset S$ be finitary, choose $p \in
\W{I}{J}$ and set $K = I \cap p_-Jp_-^{-1}$. We define:
\begin{align*}
\tPoinc(I,p,J) & := \tPoinc(K) \\ 
\Poinc(I,p,J) & := \Poinc(K) \\
w_{I,p,J} & := w_K
\end{align*}
The
above theorems imply the identities:
  \begin{align}
  \ell(p_+) - \ell(p_-) &= \ell(\wo{I}) + \ell(\wo{J}) - \ell(\wo{I, p, J}) 
\label{eq:poinc1}\\
\tPoinc(p)\tPoinc(I,p,J) &= \tPoinc(I)\tPoinc(J)\label{eq:poinc2} \\
\Poinc(p)\Poinc(I,p,J) &= \Poinc(I)\Poinc(J)\label{eq:poinc3} \\
\overline{\Poinc(p)} & = \Poinc(p).\label{eq:poinc4} 
  \end{align}

We will need the following (which is a straightforward
consequence of Howlett's theorem):

\begin{proposition} \label{prop:doublebruhat}
Let $I, J \subset S$ and $p \in \W{I}{J}$. If $x$ and $tx$ both lie in $p$ then either $t  \in W_I$ or $tx = xt^{\prime}$ for some reflection $t^{\prime} \in W_J$. 
\end{proposition}

Recall that $W$ becomes a poset when equipped with the Bruhat
order. Given finitary $I, J \subset S$ the Bruhat order on $\W{I}{J}$
(which we also denote by $\le$) is the weakest partial order such that
the quotient map
\[ \label{eq:quo}
\quo : W \to \W{I}{J}
\]
is a morphism of posets. It may be characterised by $p \le q$ if and
only if $p_- \le q_-$. We say that a  subset $C \subset \W{I}{J}$ is \emph{downwardly} (resp. \emph{upwardly})  \emph{closed} if $p \in C$ and $q \le p$ (resp. $q \ge p$) implies $q \in C$.

Given a poset $(X, \le)$ and $x \in X$ we will often abuse notation and write $\{ \le\! x\}$ (resp. $\{ <\! x \}$) for the set of elements in $X$ less (resp. strictly less) than $x$, and similarly for $\{ \ge\! x\}$ and $\{ >\! x \}$.

Let $\quo$ be as above and choose $q \in \W{K}{L}$. The set
$\quo^{-1}(q)$ always has a maximal element $p$. We have
\begin{equation*}
\quo^{-1}(\{ \le\! q \}) = \{ \le\! p \}  \text{ and } \quo^{-1}(\{ \ge\! q \}) = \{ \ge\! p \}. \end{equation*}

The following  fact will be needed in in the sequel.

\begin{lemma} \label{lem:poinc3}
Let $I \subset K$ and $J \subset L$ be finitary subsets of $S$. If $p \in \W{I}{J}$ and $q \in \W{K}{L}$ are such that $p \subset q$ then
\begin{equation*}
\frac{\Poinc(K,q,L)}{\Poinc(I,p,J)} \in \N[v,v^{-1}].
\end{equation*}
\end{lemma}

\begin{proof} We may assume that either $I =K$ and $J  = L$. If $I =
  K$ then, by imitating the arguments used in the proof of \cite{Ca},
  Lemma 2.7.1 one may show that $I \cap p_-Jp_-^{-1} \subset K \cap
  q_-Lq_-^{-1}$ and the lemma follows in this case. The case $J = L$
  follows by inversion and the fact that two conjugate subsets of $S$
  have the same Poincar\'e polynomials. \end{proof}

We will need the following proposition when we come to discuss Demazure operators.

\begin{proposition} \label{prop:doublecosetdiff}
Let $p$ be a double coset and $x \in p$. We have
\begin{equation*}
\ell(p_+) - \ell(x) = |\{ t \in T \; | \; x < tx \in p \}|.
\end{equation*}
\end{proposition}

\begin{proof} Let $u \in W_I$ and $v \in W_J$ and set $y = uxv \in p$. We claim that for all $t \in T$,
\begin{equation}
x > tx \notin p \Leftrightarrow y > (utu^{-1})y \notin p. \label{eq:lengthclaim}
\end{equation}
In order to verify this claim it is enough to show that, if $x \in p$
\begin{align*}
x > tx \notin p, s \in W_J & \Rightarrow xs > txs \\ 
x > tx \notin p, s \in W_I & \Rightarrow sx > (sts)sx.
\end{align*}
For the first statement note that either $xs > txs$ or $xs < txs$. However, as $x > tx$ the second possibility would imply $x = txs$ by Deodhar's ``Property Z'' (alternatively this follows from \cite{Hu}, Proposition 5.9) which contradicts $tx \notin p$. The second statement follows similarly. Thus we have verified (\ref{eq:lengthclaim}). It is also immediate that, for all $t \in T$,
\begin{equation*}
tx \in p \Leftrightarrow utu^{-1}y \in p.
\end{equation*}
Now, setting $y = p_+$ and using the above facts together with the maximality of $p_+ \in p$ we follow
\begin{align*}
\ell(p_+) - \ell(y_+) &= | \{ t \in T \; | \; p_+ > tp_+ \} | - | \{ t \in T \; | \; x > tx \} |  \\
& = | \{ t \in T \; | \; p_+ > tp_+ \in p \} | - | \{ t \in T \; | \; x > tx \in p \} |  \\
& = |\{ t \in T \; | \; x < tx \in p \}|. \quad \qedhere
\end{align*}
\end{proof}

\subsection{The Hecke algebra} \label{subsec:hecalg}

As always, $(W,S)$ denotes a Coxeter system. 
The \emph{Hecke algebra} $\He{}{}$ is the free $\Z[v,v^{-1}]$-module
with basis $\{ H_w \;|\; w \in W \}$ and multiplication 
\begin{equation} \label{eqn:heckemult}
H_sH_w = \left \{ \begin{array}{ll} H_{sw} &\text{ if $sw > w$ } \\
(v^{-1} - v)H_w + H_{sw} & \text{ if $sw < w$.} \end{array} \right .
\end{equation}
We call $\{ H_w \}$ the \emph{standard basis}. Each $H_w$ is
invertible and there is an involution on $\He{}{}$ which sends $H_w$
to $H_{w^{-1}}^{-1}$ and $v$ to $v^{-1}$. We will call elements fixed
by this involution \emph{self-dual}.

Let $\{ \h{w} \;|\; w \in W \}$ denote the Kazhdan-Lusztig basis of
$\He{}{}$ (see \cite{KL1}, or \cite{SoKLpolys} for an explanation
using our notation). If $I \subset S$ is finitary we have
\begin{equation}
  \label{eq:longestpar}
  \h{\wo{I}}{} = \sum_{x \in W_I} v^{\ell(\wo{I}) - \ell(x)}H_x.
\end{equation}
If $x \in W_I$ then
\begin{equation}\label{eqn:paramult}
  H_x \h{\wo{I}} = v^{-\ell(x)} \h{\wo{I}}{}.
\end{equation}
It follows that, if $K \subset I$ then,
\begin{equation}
  \label{eq:parasquare}
  \h{\wo{K}} \h{\wo{I}}{} = \Poinc(K) \h{\wo{I}}.
\end{equation}

There is a $\Z[v,v^{-1}]$-linear anti-involution $i : \He{}{} \to \He{}{}$ sending $H_x$ to $H_{x^{-1}}$. Following \cite{Lu2} we define a bilinear form:
\begin{align*}
\He{}{} \times \He{}{} & \to \Z[v,v^{-1}]\\
(f, g) & \mapsto \langle f, g \rangle = \text{ coefficient of $H_{\id}$ in $fi(g)$}.
\end{align*}
Alternatively one has:
\begin{equation} \label{eq:biform1}
\langle H_x, H_y \rangle = \delta_{x, y} \quad \text{for all $x, y \in W$.}
\end{equation}

\subsection{The Schur algebroid} \label{subsec:HeckeCat}

We want to define a certain relative version of the Hecke algebra
associated to all pairs of finitary subsets $I, J \subset S$. The most
natural way to define this is as an $\Z[v,v^{-1}]$-linear
category. Alternatively one may regard the Schur algebroid as a ring
with many objects.

For all pairs of finitary subsets $I, J \subset S$ define:  \label{lab:IHJ}
\begin{eqnarray*}
\He{I}{} &=& \h{\wo{I}}\He{}{} \\
\He{}{J} &=& \He{}{}\h{\wo{J}} \\
\He{I}{J} &=& \He{I}{} \cap \He{}{J} 
\end{eqnarray*}
Given a third finitary subset $K \subset S$ we may define a multiplication as follows
\begin{eqnarray*}
  \He{I}{J} \times \He{J}{K} &\to& \He{I}{K} \\
(h_1, h_2) &\mapsto& h_1 *_J h_2 = \frac{1}{\Poinc(J)} h_1h_2.
\end{eqnarray*}
This well defined by (\ref{eq:parasquare}). If $J = \emptyset$ we
write $*$ instead of $*_{\emptyset}$.

\begin{definition}
The \emph{Schur algebroid} is the $\Z[v,v^{-1}]$-linear category
defined as follows. The objects are finitary subsets $I \subset
S$. The morphisms between two objects $I$ and $J$ consists of the
module $\He{I}{J}$. Composition $\He{I}{J} \times \He{J}{K} \to
\He{I}{K}$ is given by $*_J$. This defines a $\Z[v,v^{-1}]$-linear
category with the identity endomorphism of $I \subset S$ given by
$\h{w_I}$.\end{definition}

\begin{remark} In the introduction the Schur algebroid was defined slightly differently. Let us regard $\He{I}{}$ as a right $\He{}{}$-module. Then it is easy to see that
\[
\Hom_{\He{}{}}(\He{J}{}, \He{I}{}) = \He{I}{J}.
\]
Indeed $\Hom_{\He{}{}}(\He{J}{}, \He{}{}) = \He{}{J}$ and $\Hom_{\He{}{}}(\He{}{}, \He{I}{}) = \He{I}{}$, and the above space is given by intersecting these two homomorphism spaces. It is easy to see that, if $\alpha \in \Hom_{\He{}{}}(\He{K}{}, \He{J}{})$ and $\beta \in \Hom_{\He{}{}}(\He{J}{}, \He{I}{})$ correspond to $f \in \He{J}{K}$ and $g \in \He{I}{J}$ then $\beta \circ \alpha \in \Hom_{\He{}{}}(\He{K}{}, \He{I}{})$ corresponds to $g *_J f$. Hence the above definition and that given in the introduction agree.
\end{remark}

We have that $h = \sum a_y H_y$ is in $\He{I}{J}$ if and only if, for
all $y \in W$, $a_{sy} = va_y$ and $a_{yt} = va_y$ for all $s \in I$
and $t \in J$ such that $sy < y$ and $yt < y$. This shows how to find
a basis for $\He{I}{J}$ as a $\Z[v,v^{-1}]$-module. Namely, for all $p
\in \W{I}{J}$ define
\begin{equation*}
  \mst{I}{p}{J} = \sum_{x \in p} v^{\ell(p_+) - \ell(x)} H_x.
\end{equation*}
It follows that, if $h = \sum a_y H_y$ is in $\He{I}{J}$ then
\begin{equation} \label{eq:inmst}
  h = \sum_{p \in \W{I}{J}} a_{p_+} \mst{I}{p}{J}.
\end{equation}
The set $\{ \mst{I}{p}{J} \; |\; p \in \W{I}{J}\}$ is clearly linearly independent over $\Z[v,v^{-1}]$ and we conclude that they form a basis, which we call the \emph{standard basis} \label{lab:standbasis} of $\He{I}{J}$.


\excise{
Given an element $h \in \He{I}{J}$ we may write $h = \sum_{p \in \W{I}{J}} h_p \mst{I}{p}{J}$. We define the \emph{support} of $h$ to be the finite set
\begin{equation*}
  \supp h = \{ p \in \W{I}{J} \; | \; h_p \ne 0 \}.
\end{equation*}}

A Kazhdan-Lusztig basis element $\h{y} \in \He{I}{J}$ if and only if
$y$ is maximal in its $(W_I, W_J)$-double coset. In general, if $p \in
\W{I}{J}$ we define
\begin{equation*} \label{lab:standkl}
  \mkl{I}{p}{J} = \h{p_+}.
\end{equation*}
We have
\begin{equation*}
\mkl{I}{p}{J} = \mst{I}{p}{J} + \sum_{q < p } h_{q_+, p_+} \mst{I}{q}{J}.
\end{equation*}
It follows that $\{ \mkl{I}{p}{J} \; |\; p \in \W{I}{J}\}$ also forms a $\Z[v,v^{-1}]$ basis for $\He{I}{J}$. We will refer to this as the \emph{Kazhdan-Lusztig basis}.

For all finitary subsets $I, J \subset S$ satisfying $I \subset J$ or $J \subset I$ we define
\begin{equation*}
\mst{I}{}{J} = \mst{I}{p}{J} \text{ where $p = W_IW_J$}.
\end{equation*} 
We call call elements of the form $\mst{I}{}{J} \in \He{I}{J}$
\emph{standard generators}. \label{lab:standgenerator} The standard
generators are the analogues of the elements $\h{s} \in \He{}{}$ and
we will see below that the set of standard generators generate the
Schur algebroid, which justifies the terminology.

The following proposition describes the action of the standard
generators on the standard basis.

\begin{proposition} \label{prop:multform} Let $I, J, K \subset S$ be finitary and assume $J \subset K$ or $J \supset K$. The action of $\mst{J}{}{K}$ on the basis $\{ \mst{I}{p}{J} \;| \; p \in \W{I}{J} \}$ is as follows:
  \begin{enumerate}
  \item If $J \supset K$ then
    \begin{equation*}
      \mst{I}{p}{J} *_J \mst{J}{}{K} = \sum_{q \in W_I \!\setminus p  / W_K} v^{\ell(p_+) - \ell(q_+)} \mst{I}{q}{J}.
    \end{equation*}
\item If $J \subset K$ then
  \begin{equation*}
    \mst{I}{p}{J} *_J \mst{J}{}{K} = v^{\ell(q_-) - \ell(p_-)} \frac{ \Poinc(I,q,K)}{ \Poinc(I,p,J)} \mst{I}{q}{K}
  \end{equation*}
where $q = W_IpW_K$ is the $(W_I, W_K)$-coset containing $p$.
\end{enumerate}
\end{proposition}

Before we prove the proposition we need a lemma.

\begin{lemma} \label{lem:wmst}
Let $I, J \subset S$ be finitary, $x \in W$ and $p = W_IxW_J$. Then
  \begin{displaymath}
    \mst{I}{}{\emptyset} * H_x * \mst{\emptyset}{}{J} = v^{\ell(p_-) - \ell(x)} \Poinc(I,p,J) \mst{I}{p}{J}.
  \end{displaymath}
\end{lemma}

\begin{proof}
  By Howlett's Theorem (\ref{Howlett}) we may write $x = up_-v$ with $u \in W_{I}$, $v \in W_{J}$ and $\ell(x) = \ell(u) + \ell(p_-) + \ell(v)$. By (\ref{eqn:paramult}) we have:
  \begin{displaymath}
    \mst{I}{}{\emptyset} * H_{x} * \mst{\emptyset}{}{J} = 
v^{\ell(p_-) - \ell(x)} \mst{I}{}{\emptyset} * H_{p_-} * \mst{\emptyset}{}{J}.
  \end{displaymath}
Thus we will be finished if we can show that 
\begin{displaymath}
  \mst{I}{}{\emptyset} * H_{p_-} * \mst{\emptyset}{}{J} = \Poinc(I,p,J) \mst{I}{p}{J}.
\end{displaymath}
We write $K = I \cap p_-Jp_-^{-1}$ so that $\Poinc(I,p,J) = \Poinc(K)$. If $s \in K$ then $sp_- = p_-s^{\prime}$ for some $p^{\prime} \in J$ and therefore
\begin{equation}
  \label{eq:stabcom}
  \h{\wo{K}}{} H_{p_-} = H_{p_-} \h{w_{K^{\prime}}}
\end{equation}
where $K^{\prime} = p_-^{-1}Kp_-$. Because $K$ and $K^{\prime}$ are conjugate $\Poinc(K) = \Poinc(K^{\prime})$.
We define $N \in \He{}{}$ by
\begin{equation*}
  N = v^{\ell(\wo{I}) - \ell(\wo{K})} \sum_{u \in D_K \cap W_I} v^{-\ell(u)} H_u
\end{equation*}
and calculate
\begin{align*}
 \mst{I}{}{\emptyset} * H_{p_-} * \mst{\emptyset}{}{J}
 & = N\h{\wo{K}}H_{p_-}\h{\wo{J}} & (\text{Howlett's theorem}) \\
 & = NH_{p_-}\h{\wo{K^{\prime}}} \h{\wo{J}} & (\ref{eq:stabcom})\\
 & = \Poinc(K) NH_{p_-}\h{\wo{J}} & (\ref{eq:parasquare}) \\
 & = \Poinc(K)v^a
 \sum_{x \in p} v^{- \ell(x)} H_x & (\text{Howlett's theorem}) \\
 & = \Poinc(K) \mst{I}{p}{J}
\end{align*}
where the last line follows because, by \eqref{eq:poinc1}, 
\begin{equation*}
a = \ell(\wo{I}) - \ell(\wo{K}) + \ell(\wo{J}) + \ell(p_-) =
\ell(p_+). \qedhere
\end{equation*}
\end{proof}

\excise{

Howlett's proposition (\ref{Howlett}) we have
\begin{equation*}
  v^{-\ell(x) - \ell(w_J)}NH_x\h{\wo{J}}{} = \sum_{w \in W_IxW_J} v^{-\ell(w)}H_w = v^{-\ell(z)}\mst{I}{x}{J}
\end{equation*}
and
\begin{equation*}
  v^{\ell(\wo{L}) - \ell(\wo{K})}N\h{\wo{K}}{} = \mst{I}{}{\emptyset}.
\end{equation*}
Putting this together yields:
\begin{eqnarray*}
  \mst{I}{}{\emptyset} * H_x * \mst{\emptyset}{}{J} & = & v^{\ell(\wo{I}) - \ell(\wo{K})} N \h{\wo{K}}{}H_x\h{\wo{J}}{} \\
& = & v^{\ell(\wo{I}) - \ell(\wo{K})} N H_x \h{\wo{x^{-1}K}}{} \h{\wo{J}}{} \\
& = & v^{\ell(\wo{I}) - \ell(\wo{K})} \Poinc(W_K) N H_x \h{\wo{J}}{} \\
& = & v^{A} \Poinc(K) \mst{I}{x}{J}
\end{eqnarray*}
with 
\begin{equation*}
  A = \ell(\wo{I}) - \ell(\wo{K}) - \ell(z) + \ell(x) + \ell(\wo{J}).
\end{equation*}
However $A = 0$ by \ref{cor:doub}.
\end{proof}}

\begin{proof}[Proof of Proposition \ref{prop:multform}]
Statement (1) follows by (\ref{eq:parasquare}) and (\ref{eq:inmst}). We now turn to (2). Let us expand
\begin{equation*}
  P = \mst{I}{}{\emptyset} * H_{p_-} * \mst{\emptyset}{}{J} *_J  \mst{J}{}{K}
\end{equation*}
in two different ways. As $\mst{\emptyset}{}{J} * \mst{J}{}{K} = \mst{\emptyset}{}{K}$ by (\ref{eq:parasquare}) we obtain, using Lemma \ref{lem:wmst}:
\begin{displaymath}
  P = \mst{I}{}{\emptyset} * H_{p_-} * \mst{\emptyset}{}{K} =
v^{\ell(q_-) - \ell(p_-)} \Poinc(I,q,K) \mst{I}{q}{K}.
\end{displaymath}
We also have (again using Lemma \ref{lem:wmst}):
\begin{displaymath}
  P = \Poinc(I,p,J) \mst{I}{p}{J} *_J \mst{J}{}{K}.
\end{displaymath}
We follow that
\begin{equation}
  \label{eq:hecke-standard}
 \mst{I}{p}{J}*_J \mst{J}{}{K} = v^{\ell(q_-) - \ell(p_-)} \frac{ \Poinc(I,q,K)}{\Poinc(I,p,J)} \mst{I}{q}{K}
\end{equation}
By Corollary \ref{lem:poinc3} and the fact that $\He{}{}$ is free as a $\Z[v,v^{-1}]$-module.
\end{proof}

Given an element $h \in \He{I}{J}$ we may write $h = \sum \lambda_p
\mst{I}{p}{J}$. We define the \emph{support} of $h$ to be the finite
set
\begin{equation*}
\supp h = \{ p \in \W{I}{J} \; | \; \lambda_p \ne 0 \}.
\end{equation*}
A second corollary of the above multiplication formulas is a description of
multiplication by a standard generator on the support.

\begin{corollary} \label{cor:heckesupp}
Let $I, J, K \subset S$ be finitary with $J \subset K$ and let
\begin{equation*}
\quo : \W{I}{J} \to \W{I}{K}
\end{equation*}
be the quotient map.
\begin{enumerate}
\item If $h \in \He{I}{J}$ then
\begin{equation*}
\supp( h *_J \mst{J}{}{K}) \subset \quo (\supp h).
\end{equation*}
\item If $h \in \He{I}{K}$ then
\begin{equation*}
\supp( h *_K \mst{K}{}{J}) \subset \quo^{-1} (\supp h).
\end{equation*}
\end{enumerate}
\end{corollary}

We will not prove the following proposition, and instead refer the
reader to \cite[Proposition 2.2.7]{SSB}:

\begin{proposition} \label{prop:transseqprod}
Given any finitary $I, J \subset S$ and $p \in \W{I}{J}$ there exists
a sequence $(J_i)_{0 \le i \le n}$ of finitary subsets of $S$ such that, for
all $0 \le i < n$ either $J_i \subset J_{i+1}$ or $J_{i} \supset
J_{i+1}$ and such that
\begin{equation*}
\mst{I}{}{J_0} *_{J_0} \mst{J_0}{}{J_1} *_{J_1} \cdots *  _{J_{n-1}}
\mst{J_{n-1}}{}{J_n}
= \mst{I}{p}{J} + \sum_{q < p} \lambda_q \mst{I}{q}{J}.
\end{equation*}
\end{proposition}

Let $R$ be a ring and $\mathcal{C}$ be an $R$-linear category. Suppose we are given a subset $X_{AB} \subset \Hom(A,B)$ for all pairs of objects $A, B \in \mathcal{C}$. We define the \emph{span} of the collection $\{ X_{AB} \}$ to be the smallest collection of $R$-submodules $\{ Y_{AB} \subset \Hom(A,B) \}$ such that:
\begin{enumerate}
\item  $X_{AB} \subset Y_{AB}$ for all $A, B \in \mathcal{C}$,
\item The collection $\{Y_{AB}\}$ is closed under composition in $\mathcal{C}$.
\end{enumerate}
We say that $\{ X_{AB} \}$ \emph{generates} $\mathcal{C}$ if the span
of $\{ X_{AB} \}$ is equal to $\Hom(A,B)$ for all $A, B \in \mathcal{C}$. Less formally, one may refer to the span of any set of morphisms in $\mathcal{C}$ and ask whether they generate the category.

Proposition \ref{prop:transseqprod} implies:

\begin{corollary} The standard generators $\mst{I}{}{J}$ for finitary $I, J \subset S$ with either $I \subset J$  or $I \supset J$ generate the Hecke category. \end{corollary}

\begin{remark} It is natural to ask what relations the arrows $\mst{I}{}{J}$ satisfy. We have not looked into this. \end{remark}

In the previous subsection we defined a bilinear form on $\He{}{}$. We now generalise this construction and define a bilinear form on each $\He{I}{J}$ for $I, J \subset S$ finitary. Recall that $i: \He{}{} \to \He{}{}$ denotes the $\Z[v,v^{-1}]$-linear anti-involution sending $H_x$ to $H_{x^{-1}}$. As $\h{\wo{I}}$ and $\h{\wo{J}}$ are fixed by $i$ it follows that $i$ restricts to an isomorphism of $\Z[v,v^{-1}]$-modules
\begin{equation*}
i : \He{I}{J} \to \He{J}{I}.
\end{equation*}
We define \label{lab:biform}
\begin{align*} 
\He{I}{J} \times \He{I}{J} & \to \Z[v,v^{-1}] \\
(f, g) & \mapsto  \langle f, g \rangle := \text{ coefficient of $H_{\id}$ in $f *_J i(g)$.}
\end{align*}
We do not include reference to $I$ and $J$ in the notation, and hope that this will not lead to confusion. It follows from the definition that if $I, J, K \subset S$ are finitary and $f \in \He{I}{J}$, $g \in \He{J}{K}$ and $h \in \He{I}{K}$ then
\begin{equation} \label{eq:formident}
  \langle f *_J g, h \rangle = \langle f , h *_K i(g) \rangle.
\end{equation}
The following lemma describes the bilinear form on the standard basis of $\He{I}{J}$.

\begin{lemma}  \label{lem:biform2}
Let $I, J \subset S$ be finitary. For all $p, q \in \W{I}{J}$ we have
\begin{equation*}
\langle \mst{I}{p}{J} , \mst{I}{q}{J} \rangle = v^{\ell(p_+) - \ell(p_-)}
\frac{\Poinc(p)}{\Poinc(J)} \delta_{p,q}.
\end{equation*}
\end{lemma}

\begin{proof} Let $f, g \in \He{I}{J}$ and write $\tilde{f}$, $\tilde{g}$ for the elements $f$ and $g$ regarded as elements of $\He{}{}$. It is clear from the definition that
\begin{equation*}
\langle f, g \rangle = \frac{1}{\Poinc(J)} \langle \tilde{f}, \tilde{g} \rangle.
\end{equation*}
where the second expression is the bilinear form calulated in $\He{}{}$. We may then calculate using \eqref{eq:biform1}. If $p \ne q$ then $\langle \mst{I}{p}{J} , \mst{I}{q}{J} \rangle = 0$. If $p = q$ we have
\begin{equation*}
\langle \mst{I}{p}{J}, \mst{I}{q}{J} \rangle = 
 \frac{1}{\Poinc(J)} \sum_{x \in p} v^{2(\ell(p_+) - \ell(x))}
 = v^{\ell(p_+) - \ell(p_-)}
\frac{\Poinc(p)}{\Poinc(J)} . \qedhere
\end{equation*} 
\end{proof}

\section{Bimodules and homomorphisms} \label{sec:not}

Fix a field $k$ of characteristic 0. We consider rings $A$ satisfying
\begin{align}
  \label{eq:ringcond}
  & \text{$A = \oplus_{i \ge 0} A^i$ is a finitely generated, positively
graded} \\ &\text{commutative ring with $A^0 = k$.} \nonumber
\end{align}

We denote by $\lMod{A}$ and $\rMod{A}$ the category of graded left and right $A$-modules.
All tensor products are assumed to take place over  $k$, unless otherwise specified.
If $A_1$ and $A_2$ are two rings satisfying  (\ref{eq:ringcond}) we write $\bMod{A_1}{A_2}$ for the category of $(A_1, A_2)$-bimodules, upon which the left and right action of $k$ agrees. As all rings are assumed commutative we have an equivalence between $\bMod{A_1}{A_2}$ and $\lMod{A_1 \otimes A_2}$. We  generally prefer to work in $\bMod{A_1}{A_2}$, but will occasionally switch to $\lMod{A_1\otimes A_2}$ when convenient.

Given a graded module $M = \oplus M^i$ we define the
shifted module $M[n]$ by $(M[n])^i = M^{n+i}$. The endomorphism ring
of any finitely generated object in $\lMod{A}$, $\rMod{A}$ or
$\bMod{A_1}{A_2}$ is finite dimensional and hence any finitely
generated module satisfies Krull-Schmidt (for example, by adapting the proof in
\cite{Pie}).

Given a Laurent polynomial with positive coefficients
\begin{displaymath}
 P = \sum a_iv^i \in \N[v,v^{-1}]
\end{displaymath}
and an object $M$ in $\lMod{A}$, $\rMod{A}$ or $\bMod{A_1}{A_2}$, we define
\begin{displaymath}
 P \cdot M = \bigoplus M[i]^{\oplus a_i}.
\end{displaymath}
If $P, Q \in \N[v,v^{-1}]$ and $M$ and $N$ are finitely generated
modules such that
\begin{displaymath}
 P \cdot M \cong PQ \cdot N
\end{displaymath}
we may ``cancel $P$'' and conclude (using Krull-Schmidt) that
\begin{displaymath}
 M \cong Q \cdot N.
\end{displaymath}

Given two modules $M, N \in \bMod{A_1}{A_2}$ a morphism $\phi : M \to N$ of (ungraded) $(A_1,A_2)$-bimodules is of \emph{degree i} if $\phi(M^m) \subset \phi(N^{m+i})$ for all $m \in \Z$. We denote by $\Hom(M,N)^i$ the space of all morphisms of degree $i$ and
\begin{equation*}
\Hom(M,N) = \bigoplus_{i \in \Z} \Hom(M,N)^i.
\end{equation*}
We make $\Hom(M,N)$ into an object of $\bMod{A}{B}$ by defining an action of $a \in A$ and $b \in B$ on $f \in \Hom(M,N)$ via
\begin{equation*}
  (afb)(m) = f(amb) = af(m)b
\end{equation*}
for all $m \in M$. If $M$ and $N$ are objects in $\lMod{A}$ we similarly define $\Hom_A(M,N) \in \lMod{A}$. (We will only omit the subscript for morphisms of bimodules but will sometimes write $\Hom_{A_1\!-\!A_2}(M,N)$ if the context is not clear. We \emph{never} use $\Hom(M,N)$ to denote external (i.e. degree 0) homomorphisms.)

One may check that, if $P, Q \in \N[v,v^{-1}]$, then
\begin{equation*}
  \Hom(P \cdot M, Q \cdot N) \cong \overline{P}Q \cdot \Hom(M,N).
\end{equation*}
where $P \mapsto \overline{P}$ denotes the involution on $\N[v,v^{-1}]$ sending $v$ to $v^{-1}$.

In the sequel we will need various natural isomorphisms between homomorphism spaces, which we recall here.  Let $A_1, A_2$ and $A_3$ be three rings satisfying (\ref{eq:ringcond}). Let $M_{ij} \in \bMod{A_i}{A_j}$ for $i, j \in \{1,  2, 3 \}$. In $\bMod{A_1}{A_3}$ one has isomorphisms
\begin{align}
  \Hom_{A_1\!-\!A_3}(& M_{12}\otimes_{A_2} M_{23}, M_{13}) \nonumber \\
 \label{eq:adj1} & \cong \Hom_{A_1\!-\!A_2}(M_{12}, \Hom_{A_3}(M_{23}, M_{13})) \\
& \label{eq:adj2}
  \cong \Hom_{A_2\!-\!A_3}(M_{23}, \Hom_{A_1}(M_{12}, M_{13}))
\end{align}
because all three modules describe the same subset of maps $M_{12} \times M_{23} \to M_{13}$. For similar reasons, if $N \in \lMod{A_1}$ one has an isomorphism in $\lMod{A_1}$,
\begin{align}
  \label{eq:adj3}
  \Hom_{A_1}( & M_{12} \otimes_{A_2}  M_{23}, N) \cong \Hom_{A_2}(M_{23}, \Hom_{A_1}(M_{12}, N)).
\end{align}
Furthermore, this is an isomorphism in $\lMod{A_1\otimes A_3}$ if both sides are made into $A_1\otimes A_3$-modules in the only natural way possible.

If $M_{32}$ is graded free of finite rank as a right $A_2$-module one has an isomorphism
\begin{equation}
  \label{eq:adj4}
  \Hom_{A_2}(M_{32}, M_{12}) \cong M_{12} \otimes_{A_2} \Hom_{A_2}(M_{32}, A_2)
\end{equation}
in $\bMod{A_1}{A_3}$.

\section{Invariants, graphs and standard modules}  \label{sec:stand}

In this section we introduce standard modules, which are the building blocks of Soergel bimodules. Due to the inductive definition of Soergel bimodules, it will be necessary to be able to precisely describe the effect of extension and restriction of scalars on standard modules. Restriction turns out to be straightforward (Lemma \ref{cor:res}). Extension of scalars is more complicated, and we first need to define certain auxillary $(R,R)$-bimodules $R(p)$.

The structure is as follows.  In Section \ref{subsec:reffaith} we define what it means for a representation to be reflection faithful and recall some facts about invariant subrings. In the Section \ref{subsec:standard} we define standard objects
 and analyse the effect of restriction of scalars on them. In Section \ref{subsec:enlarge} we define the bimodules $R(p)$ and in Section \ref{subsec:extend} we use them to describe extension of scalars. In Section \ref{sec:support} we introduce the notion of support, which will be essential in what follows.

\subsection{Reflection faithful representations and invariants} \label{subsec:reffaith}

Let $(W,S)$ be a Coxeter system with reflections $T \subset
W$. A \emph{reflection faithful} representation of $W$ is a finite
dimensional representation $V$ of $W$ such that:
\begin{enumerate}
\item The representation is faithful;
\item \label{reflectionfaithfulcondition}
We have $\codim V^w = 1$ if and only if $w$ is a reflection.
\end{enumerate}
If $W$ is finite it is straightforward to see that the geometric
representation over $\mathbb{R}$ (\cite{Hu}, Proposition 5.3) satisfies the
second condition above,  because it preserves a positive definite
bilinear form. If $W$ is
infinite, this is not the case in general. However, one has
(\cite[Proposition 2.1]{SoBimodules}):

\begin{proposition} Given any Coxeter system $(W,S)$ there exists a reflection faithful representation of $W$ on a finite dimensional real vector space $V$. \end{proposition}

Let $V$ be a reflection faithful representation over an infinite field
$k$ of characteristic not equal to 2. Let $R$ be the graded ring of regular functions on $V$, with $V^*$ sitting in degree 2. \label{lab:R}
Because $k$ is infinite we may identify $R$ with
the symmetric algebra on $V^*$. As $W$ acts on $V$ it also acts on $R$
on the left via $(wf)(\lambda) = f(w^{-1}\lambda)$ for all $\lambda \in V$

If $w \in W$ we denote by $R^w$ the invariants under $w$. If $I
\subset S$ we denote by $R^I$ the invariants under
$W_I$. \label{lab:RI} Recall the definition of $\tPoinc(I)$ from
Section \ref{subsec:coxgroups}. Throughout this paper we assume:
\begin{equation} \label{assump:freeness}
\begin{array}{c}
\text{For all finitary $I \subset S$, $R$ is graded free
  over $R^I$,} \\
\text{and one has an isomorphism of graded $R^I$-modules:} \\
R \cong \tPoinc(I) \cdot R^I.
\end{array}
\end{equation}

\begin{remark} If $k$ is of characteristic 0 \eqref{assump:freeness}
  is always true. If $W$ is a finite Weyl group, then $W$ acts on the
  weight lattice of the corresponding root system and one obtains a
  representation over any field by extension of scalars. In this case,
  \eqref{assump:freeness} is true if the characteristic of $k$ is not
  a torsion prime for $W$ (see \cite{Dem}).
\end{remark}

Throughout we fix a reflection faithful representation such that
\eqref{assump:freeness} holds. The above assumptions imply (see
\cite[Corollary 2.1.4 and Corollary 3.2.3]{SSB}):

\begin{lemma} \label{cor:relinv} Let $I \subset J$ be finitary.
\begin{enumerate}
\item 
The $R^J$-module $R^I$ is a graded free
one has an isomorphism:
\begin{equation*}
R^I \cong \frac{ \tPoinc(J)}{\tPoinc(I)} \cdot R^J.
\end{equation*}
\item 
\label{cor:reldual}
We have an isomorphism:
\begin{displaymath}
\Hom_{R^{J}}(R^{I}[\ell(\wo{J}) - \ell(\wo{I})], R^{J}) \cong R^{I}[\ell(\wo{J}) - \ell(\wo{I})].
\end{displaymath}
\end{enumerate}
\end{lemma}

Because of our assumptions all reflections $t \in T$ act via
\begin{equation} \label{eq:reflectact}
  t(\lambda) = \lambda - 2h_t(\lambda)v_t
\end{equation}
for some linear form $h_t \in V^*$ and vector $v_t \in V$. The pair $(h_t, v_t)$ is only determined up to a choice of scalar. However, one may choose $h_t \in V^*$ such that
\begin{equation} \label{eq:hident}
  xh_s = h_t \text{ if $xsx^{-1} = t$}
\end{equation}
where we regard $V^*$ as a $W$-module via the contragredient
action. The elements $h_t \in V^*$ (which give equations for the
hyperplane $V^t$) will be important in the sequel. For this reason we
make a fixed choice of the set $\{ h_t \;|\; t \in T \}$ with the only
restriction being that (\ref{eq:hident}) should hold. \label{lab:ht}
An immediate consequence of the condition 2) in the definition of reflection faithful is the following (\cite[Bemerkung 1.6]{SoBimodules}):
\begin{lemma} The elements of $\{ h_t \; | \; t \in T \} \subset V^*$ are pairwise linearly independent. \end{lemma}

\subsection{Singular standard modules} \label{subsec:standard}

In this section we define ``standard modules''. These are graded $(R^I,R^J)$-bimodules indexed by triples $(I,p, J)$ where $I, J \subset S$ are finitary and and $p \in \W{I}{J}$ is a double coset.

\begin{definition} \label{lab:standard}
Let $I, J \subset S$ be finitary, $p \in \W{I}{J}$ and set $K = I \cap p_-Jp_-^{-1}$. The \emph{standard module indexed by $(I, p, J)$}, denoted $\RS{I}{p}{J}$, is the ring $R^K$ of $W_{K}$-invariant functions in $R$. We make $\RS{I}{p}{J}$ into an object in $\bMod{R^I}{R^J}$ by defining left and right actions as follows:
  \begin{align*}
    r \cdot m &= rm & \text{for $r \in R^I$ and $m \in \RS{I}{p}{J}$}\\
 m \cdot r &= m(p_-r) & \text{for $m \in \RS{I}{p}{J}$ and $r \in R^J$}
  \end{align*}
(where $rm$ and $(p_-r)m$ denotes multiplication in $R^{K}$). If $I = J = \emptyset$ we write $R_w$ instead of $\RS{I}{w}{J}$.
\end{definition}

 This action is well-defined because if $r \in R^I$ (resp. $r \in
 R^J$) then $r$  (resp. $p_-r$) lies in $R^{K}$. In the future we will
 supress the dot in the notation for the left and right action.  If
 $p$ contains $\id \in W$ we sometimes omit $p$ and write simply
 $\RS{I}{}{J}$. Note that the graded rank of the standard modules may
 vary across double cosets.

The following lemma describes the effect of restriction of scalars on standard objects.

\begin{lemma} \label{cor:res}
Let $w \in W$,  $I,J \subset S$ be finitary and $p = W_IwW_J$ be the $(W_I, W_J)$-double coset containing $w$.  Then in $\bMod{R^I}{R^J}$ we have an isomorphism:
  \begin{equation*}
    _{R^I}(R_w)_{R^J} \cong \tPoinc(I, p, J) \cdot \RS{I}{p}{J}.
  \end{equation*}
Furthermore, if $I \subset K$, $J \subset L$ are finitary and $q = W_KpW_L$ then
\begin{equation*}
  _{R^K}(\RS{I}{p}{J})_{R^L} \cong \frac{\tPoinc(K,q,L)}{\tPoinc(I,p,J)} \cdot \RS{K}{q}{L}
\end{equation*}
in $\bMod{R^K}{R^L}$.
\end{lemma}

\begin{proof} If $v \in W_J$ then $R_w$ and $R_{wv}$ become isomorphic
  when we view them as objects in $\bMod{R}{R^J}$. Similarly, if $u
  \in W_I$ then the map $r \mapsto ur$ gives an isomorphism between
  $R_w$ and $R_{uw}$ when regarded as objects in $\bMod{R^I}{R}$. Thus
  we may  assume without loss of generality that $w = p_-$. Define $K
  = I \cap p_-Jp_-^{-1}$ so that $\tPoinc(I,p,J) = \tPoinc(K)$. The
  first isomorphism follows from the definition of $\RS{I}{p}{J}$ and
  the decomposition (see \eqref{assump:freeness})
  \begin{equation*}
    R \cong \tPoinc(K) \cdot R^{K}.
  \end{equation*}

For the second statement note that, by the transitivity of restriction and the above isomorphism we have
\begin{equation*}
  \tPoinc(I,p,J) \cdot {}_{R^K}(\RS{I}{p}{J})_{R^L} = \tPoinc(K,q,L) \cdot \RS{K}{q}{L} \quad \text{ in $\bMod{R^K}{R^L}$}.
\end{equation*}
As $\tPoinc(K,q,L)/\tPoinc(I,p,J) \in \N[v,v^{-1}]$ by Lemma \ref{lem:poinc3} we may divide by $\tPoinc(I,p,J)$. The claimed isomorphism follows. \end{proof}

\subsection{Enlarging the regular functions} \label{subsec:enlarge}

Our ultimate aim for the rest of this section is to understand the effect of extending scalars on standard modules. However, in order to do this we need to introduce certain auxillary modules $R(X) \in \bMod{R}{R}$ corresponding to finite subsets $X \subset W$.

Given $w \in W$ we define its (twisted) graph
\begin{displaymath}
\Gr{}{w}{} = \{ (w\lambda, \lambda) \; | \; \lambda \in V \}
\end{displaymath}
which we view as a closed subvariety of $V \times V$. Given a finite subset
$X \subset W$ we denote by $\Gr{}{X}{}$ the subvariety
\begin{equation*}
\Gr{}{X}{} = \bigcup_{w \in X} \Gr{}{w}{}.
\end{equation*}
We will denote by $\SO(\Gr{}{X}{})$ the regular functions on $\Gr{}{X}{}$ which  has the structure of an $R$-bimodule via the inclusion $\Gr{}{X}{} \hookrightarrow V\times V$.

For all $x \in  W$ consider the inclusion
\begin{align*}
  i_x : V & \hookrightarrow V \times V \\
\lambda & \mapsto (\lambda, x^{-1}\lambda).
\end{align*}
This provides an isomorphism of $V$ with $\Gr{}{x}{}$ and an explicit identification of $R_x$ and $\SO(\Gr{}{x}{})$ as $R$-bimodules.

The following lemma will be important in the next section (its proof follows by the same arguments as
  \cite[Lemma 2.2.2]{SoHarishChandra}).

\begin{lemma} \label{lem:SoergelIso}
Let $I \subset S$ be finitary. We have an isomorphism of graded $k$-algebras
\begin{equation*}
  R \otimes_{R^I} R \cong \SO(\Gr{}{W_I}{}).
\end{equation*}
\end{lemma}

Recall that, for all $t \in T$, we have chosen an equation $h_t \in V^*$ for the hyperplane fixed by $t$. We will denote by $(h_t) \subset R$ the ideal generated by $h_t$.
We now come to the definition of the $R$-bimodules $R(X)$.

\begin{deprop} \label{prop:Rpdef}
Let $X \subset W$ be a finite subset. Consider the subspace
\begin{equation*}
R(X) = \left \{ f = (f_x) \in \bigoplus_{x \in X} R \; \middle | \; 
\begin{array}{c} f_x - f_{tx} \in (h_t) \\ \text{ for all } t \in T
\text{ and } x, tx \in X\end{array}  \right \} \subset \bigoplus_{x \in X} R.
\end{equation*}
Then $R(X)$ is a graded $k$-algebra under componentwise multiplication and becomes an object of $\bMod{R}{R}$ if we define left and right actions of $r \in R$ via
\begin{align*}
 (rf)_x & = rf_x \\
 (fr)_x & = f_x(xr)
\end{align*}
for $f = (f_x) \in R(X)$. If a pair of subgroups $W_1, W_2 \subset W$ satisfy $W_1X = X = XW_2$ then $R(X)$ carries commuting left $W_1$- and right $W_2$-actions if we define
\begin{align*}
  (uf)_x &= uf_{u^{-1}x} & \text{for $u \in W_1$},\\
  (fv)_x &= f_{xv^{-1}} & \text{for $v \in W_2$}.
\end{align*}
If $X = \{ x \}$ is a singleton then $R(X) \cong R_x$. If $X = \{x, y \}$ consists of two elements we write $R_{x,y}$ instead of $R(X)$.
\end{deprop}

\begin{proof} It is straightforward to check that $R(X)$ is a graded subring containing $k$. In order to see that the left and right $R$-operations preserve $R(X)$ it is therefore enough to check that $(r)_{x \in X}$ and $(xr)_{x \in X}$ are elements of $R(X)$ for all $r \in R$. This is clear for $(r)_{x \in X}$ and for $(xr)_{r \in X}$ it follows from the formula $tg = g  - g(v_t)h_t$ for $g \in V^*$. The right $W_2$-operation clearly preserves $R(p)$. For the left $W_1$-operation if $x, tx \in X$ one has, using (\ref{eq:hident}),
\begin{equation*}
  (wf)_x - (wf)_{tx} = w(f_{w^{-1}x} - f_{w^{-1}tx}) \in (w(h_{w^{-1}tw})) = (h_t).
\end{equation*}
The operations clearly commute and the fact that $R(X) \cong R_x$ if $X = \{ x \}$ is immediate from the definitions.
\end{proof}

\begin{remark} \rule{0mm}{5mm}
\begin{enumerate}
\item We have defined $R(X)$ for general finite subsets $X \subset W$ but will only ever need two cases: 
\begin{enumerate}
\item  $X = p$ is a $(W_I, W_J)$-double coset for finitary $I, J \subset S$.
\item $X  = \{ x, tx \}$ for some $x\in W$ and reflection $t \in T$.
\end{enumerate}
\item
The graded ring $R(X)$ has a natural description in terms of the Bruhat graph of $W$. Let $\mathcal{G}_X$ be the full subgraph of the Bruhat graph of $W$ with vertices $X$. Then an element of $R(X)$ can be thought of as a choice of $f_x \in R$ for every vertex $x \in \mathcal{G}_p$, subject to the conditions that $f_x - f_y$ lies in $(h_t)$ whenever $x$ and $y$ are connected by an edge labelled $t$. Under this description the left action of $R$ is just the diagonal action, and the right action is the diagonal action ``twisted'' by the label of each vertex. The left $W_1$- and right $W_2$-actions are induced (with a twist for the action of $W_1$) by the left and right multiplication action of $W_1$ ad $W_2$ on $X$.
\excise{
\item Assume that $G \supset B \supset T$, $S \subset W$, $V = \Lie T$ are as in the introduction and $X = p \in W_I \setminus W / W_J$. As left modules one may identify $R(X)$ with the $T$-equivariant cohomology of $P_IpP_J/B \subset G/B$.}
\end{enumerate}
\end{remark}

The following proposition gives a useful alternative description of $R(X)$.

\begin{proposition} \label{prop:exactsequence}
Let $X \subset W$ be a finite set. There exists an exact sequence in $\bMod{R}{R}$
  \begin{equation*}
    0 \to R(X) \to \bigoplus_{x \in X} R_x \to \bigoplus_{x <tx \in X \atop t \in t} R_{x}/(h_t)
  \end{equation*}
  where the maps are as described in the proof.
\end{proposition}

\begin{proof} The first map is the inclusion of $R(X)$ into $\bigoplus_{x \in X} R_x$ which is clearly a morphism of $R$-bimodules. We describe the second map by describing its components
  \begin{equation*}
    R_x \to R_y/(h_t).
  \end{equation*}
This map is zero if $x \notin \{ y, ty \}$. Otherwise it is given by
\begin{equation*}
  f \mapsto \epsilon_{x,  tx} f + (h_t)
\end{equation*}
where $\epsilon_{x,tx}$ is defined by
\begin{equation*}
  \epsilon_{x,tx} =  \left \{ \begin{array}{ll} 1 & \text{if $x < tx$} \\
-1 & \text{ if $x > tx$} \end{array} \right . .
\end{equation*}
This is a morphism in $\bMod{R}{R}$ because this is true of the quotient map $R_x \to  R_{y}/(h_t)$ whenever $x = y$ or $x = ty$. Lastly a tuple $(f_x) \in \oplus R_x$ is mapped to zero if $f_x = f_{tx}$ in $R_x / (h_t)$ for all $x, tx \in X$ and $t \in T$, which is exactly the condition for $(f_x)$ to belong to $R(X)$.
\end{proof}

The following lemma explains the title of this subsection.

\begin{lemma} \label{lem:Rpinjection}
Let $X \subset W$ be finite. The map
  \begin{align*}
    \rho: \SO(\Gr{}{X}{}) & \to R(X) \\
f & \mapsto (i_x^*f)_{x \in p}
  \end{align*}
is well-defined, injective and a morphism in $\bMod{R}{R}$.
\end{lemma}

\begin{proof} 
Any regular function $f \in \SO(\Gr{}{X}{})$ is determined by its restriction to all $\Gr{}{x}{}$ for $x \in X$, which is just the tuple
\begin{equation*}
  (i_x^*f)_{x \in p} \in \bigoplus_{x \in p} R.
\end{equation*}
We claim that this tuple lies in $R(X)$. Indeed, we just need to check that $i_x^*f$ and  $i_{tx}^*f$ agree on $V^t$ if $x, tx \in p$ for some $t \in T$ and this is straightforward. It follows that the map is an injection of graded $k$-algebras, in particular an injection in $\bMod{R}{R}$.
\end{proof}

\begin{remark}
In general the map
\begin{equation*}
  \rho : \SO(\Gr{}{X}{}) \hookrightarrow R(X)
\end{equation*}
is not surjective. The question as to when it is seems quite
subtle. See \cite[Remark 2.3.3]{SSB} for a discussion.
\end{remark}

Because $R(X)$ has the structure of a graded $k$-algebra we have an injection
\begin{equation*}
  R(X) \hookrightarrow \Hom(R(X), R(X)).
\end{equation*}
In fact:
\begin{proposition} \label{prop:RXend} For all finite subsets $X \subset W$ we have
  \begin{equation*}
    \Hom(R(X),R(X))  = R(X).
  \end{equation*}
\end{proposition}

\begin{proof}
For the course of the proof it will be more convienient to regard $R(X)$ as graded
 left $R\otimes R$-module. Let $\varphi : R(X) \to R(X)$ be a morphism in $\lMod{R\otimes R}$ and denote by $f = (f_x)_{x \in X}$ the image of 1. Choose $m = (m_x) \in R(X)$. We will be finished if we can show that $\varphi(m)_z = m_zf_z$ for all $z \in X$. Let us choose $z \in X$ and let $g \in R\otimes R$ be a  function that vanishes on $\Gr{}{y}{}$ for $z  \ne y$ but not on $\Gr{}{z}{}$, and let $(g_x)$ denote its image in $R(X)$ (the result of acting with $g$ on $1 \in R(X)$). Note that
\begin{equation*}
  (gm)_x= \delta_{x,z}g_x m_x
\end{equation*}
and so $gm$ is in the image of $R \otimes R$. Hence
\begin{equation*}
  g_z\varphi(m)_z =  (g\varphi(m))_z = \varphi(gm)_z = f_zg_zm_z
\end{equation*}
and hence $\varphi(m)_z = m_zf_z$ as $g_z$ is non-zero.
\end{proof}

\subsection{Standard modules and extension of scalars} \label{subsec:extend}

The aim of this subsection is to study the effect of extension of scalars on standard modules. That is, we want to understand the bimodules
\begin{equation*}
  R^K \otimes_{R^I} \RS{I}{p}{J} \otimes_{R^J} R^L \in \bMod{R^K}{R^L}
\end{equation*}
where $K \subset I$ and $L \subset J$ are finitary. The key is provided by the bimodules $R(X)$ introduced in the previous section.

For the rest of this subsection fix finitary subsets $I, J \subset S$ and a double coset $p \in \W{I}{J}$. Recall that the bimodules $R(p)$ have commuting left $W_I$- and right $W_J$-actions.  Of course we can make this into a left $W_I \times W_J$ action by defining $(u,v)m = umv^{-1}$ for all $m \in R(p)$.

\begin{theorem} \label{thm:ind}
Let $I \supset K$ and $J \supset L$. There exists an isomorphism
  \begin{equation*}
    R^K \otimes_{R^I} \RS{I}{p}{J} \otimes_{R^J} R^L \stackrel{\sim}{\to} R(p)^{W_K \times W_L}
  \end{equation*}
in $\bMod{R^K}{R^L}$. \end{theorem}

The theorem will take quite a lot of effort to prove. In Lemmas \ref{lem:Rpinv} and \ref{lem:phicommutes} below we construct a morphism
 \begin{equation*}
    R \otimes_{R^I} \RS{I}{p}{J} \otimes_{R^J}  R \to R(p).
  \end{equation*}
  commuting with natural actions of $W_K \times W_L$ on both sides. By considering invariants one may reduce the theorem to showing that this map is an isomorphism.

Let us first describe the $W_I \times W_J$ actions. By Proposition \ref{prop:Rpdef} and the discussion at the beginning of this section there is an $W_I \times W_J$-action on $R(p)$. We define a $W_I \times W_J$-action on $R \otimes_{R^I} \RS{I}{p}{J} \otimes_{R^J} R$ via
\begin{equation*}
  (u,v) f \otimes g \otimes h = uf \otimes g \otimes vh.
\end{equation*}
It is easy to see that this action is well-defined.

The following lemma tells us how to find the standard module $\RS{I}{p}{J}$ as a submodule of $R(p)$.

\begin{lemma} \label{lem:Rpinv}
In $\bMod{R^I}{R^J}$ we have an isomorphism
\begin{equation*}
R(p)^{W_I \times W_J} \cong \RS{I}{p}{J}.
\end{equation*}
\end{lemma}

\begin{proof} Let $K  = I \cap  p_-Jp_-^{-1}$ and choose $f \in R(p)^{W_I \times W_J}$. If $u \in W_K$ then $up_- = p_-v$ for some $v \in W_J$ and $f_{p_-} = ((u,v)f)_{p_-} = uf_{p_-}$. In other words $f_{p_-} \in R^K$. Hence we obtain a map
\begin{align*}
R(p)^{W_I \times W_J} &\to \RS{I}{p}{J} \\
(f_x) & \mapsto f_{p_-}
\end{align*}
which is obviously injective and a morphism in $\bMod{R^I}{R^J}$.

It remains to show surjectivity. To this end choose $m \in \RS{I}{p}{J}$ and consider  the tuple $f = (f_x) \in \oplus_{x \in p} R$ where, for each $x \in p$ we choose $u \in W_I, v \in W_J$ with $x = up_-v$ and define $f_x = um$. This  is well defined because if $up_-v = u^{\prime}p_-v^{\prime}$ with $u, u^{\prime} \in W_I$ and $v, v^{\prime} \in W_J$ then $u^{-1}u^{\prime} \in W_I \cap p_-W_Jp_-^{-1} = W_K$ by Kilmoyer's Theorem (\ref{Kilmoyer}), and hence $um = u^{\prime}m$ as $m$ is invariant under $W_K$. The tuple $(f_x)$ also lies in $R(p)$ as if $x$ and $tx$ both lie in $p$ then by Proposition \ref{prop:doublebruhat} either $t \in W_I$ (in which case $f_{tx} = tf_x$) or $tx = xt^{\prime}$ for some reflection $t^{\prime}$ in $W_J$ (in which case $f_{x} = f_{tx}$). Lastly, it it easy to check that $f$ is $W_I \times W_J$ invariant. As $f$ gets mapped to $m$ under the above map, we see that the map is indeed surjective.
\end{proof}

Having identified $\RS{I}{p}{J}$ as a $(R^I, R^J)$-submodule of $R(p)$ we obtain by adjunction a morphism
\begin{equation*}
\mu : R \otimes_{R^I} \RS{I}{p}{J} \otimes_{R^J} R \to R(p).
\end{equation*}
We will see below that this is an isomorphism. However first we need:

\begin{lemma} \label{lem:phicommutes}
The morphism $\mu$ commutes with the $W_I \times W_J$-actions on both modules. \end{lemma}

\begin{proof} This is a technical but straightforward calculation. Let $a = r_1 \otimes m \otimes r_2 \in R \otimes_{R^I} \RS{I}{p}{J} \otimes_{R^J} R$ and $(u,v) \in W_I \times W_J$. We want to show that $\mu((u,v)a)  = (u,v)\mu(a)$.

Under $\mu$, $a$ gets mapped to $f = (f_z) \in R(p)$ where
\begin{equation*}
  f_z = r_1(xm)(zr_2)
\end{equation*}
if $z = xp_-y$ with $x \in W_I$ and $y \in W_J$. Similarly $(u,v)a = ur_1 \otimes m \otimes vr_2$ gets mapped to $\tilde{f} = (\tilde{f}_z) \in R(p)$ where
\begin{equation*}
  \tilde{f}_z = ur_1(xm)(zvr_2).
\end{equation*}
We need to show that $(u,v)f = \tilde{f}$. This follows from
\begin{equation*}
  ((u,v)f)_z = uf_{u^{-1}zv} = u(r_1(u^{-1}xm)(u^{-1}zvr_2)) = ur_1(xm)(zvr_2) = \tilde{f}_z. \qedhere
\end{equation*}

\end{proof}

\begin{proof}[Proof of Theorem \ref{thm:ind}] By considering $W_K \times W_L$ invariants it is enough to show that the morphism $\mu$ constructed above is an isomorphism. This will follow from two facts which we verify below:
  \begin{enumerate}
  \item Both $R(p)$ and $R \otimes_{R^I} \RS{I}{p}{J} \otimes_{R^J} R$ are isomorphic to $\tPoinc(p) \cdot R$ as graded left $R$-modules;
   \item The morphism $\mu$ is injective.
  \end{enumerate}
Indeed (1) says that each  graded component of $R(p)$ and $R \otimes_{R^I} \RS{I}{p}{J} \otimes_{R^J} R$ is of the same (finite) dimension over $k$. Using (2) we then see that $\varphi$ is an isomorphism on each graded component and hence is an isomorphism.

We start by establishing (1) for $R \otimes_{R^I} \RS{I}{p}{J}
\otimes_{R^J} R$. Choose $w \in p$. By \eqref{assump:freeness} we have
an isomorphism of left $R$-modules: 
   \begin{equation*}
      R \otimes_{R^I} R_w \otimes_{R^J} R \cong \tPoinc(I)\tPoinc(J) \cdot R.
    \end{equation*}
Hence, by Lemma \ref{cor:res} we have (again as left $R$-modules):
\begin{equation*}
  \tPoinc(I,p,J) \cdot R \otimes_{R^I} \RS{I}{p}{J} \otimes_{R^J} R
\cong \tPoinc(I)\tPoinc(J) \cdot R
\end{equation*}
Dividing by $\tPoinc(I,p,J)$ and using Lemma \ref{lem:poinc3} we conclude that
\begin{equation} \label{eq:indlmod}
  R \otimes_{R^I} \RS{I}{p}{J} \otimes_{R^J} R
\cong \tPoinc(p) \cdot R \qquad \text{ in $\lMod{R}$}
\end{equation}
as claimed.

It seems much harder to establish (1) for $R(p)$. This is Corollary \ref{cor:Rpleft} of the next section, which we prove using Demazure operators.

The rest of the proof will be concerned with (2). Choose again $w \in p$. Using Lemma \ref{lem:SoergelIso} we may identify $R \otimes_{R^I} R_w \otimes_{R^J} R$ with the regular functions on the variety
  \begin{equation*}
    Z = \left \{ (\lambda, \mu, \nu) \middle | \begin{array}{c} \lambda = u\mu  \text{ for some $u \in W_I$}\\\mu = wv \nu \text{ for some $v \in W_J$} \end{array} \right \} \subset V \times V \times V.
  \end{equation*}
We have an obvious projection map $Z \to \Gr{}{p}{}$ sending $(\lambda, \mu, \nu)$ to $(\lambda, \nu)$ and hence we have a morphism in $\bMod{R}{R}$ (in fact of $k$-algebras)
\begin{equation*}
\SO(\Gr{}{p}{}) \to R \otimes_{R^I} R_w \otimes_{R^J} R.
\end{equation*}
Taking $w = p_-$ this map lands in $R \otimes_{R^I} \RS{I}{p}{J} \otimes_{R^J} R$ regarded as a submodule of $R \otimes_{R^I} R_w \otimes_{R^J} R$. We conclude the existence of a commutative diagram
\begin{equation*}
  \xymatrix@C=0cm{
& \SO(\Gr{}{p}{}) \ar[dl] \ar[dr]^{\rho} & \\
R \otimes_{R^I} \RS{I}{p}{J} \otimes_{R^J} R \ar[rr]^(0.65){\varphi} & & R(p) }
\end{equation*}
where $\rho$ is as in Lemma \ref{lem:Rpinjection}.

We now argue that all arrows become isomorphisms after tensoring with $\Quot R$. As $\rho$ is injective and $\Quot R$ is flat over $R$ it is enough to show that all modules have dimension $|p|$ over $\Quot R$ after applying $\Quot R \otimes_R -$. This is indeed the case:
\begin{enumerate}
  \item $\SO(\Gr{}{p}{})$: For the same reasons as in the proof of Lemma \ref{lem:SoergelIso}.
  \item $R \otimes_{R^I} \RS{I}{p}{J} \otimes_{R^J} R$: This follows from (\ref{eq:indlmod})
  \item $R(p)$: By applying $\Quot R \otimes_R -$ to the exact sequence in Proposition \ref{prop:exactsequence}.
\end{enumerate}
We conclude that all maps (in particular $\mu$) become isomorphisms after applying $\Quot R \otimes_R -$.

To conclude the proof, note that by the above arguments $R \otimes_{R^I} \RS{I}{p}{J} \otimes_{R^J} R$ is torsion free as a left $R$-module. Hence $\mu$ is injective if and only if this is true after applying $\Quot R \otimes_R -$. Thus $\mu$ is injective as claimed.
\end{proof}

We may use this theorem to determine the morphisms between standard modules. Recall that $\RS{I}{p}{J}$ was defined as a subring of $R$, and therefore has the structure of a  $k$-algebra compatible with its $(R^I, R^J)$-bimodule structure. Therefore we certainly have an injection
\begin{equation*}
  \RS{I}{p}{J} \hookrightarrow \Hom(\RS{I}{p}{J}, \RS{I}{p}{J}).
\end{equation*}
In fact:

\begin{corollary} \label{cor:standardhoms}
For $p,q \in \W{I}{J}$ we have
\begin{equation*}
  \Hom(\RS{I}{p}{J}, \RS{I}{q}{J}) = \left \{ \begin{array}{ll}
\RS{I}{p}{J} & \text{if $p = q$} \\
0 & \text{otherwise.} \end{array}\right . 
\end{equation*}
\end{corollary} 

\begin{proof} Extension  of scalars give us an map
  \begin{equation*}
    \Hom(\RS{I}{p}{J}, \RS{I}{q}{J}) \to \Hom(R \otimes_{R^I} \RS{I}{p}{J} \otimes_{R^J} R , R \otimes_{R^I} \RS{I}{q}{J} \otimes_{R^J} R)
  \end{equation*}
which is injective because we may again restrict to
$\bMod{R^I}{R^J}$. By the above theorem the latter module is
isomorphic to $\Hom(R(p), R(q))$. This is 0 if $p \ne q$ because
$\Hom(R_x, R_y) = 0$ if $x \ne y$. Otherwise $\Hom(R(p),R(p)) = R(p)$
by Proposition \ref{prop:RXend}, and so $\Hom(\RS{I}{p}{J},
\RS{I}{p}{J})$ consists of those $\alpha \in \Hom(R(p),R(p))$ for
which $\alpha(1) \in \RS{I}{p}{J}$. Hence $\Hom(\RS{I}{p}{J},
\RS{I}{p}{J})= \RS{I}{p}{J}$ as claimed. \end{proof}

\subsection{Support} \label{sec:support}

Let $X$ be an affine variety over $k$ and $A$ its $k$-algebra of
regular functions. We will make use of the equivalence between
(finitely-generated) $A$-modules and (quasi)-coherent sheaves on $X$
(see \cite{Har}, Chapter II, Corollary 5.5). If $M$ is an $A$-module, and $\mathcal{M}$ is the
corresponding quasi-coherent sheaf on $X$, then the \emph{support} of
$\mathcal{M}$, which we will denote $\supp M$ by abuse of notation,
consists of those points $x \in X$ for which $\mathcal{M}_x \ne 0$. The support of a section $m \in M$, denoted $\supp m$, is the support of the submodule generated by $m$.
It follows from the definition that if $M^{\prime} \hookrightarrow M \sur M^{\prime\prime}$ is an
exact sequence of $A$-modules then
\begin{equation} \label{eq:suppexact}
\supp M = \supp M^{\prime} \cup \supp M^{\prime\prime}.
\end{equation}
If $M$ is finitely generated then the support of $M$ is the closed
subvariety of $X$ determined by the annihilator of $M$ (\cite{Har}, II, Exercise 5.6(b)).

Let $f : X \to Y$ be a map of affine varieties and $A \to B$ be the
corresponding map of regular functions. If $M$ and $N$ are $A$- and $B$-modules respectively, then
\begin{align}
f(\supp N) \subset \supp ({}_{A}N) \subset \overline{f(\supp N)}, \label{eq:supp1} \\
\supp( B \otimes_A M) = f^{-1}(\supp M). \label{eq:supp2}
\end{align}
The first is an exercise, and the second is Exercise 19(viii), Chapter 3 of \cite{AM} for finitely generated $M$, but seems to be true in general (in any case we only need it for finitely generated $M$). It follows that if $f$ is finite (hence closed) and $N$ is finitely generated, then
\begin{equation} \label{eq:supp3}
  f(\supp N) = \supp ({}_{A}N).
\end{equation}

The rest of this section will be concerned with applying notions of support to objects in $\bMod{R^I}{R^J}$, where $I, J \subset S$ are finitary. This is possible as we may regard any such object as an $R^I\otimes R^J$-module. We identify $R^I \otimes R^J$ with the regular functions on  the quotient $V / W_I \times V/W_J$. Thus, given any $M \in \bMod{R^I}{R^J}$, $\supp M \subset V/W_I \times V/W_I$.

In Section \ref{subsec:enlarge}, we defined the twisted graph $\Gr{}{x}{} \subset V
\times V$ as well as $\Gr{}{C}{}$ for finite subsets $C \subset W$. For a double coset $p \in \W{I}{J}$ denote by $\Gr{I}{p}{J}$ the image of $\Gr{}{p}{}$ under the quotient
map $V \times V \to V/ W_I \times V / W_J$. The subvariety $\Gr{I}{p}{J}$ is equal to the image of $\Gr{}{x}{}$ for any $x \in p$ and thus is irreducible. Given any set $C \subset \W{I}{J}$, we define
\begin{equation*} \label{lab:IgrpJ}
  \Gr{I}{C}{J} = \bigcup_{p \in C} \Gr{I}{p}{J}
\end{equation*}
which we understand as a subvariety if $C$ is finite, and as a set if $C$ is infinite.

We will be interested in those $M \in \bMod{R^I}{R^J}$ whose support
is contained in $\Gr{I}{C}{J}$ for some finite set $C \subset
\W{I}{J}$. Given finitary $I \subset K$ and $J \subset L$ we have
functors of  restriction and extension of scalars between
$\bMod{R^I}{R^J}$ and $\bMod{R^K}{R^L}$. Because the inclusion $R^K
\otimes R^L \to R^I \otimes R^J$ corresponds to the finite map
\begin{equation*}
  V/W_I \times V/W_J \to V/W_K \times V/W_L
\end{equation*}
we may translate (\ref{eq:supp2}) and (\ref{eq:supp3}) as follows:

\begin{lemma} \label{lem:support}
Let $I \subset K$ and $J \subset K$ be finitary subsets of $S$ and let
\begin{equation*}
  \quo : \W{I}{J} \to \W{K}{L}
\end{equation*}
denote the quotient map.
 \begin{enumerate}
 \item If $M \in \bMod{R^I}{R^J}$ and $\supp M = \Gr{I}{C}{J}$
for some finite subset $C \subset \W{I}{J}$ then $\supp ({}_{R^K}M_{R^L}) = \Gr{K}{\quo(C)}{L}$.
\item If $N \in \bMod{R^K}{R^L}$ and $\supp M = \Gr{I}{C^{\prime}}{J}$
for some finite subset $C^{\prime} \subset \W{K}{L}$ then $\supp (R^I \otimes_{R^K} M \otimes_{R^L} R^J) = \Gr{I}{\quo^{-1}(C^{\prime})}{J}$.
 \end{enumerate}
The same is true with ``$=$'' replaced with ``$\subset$'' throughout.
\end{lemma}

Given a set $C \subset \W{I}{J}$ and $M \in \bMod{R^I}{R^J}$ we denote by $\Gamma_C M$ the submodule of sections with support in $\Gr{I}{C}{J}$. That is
\begin{equation*} \label{lab:sup}
  \Gamma_C M = \{ m \in M \; | \; \supp m \subset \Gr{I}{C}{J} \}.
\end{equation*}
Recall from Proposition \ref{prop:doublebruhat} that the Bruhat order on $W$ descends to a partial order on $\W{I}{J}$
and that, given $p \in \W{I}{J}$, we write $\{ \le p \}$ for the set of elements in $\W{I}{J}$ which are smaller than $p$ (and similarly for $\{ < p \}$, $\{ \ge p \}$ and $\{ > p \}$). We also abbreviate
\begin{equation*}
  \Gr{I}{\le p}{J} = \Gr{I}{\{ \le p\} }{J}\text{ and } \Gamma_{\le p} M = \Gamma_{\{ \le p \}} M
\end{equation*}
and analogously for $\Gr{I}{< p}{J}$, $\Gamma_{< p} M$, $\Gr{I}{\ge p}{J}$ etc.  The following additional notation will be useful:\label{lab:sup2}
\begin{align*}
  \Gamma^pM &= M / \Gamma_{\ne p} M \\
  \Gamma^{\le}_p M  & = \Gamma_{\le p} M /\Gamma_{< p} M \\
  \Gamma^{\ge}_p M  & = \Gamma_{\ge p} M /\Gamma_{> p} M. 
\end{align*}

Recall that in Subsection \ref{subsec:enlarge} we defined $R(X) \in \bMod{R}{R}$ for any finite subset $X  \subset W$.

\begin{lemma} \label{lem:Rpsupp}
The support of $f = (f_x) \in R(X)$ is $\Gr{}{C}{}$, where
  \begin{equation*}
    C = \{ x \in X \; | \; f_x \ne 0 \}.
  \end{equation*}
\end{lemma}

\begin{proof} Because we may identify $R_x$ as an $R \otimes R$-module with the regular functions on the irreducible $\Gr{}{x}{}$ it follows that every $0 \ne m \in R_x$ has support equal to $\Gr{}{x}{}$. The lemma than follows by considering the embedding of $R(X)$ in $\bigoplus_{x \in X} R_x$.\end{proof}

\begin{lemma} \label{lem:standsupp} Let $I, J \subset S$ be finitary and $p \in \W{I}{J}$. The support of any non-zero $m \in \RS{I}{p}{J}$ is $\Gr{I}{p}{J}$.  \end{lemma}

\begin{proof} This follows from (\ref{eq:supp3}), Lemma \ref{lem:Rpsupp} above and the fact that we may view $\RS{I}{p}{J}$ as an $(R^I, R^J)$-submodule of $R(p)$ (Lemma \ref{lem:Rpinv}). \end{proof}

\excise{

\begin{lemma} \label{lem:generalsupport} Let  $X$ be an affine variety
with a decomposition
 \begin{equation*}
   X = X_1 \cup X_2 \cup \dots X_n
 \end{equation*}
into closed subsets. Let $A$ be the ring of regular functions on $X$
and let $M$ be an $A$-module. Suppose we have a filtration
 \begin{equation*}
   0 = M_0 \subset M_1 \subset \cdots \subset M_n = M
 \end{equation*}
dof $M$ such that the support of any non-zero element in $M_i /M_{i-1}$
is contained in $X_i$, but not in $\cup_{i \ne j} X_j$. Then $M_i$ is
precisely the submodule of sections with support in $X_1  \cup \dots
\cup X_i$.
\end{lemma}

\begin{proof} Considering the exact sequence $M_{i-1} \hookrightarrow
M_i \sur M_i/M_{i-1}$ we see that $\supp M_i \subset X_i \cup \supp
M_{i-1}$ and by induction $\supp M_i \subset X_1 \cup \dots \cup M_i$.
On the other hand, let $b \in M$ and suppose that $b$ has support
contained in $X_1 \cup \dots \cup X_i$. If $b$ is in $M_j$, but  not
in $M_{j-1}$, then the image of $b$ in $M_i/M_{i-1}$ is non-zero and
hence the support of $b$ is not contained in  $\cup_{i \ne j} X_j$.
Hence $j \le i$. \end{proof}}

\section{Equivariant Schubert calculus}

\label{subsec:demRX}

The aim of this subsection is to define Demazure operators on $R(X)$ and use them to construct filtrations on $R(p)$ for finite double cosets $p \subset W$, as well as invariant subrings thereof. This discussion was influenced by \cite{KT}, where a similar situation is discussed.

Recall that in Section \ref{subsec:enlarge}  we defined, for all finite sets $X \subset W$ a bimodule $R(X) \in \bMod{R}{R}$. Moreover, given subgroups $W_1, W_2 \subset W$ such that $W_1X = X = XW_2$, the bimodule $R(X)$ carries commuting left $W_1$- and right $W_2$-actions.

\begin{deprop}\label{prop:demdef}
Let $X, W_1, W_2 \subset W$ be as above.

\begin{enumerate}
\item For all reflections $t \in W_1$ there exists an operator
$f \mapsto \partial_tf$ on $R(X)$, the \emph{left Demazure operator to $t$}, uniquely determined by
\begin{equation*}
  f - tf = 2h_t(\partial_tf) \quad \text{for all $f \in R(p)$.}
\end{equation*}
This is a morphism in $\bMod{R^t}{R}$.

\item For all reflections $t \in W_2$ there exists an operator $f \mapsto f\partial_t$ on $R(X)$, the \emph{right Demazure operator to $t$}, uniquely determined by
\begin{equation*}
  f - ft = (f \partial_t)2h_t \quad \text{for all $f \in R(p)$.}
\end{equation*}
This is a morphism in $\bMod{R}{R^t}$.
\end{enumerate}
\end{deprop}

\begin{proof} We first treat the case of the left Demazure operator. Uniqueness is clear as $R(X)$ is torsion free as a left $R$-module.
Rewriting the condition at $x \in p$ we see that, if $f \in R(X)$, $\partial_tf$ must be given by
\begin{equation*}
    (\partial_tf)_x = \frac{f_x - tf_{tx}}{2h_t}.
  \end{equation*}
A priori this defines an element of $\Quot R$. However, by definition of $R(X)$, $f_{x} - f_{tx}$ and hence $f_x - tf_{tx}$ lies in $(h_t)$. Thus $(\partial_tf)_x \in R$ for all $x \in X$.

It remains to see that $\partial_tf \in R(X)$. Because $f - tf \in R(X)$ it is clear that
\begin{equation*}
  (\partial_tf)_x - (\partial_tf)_{t^{\prime}x} \in (h_{t^{\prime}})
\end{equation*}
whenever $t^{\prime} \ne t$ and $x, t^{\prime}x \in X$. Writing out the definitions, on also sees that
\begin{equation*}
  (\partial_tf)_x - (\partial_tf)_{tx}
\end{equation*}
it $t$-anti-invariant, and hence $(\partial_tf)_x - (\partial_tf)_{tx} \in (h_t)$. It follows that $\partial_tf \in R(X)$ and hence the left Demazure operator to $t$ exists.

It is clear that the left Demazure operator for $t \in W_I$ commutes with multiplication on the left with a $t$-invariant function. For the right action of $r \in R$ on $f \in R(X)$ one has
\begin{equation*}
  (\partial_t(fr))_x = \frac{(fr)_x - t(fr)_{tx}}{2h_t} = \frac{f_x - f_{tx}}{2h_t}xr = ((\partial_tf)r)_x.
\end{equation*}
In particular, $f \mapsto \partial_tf$ is a morphism in $\bMod{R^t}{R}$ as claimed.

We now treat the case of the right Demazure operator for a reflection $t \in W_2$. The operator is clearly unique if it exists and $f\partial_t$ for $f \in R(X)$ must be given by
\begin{equation*}
  (f\partial_t)_x = \frac{f_x - f_{xt}}{2xh_t}.
\end{equation*}
Similarly to above one checks that $(f\partial_t)_x \in R$ for all $x \in X$ and then that $f\partial_t \in R(X)$, using the definition of $R(X)$ and (\ref{eq:hident}). It is then straighforward to see that $f \mapsto f\partial_t$ is a morphism in $\bMod{R}{R^t}$.
\end{proof}

Recall from Section \ref{sec:support} that the support of an element $f \in R(X)$ is easy to calculate: it is the set $\Gr{}{A}{}$ where $A = \{ x \in X \; | \; f_x \ne 0 \}$. The following lemma is then an immediate consequence of the definition of the Demazure operators.

\begin{lemma} \label{lem:updown}
  Let $f \in R(X)$ such that $\supp f \subset \Gr{}{A}{}$ for some $A \subset X$.
  \begin{enumerate}
  \item If $t \in W_I$ is a reflection then $\supp \partial_tf \subset \Gr{}{A \cup tA}{}$.
\item If $t \in W_J$ is a reflection then $\supp f \partial_t \subset \Gr{}{A \cup At}{}$.
  \end{enumerate}
\end{lemma}

For the rest of this section fix two finitary subsets $I, J \subset S$ as well as a double coset $p \in \W{I}{J}$.
We now come to the main theorem of this section, which purports the existence of certain special elements in $R(p)$.

\begin{theorem} \label{thm:doublecosetfunctions}
There exists $\phi_x \in R(p)$ for all $x \in p$, unique up to a scalar, such that
\begin{enumerate}
\item $\deg \phi_x = 2(\ell(p_+) - \ell(x))$,
\item $\supp \phi_x \subset \Gr{}{\le x}{}$ and $(\phi_x)_x \ne 0$.
\end{enumerate}
The set $\{ \phi_w \; | \; w \in p \}$ builds a homogeneous basis for $R(p)$ as a left or right $R$-module.
\end{theorem}

\begin{proof} Let us first assume that there exists $\phi_x \in R(p)$ for all $x \in p$ satisfying the conditions of the theorem. We will argue that they are then unique and form a basis for $R(p)$ as a left or right $R$-module.

Suppose that $f \in R(p)$ has support contained in $\Gr{}{A}{}$ for
some downwardly closed subset $A \subset p$ and choose $x \in A$
maximal. As $f_{tx} = 0$ for all $t \in T$ with $x < tx \in p$, from the definition of $R(p)$ we see that $f_x$ is divisible by
\begin{equation*}
  \alpha_x = \prod_{t \in T \atop x < tx \in p} h_t.
\end{equation*}
As $\deg \alpha_x = 2|\{ t \in T \; | \; x < tx \in p \} | = 2(\ell(p_+) - \ell(x))$ by Proposition \ref{prop:doublecosetdiff} we see that $(\phi_x)_x$ is a non-zero scalar multiple of $\alpha_x$. Hence, we may find $r \in R$ such that
\begin{equation*}
  \supp (f - r\phi_x) \subset \Gr{}{A \setminus \{x \}}{}.
\end{equation*}
It follows by induction that the $\{ \phi_x \}$ span $R(p)$ as a left $R$-module. They are clearly linearly independent when we consider $R(p)$ as a left $R$-module by the support conditions. Hence they form a basis for $R(p)$ as a left $R$-module. Identical arguments show that they are also a basis for $R(p)$ as a right $R$-module.

We can also use the above facts to see that $\phi_x$ for $x \in p$ is unique up to a scalar. Indeed, if $\phi_x$ and $\phi_x^{\prime}$ are two candiates we may find $\lambda \in k$ such that $\phi_x - \lambda\phi_x^{\prime}$ is supported on $\Gr{}{\le x \setminus \{ x \}}{}$. By the above arguments $\phi_x - \lambda\phi_x^{\prime}$ has degree strictly greater than $2(\ell(p_+) - \ell(x))$ and hence is zero.

It remains to show existence. To get started consider $\vartheta =  (\vartheta_x) \in \oplus_{x \in p} R$ defined by
\begin{equation*}
  \vartheta_x = \left \{ \begin{array}{ll} \alpha_{p_-} & \text{ if $x = p_-$} \\
0 & \text{otherwise.} \end{array} \right . 
\end{equation*}
Clearly $\vartheta \in R(p)$ and $\deg \vartheta = 2(\ell(p_+ - \ell(p_-))$ (again by Proposition \ref{prop:doublecosetdiff}). Hence we may set $\phi_{p_-} = \vartheta$.

Now assume by induction that we have found $\phi_x$ for all $x \in p$ with $\ell(x) < m$ and choose $y \in p$ of length $m$. By Howlett's theorem (\ref{Howlett}) there exists a simple reflection $s \in W_I$ or $t \in W_J$ such that either $y > sy \in p$ or $y > yt \in p$. In the first case consider $\vartheta = \partial_s\phi_{sy} \in R(p)$. We have
\begin{enumerate}
\item $\deg \vartheta = \deg \phi_{sy} - 2  = 2(\ell(p_+) - \ell(y))$,
\item $\supp \vartheta \subset \Gr{}{\le y}{}$ (by Lemma \ref{lem:updown}) and $\vartheta_y \ne 0$ because $(\phi_{sy})_{sy} \ne 0$.
\end{enumerate}
Hence we may set $\phi_y = \vartheta$. Similarly in the second case we may take $\phi_y = \phi_{yt}\partial_t$. It follows by  induction that the elements $\{ \phi_w \; | \; w  \in p\}$ exist. \end{proof}

The first corollary of this theorem is a description of $R(p)$ as a left $R$-module, needed during the proof of Theorem \ref{thm:ind}.

\begin{corollary} \label{cor:Rpleft}
As graded left $R$-modules we have an isomorphism
  \begin{equation*}
    R(p) \cong \tPoinc(p)  \cdot R.
  \end{equation*}
\end{corollary}

\begin{proof}
If $P = \sum_{x \in p} v^{2(\ell(x) - \ell(p_+))}$ it follows from the theorem that
\begin{equation*}
R(p) \cong P \cdot R \text{ in $\lMod{R}$}.
\end{equation*}
However
\begin{equation*}
P = 
\overline{ v^{2\ell(p_+)} \sum_{x \in p} v^{-2\ell(x)}}
= \overline{v^{\ell(p_+) - \ell(p_-)}\Poinc(p)} = v^{\ell(p_-) - \ell(p_+)}\Poinc(p) = \tPoinc(p)
\end{equation*}
using the self-duality of $\pi(p)$ (see (\ref{eq:poinc4})) for the third step.\end{proof}

\begin{corollary} \label{cor:indfilt}

Let $K \subset I$, $L \subset J$ and $C \subset \W{K}{L}$ be
downwardly closed. For all maximal $q \in C$ such that $q \subset p$ we have an isomorphism in $\bMod{R^K}{R^L}$:
  \begin{equation*}
    \Gamma_{C} R(p)^{W_K \times W_L} / \Gamma_{C \setminus \{ q \}} R(p)^{W_K \times W_L}  \cong \RS{K}{q}{L}[2(\ell(q_+) - \ell(p_+))]. 
  \end{equation*}
\end{corollary}

\begin{proof} For the course of the proof let us write $\phi_w^p$ (resp. $\phi_y^q$) for the functions in $R(p)$ (resp. $R(q)$) given to us by Theorem \ref{thm:doublecosetfunctions}. These are well defined up to a scalar and we make a fixed but arbitrary choice. Also denote by $\quo : W \to \W{K}{L}$ the quotient map.

The map $(f_x)_{x \in p} \mapsto (f_x)_{x \in q}$ from $R(p)$ to
$R(q)$, in which we forget $f_x$ for $x \notin q$, allows us to
identify  $\Gamma_{\quo^{-1}(C)} R(p) / \Gamma_{\quo^{-1}(C \setminus
  \{q\})} R(p)$ with an ideal in $R(q)$. Keeping this in mind we
obtain a map (of $R(q)$-modules):
\begin{align*}
  R(q)[2(\ell(q_+) - \ell(p_+))] & \to \Gamma_{\quo^{-1}(C)}R(p) / \Gamma_{\quo^{-1}(C \setminus \{q\})} R(p) \\
1 & \mapsto \phi_{q_+}^{p}.
\end{align*}
As $\partial_s \phi_{q_+}^p = \phi_{q_+}^p \partial_t = 0$ for all $s \in K$ and $t \in L$, $\phi_{q_+}^{p}$ is $W_K \times W_L$-invariant. Thus $(\phi_{q_+}^p)_x \ne 0$ for all $x \in q$, and the above map is injective.

Let us consider the image of $\phi_x^q \in R(p)$ for $x \in q$ in the right hand side. It has degree
\begin{equation*}
  \deg \phi_x^q + \deg \phi_{q_+}^p = 2(\ell(p_+) - \ell(x))
\end{equation*}
and has support contained in $\Gr{}{\le x}{}$. Hence, by the
uniqueness statement in Theorem \ref{thm:doublecosetfunctions}, it is
a non-zero scalar multiple of (the image of) $\phi_x^p$. It is a
consequence of Theorem \ref{thm:doublecosetfunctions} that $\phi_x^p$
for $x \in q$ build a  basis for the right hand side as a left
$R$-module, and we conclude that the map is an isomorphism.

The $W_K \times W_L$ action on $R(p)$ preserves both
$\Gamma_{\quo^{-1}(C)}R(p)$ and $\Gamma_{\quo^{-1}(C \setminus \{ q
  \})} R(p)$ and hence we have a $W_K \times W_L$-action on both
modules. As $W_K \times W_L$ acts  through $k$-algebra automorphisms
the above map commutes with the $W_K \times W_L$-action on both
modules. 

Hence we have an exact sequence of $W_K \times W_L$-modules:
\begin{equation} \label{eq:isitexact}
\Gamma_{\quo^{-1}(C \setminus \{q\})} R(p) \hookrightarrow
\Gamma_{\quo^{-1}(C)}R(p) \sur
R(q)[2(\ell(q_+) - \ell(p_+))]
\end{equation}
We claim that this sequence stays exact after taking $W_k \times
W_L$-invariants.

Write $q_+ = uq_-v$ with $u \in W_{K \cap q_-Lq_-^{-1}}$ and $v \in W_L$ (unique by Theorem \ref{Howlett}). If we choose reduced expressions $u = s_1 \dots s_m$ and $v = t_1 \dots t_n$ then we saw in the proof of Theorem \ref{thm:doublecosetfunctions} that we may assume
\[
\phi_{q_+}^p = (\partial_{s_1} \dots \partial_{s_k} \phi_{q_-}^p \partial_{t_1} \dots \partial_{t_k}).
\]
From the formulas for Demazure operators give in \ref{prop:demdef} it is clear that, for any $r \in R$,
\[
 (\partial_{s_1} \dots \partial_{s_k} (r\phi_{q_-}^p) \partial_{t_1} \dots \partial_{t_k})_{q_+} = (ur)\phi_{q^+}^p.
\]

Now let $f \in R(q)^{W_K \times W_L}[2(\ell(q_+) - \ell(p_+))]$. If we set $\widetilde{r} = u^{-1}f_{q_+}$ then
\[
\widetilde{f} := (\partial_{s_1} \dots \partial_{s_k} (\widetilde{r}\phi_{q_-}^p) \partial_{t_1} \dots \partial_{t_k})_{q_+} = (u\widetilde{r})\phi_{q^+}^p \in \Gamma_{\quo^{-1}(C)}R(p)^{W_K \times W_L}
\]
and $\widetilde{f}$ maps to $f$ under the above surjection (it is enough to check that they have the same value of $q_+$ by $W_K \times W_L$-invariance, and this follows by construction). Hence the above sequence stays exact when we consider $W_K \times W_L$-invariants.

We conclude that 
\begin{equation*}
  \RS{K}{p}{L}[2(\ell(q_+) - \ell(p_+))] \cong (\Gamma_{\pi^{-1}(C)} R(p))^{W_K \times W_L} / (\Gamma_{\pi^{-1}(C \setminus \{q\})} R(p))^{W_K \times W_L}.
\end{equation*}
However, by (\ref{eq:supp3}), $(\Gamma_{\quo^{-1}(C)}R(p))^{W_K \times W_L} = \Gamma_C (R(p)^{W_K \times W_L})$ and similarly for $\Gamma_{\quo^{-1}(C\setminus\{q\})} R(p)$. The claimed isomorphism then follows.
\end{proof}

In the sequel it will be useful to have the above corollary in a
slightly different form (which follows from Theorem \ref{thm:ind}):

\begin{corollary} \label{cor:indfilt2}
Let $J \supset K$ and $C \subset \W{I}{K}$ be downwardly closed. If $q \in  C$ is maximal and $p \supset q$ then we have an isomorphism in $\bMod{R^I}{R^K}$
  \begin{equation*}
    \Gamma_{C} (\RS{I}{p}{J} \otimes_{R^J} R^K) / \Gamma_{C\setminus \{ q \}} (\RS{I}{p}{J}\otimes_{R^J} R^K)
\cong \RS{I}{q}{K}[2(\ell(q_+) - \ell(p_+))].
  \end{equation*}
\end{corollary}

\section{Flags, characters and translation} \label{sec:flags}

In this section we define and study the categories of objects with
nabla and delta flags. These categories provide the first step in the
categorication of the Schur algebroid.

Recall from the introduction that to any $M \in \bMod{R^I}{R^J}$ one may associate two filtrations, and that $M$ has a nabla (resp. delta) flag if these filtrations are exhaustive and the successive quotients in the first (resp. second) filtration are isomorphic to a finite direct sum of shifts of standard modules. Given an object with a nabla or delta flag it is natural to consider its ``character'' in $\He{I}{J}$, which counts the graded multiplicity of standard modules these subquotients.

The key results of this section are Theorems \ref{thm:translation} and \ref{thm:translation2}, which tell us that if $J \subset K$ then the functors of restriction and extension of scalars between $\bMod{R^I}{R^J}$ and $\bMod{R^I}{R^K}$ restrict to functors between the corresponding categories of objects with nabla or delta flags. Moreover, after normalisation, one may describe the effect of these functors on the characters in terms of multiplication in the Hecke category.

The structure of this section is as follows. In Section \ref{subsec:nabla} we define the subcategory of modules with nabla flags and the nabla character, and begin the proof of Theorem \ref{thm:translation}. The proof involves certain technical splitting and vanishing statements, which we postpone to Section \ref{subsec:vansplit}. In Section \ref{subsec:delta} we define the subcategory of modules with delta flags and the delta character, as well as a duality which is used to relate the categories of object with delta and nabla flags and prove Theorem \ref{thm:translation2}.

\subsection{Objects with nabla flags and translation} \label{subsec:nabla}

For the duration of this section fix finitary subsets $I, J \subset S$. Denote by $\RR{I}{J}$ the full subcategory of modules $M \in \bMod{R^I}{R^J}$ such that:
\begin{enumerate}
\item $M$ is finitely generated, both as a left $R^I$-module, and as a right $R^J$-module;
\item there exists a finite subset $C \subset \W{I}{J}$ such that $\supp M \subset \Gr{I}{C}{J}$.
\end{enumerate}
Recall that we call a subset $C \subset \W{I}{J}$ downwardly closed if
\begin{equation*}
  C = \{ p \in \W{I}{J} \; | \; p \le q \text{ for some } q \in C\}.
\end{equation*}
We now come to the definition of objects with nabla flags.

\begin{definition}
The category of \emph{objects with nabla flags}, denoted $\nabflag{I}{J}$, is the full subcategory of modules $M \in \RR{I}{J}$ such that, for all downwardly closed subsets $C \subset \W{I}{J}$ and maximal elements $p \in C$, the subquotient
  \begin{displaymath}
    \Gamma_{C}M / \Gamma_{C \setminus \{ p \} }M
  \end{displaymath}
is isomorphic to a direct sum of shifts of modules of the form $\RS{I}{p}{J}$ (which is necessarily finite because $M \in \RR{I}{J}$).
\end{definition}

We begin with a lemma that simplifies the task of checking whether a module $M \in \RR{I}{J}$ belongs to $\nabflag{I}{J}$. We call an enumeration $p_1, p_2, \dots$ of the elements of $\W{I}{J}$ a \emph{refinement of the Bruhat order} if $p_i \le p_j$ implies that $i \le j$. If we let $C(m) = \{ p_1, p_2, \dots, p_m \}$ then all the sets $C(m)$ are downwardly closed, and $p_m \in C(m)$ is maximal. Hence, if $M \in \nabflag{I}{J}$ then $\Gamma_{C(m)} M / \Gamma_{C(m-1)} M$  is isomorphic to a direct sum of shifts of $\RS{I}{p_m}{J}$. In fact, the converse is true:

\begin{lemma}[Soergel's ``hin-und-her'' lemma] \label{lem:hinundher}
  Let $p_1, p_2, \dots$ and $C(m)$ be as above. Suppose $M \in \RR{I}{J}$ is such that, for all $m$, the subquotient
  \begin{equation*}
    \Gamma_{C(m)} M /\Gamma_{C(m-1)} M
  \end{equation*}
is isomorphic to a direct sum of shifts of $\RS{I}{p_m}{J}$. Then $M  \in \nabflag{I}{J}$.

Moreover, if $p = p_m$ then the natural map
  \begin{displaymath}
    \Gamma_{\le p} M / \Gamma_{< p} M \to \Gamma_{C(m)} M / \Gamma_{C(m-1)}M
  \end{displaymath}
is an isomorphism.
\end{lemma}

\begin{proof} Let $C \subset \W{I}{J}$ be a downwardly closed subset and $p \in C$ be maximal. We need to show that 
  \begin{displaymath}
    \Gamma_{C}M / \Gamma_{C \setminus \{ p \} }M
  \end{displaymath}
is isomorphic to a direct sum of shifts of modules of the form $\RS{I}{p}{J}$.

Let $p, p^{\prime} \in \W{I}{J}$ be incomparable in the Bruhat order. We will see in the next section (Lemma \ref{lem:vanish}) that $\Ext_{R^I\otimes R^J}^1(\RS{I}{p}{J}, \RS{I}{p^{\prime}}{J}) = 0$. In particular, if $p_i$ and $p_{i+1}$ are incomparable in the Bruhat order then $\Gamma_{C(i+1)} M / \Gamma_{C(i-1)} M$ is isomorphic to a direct sum of shifts of modules $\RS{I}{p_i}{J}$ and $\RS{I}{p_{i+1}}{J}$. Hence, if we let $C^{\prime}$ be associated to the sequence obtained by swapping two elements $q_{i}$ and $q_{i+1}$ we see that the natural maps
  \begin{eqnarray*}
    \Gamma_{C(i)}M / \Gamma_{C(i-1)}M & \to& \Gamma_{C^{\prime}(i+1)}M / \Gamma_{C^{\prime}(i)}M \\
\Gamma_{C^{\prime}(i)}M / \Gamma_{C^{\prime}(i-1)}M & \to & \Gamma_{C(i+1)}M / \Gamma_{C(i)}M   \end{eqnarray*}
are isomorphisms.  
 
Now let $C \subset \W{I}{J}$ be downwardly closed and $p \in C$ maximal. After swapping finitely many many elements of our sequence we may assume $C(m) = C$ and $p_m = p$ and the first statement follows. The second statement follows by taking $C = \{ \le p \}$.\end{proof}

We now want to define the ``character'' of an object $M \in \nabflag{I}{J}$.  It is natural to renormalise $\RS{I}{p}{J}$ and define \label{lab:nab}
\begin{displaymath}
  \nab{I}{p}{J} = \RS{I}{p}{J}[\ell(p_+)].
\end{displaymath}
If $p$ contains the identity, we sometimes omit $p$ and write $\nab{I}{}{J}$.

By assumption, if $M \in \nabflag{I}{J}$ we may find polynomials $g_{p}(M) \in \N[v,v^{-1}]$ such that, for all $p \in \W{I}{J}$ we have
\begin{displaymath}
  \Gamma_{\le p}M / \Gamma_{< p}M \cong g_{p}(M) \cdot \nab{I}{p}{J}.
\end{displaymath}

We now define the \emph{nabla character} by
\begin{eqnarray*}
\nabchar : \nabflag{I}{J} &\to& \He{I}{J} \\
M & \mapsto & \sum_{p \in \W{I}{J}}  \overline{g_{p}(M)} \mst{I}{p}{J}.
\end{eqnarray*}

We now come to the definition of translation functors, which (up to a shift) are the functors of extension and restriction of scalars.

\begin{definition} Let $K  \subset S$ be finitary.
\begin{enumerate}
\item If $J \subset K$ the functor of ``translating onto the wall'' is:
  \begin{eqnarray*}
    - \cdot \on{J}{K} : \bMod{R^I}{R^J} &\to& \bMod{R^I}{R^K} \\
M &\mapsto & M_{R^K}[\ell(\wo{K}) - \ell(\wo{J})].
  \end{eqnarray*}
\item If $J \supset K$ the functor of  ``translating out of the wall'' is:
  \begin{eqnarray*}
    - \cdot \out{J}{K} : \bMod{R^I}{R^J} &\to& \bMod{R^I}{R^K} \\
M &\mapsto & M\otimes_{R^J}R^K.
  \end{eqnarray*}
\end{enumerate}
\end{definition}

\begin{remark} Of course it is also possible to define translation functors ``on the left''. We have chosen to only define and work with translation functors acting on one side because it simplifies the exposition considerably. \end{remark}

The following theorem is fundamental to all that follows. It shows that translation functors preserve the categories of objects with nabla flags and that we may describe the effect of translation functors on characters.

\begin{theorem} \label{thm:translation} Let $K \subset S$ be finitary with $J \subset K$ or $K \subset J$.
\begin{enumerate}
\item If $M \in \nabflag{I}{J}$ then $M \cdot \on{J}{K} \in \nabflag{I}{K}$.
\item The following diagrams commute:
\begin{displaymath}
\xymatrix@C=1.5cm{
\nabflag{I}{J} \ar[d]^{\ch_{\nabla}} \ar[r]^{ - \cdot \on{J}{K}}& \nabflag{I}{K} \ar[d]^{\ch_{\nabla}} &
\nabflag{I}{J} \ar[d]^{\ch_{\nabla}} \ar[r]^{[1]}& \nabflag{I}{J} \ar[d]^{\ch_{\nabla}} \\
\He{I}{J} \ar[r]^{- *_J \mst{J}{}{K} } & \He{I}{K} &
\He{I}{J} \ar[r]^{\cdot v^{-1}} & \He{I}{J}
}
\end{displaymath}
\end{enumerate}
\end{theorem}

Before we can prove this we will need a preparatory result.

\begin{proposition} \label{prop:transsupp}
Let $J \subset K$ be finitary and 
  \begin{equation*}
    \quo : \W{I}{J} \to \W{I}{K}
  \end{equation*}
 be the quotient map. Let $C \subset \W{I}{K}$ be downwardly closed.
  \begin{enumerate}
  \item If $M \in \nabflag{I}{J}$  then
    \begin{equation*}
      (\Gamma_{\quo^{-1}(C)} M ) _{R^{K}} =  \Gamma_C( M_{R^K} ).
      \end{equation*}
\item If $M \in \nabflag{I}{K}$ then
  \begin{equation*}
    ( \Gamma_C M) \otimes_{R^K} R^J  = \Gamma_{\quo^{-1}(C)} (M \otimes_{R^K} R^J).
  \end{equation*}
  \end{enumerate}
\end{proposition}

\begin{proof} (1) is a direct consequence of (\ref{eq:supp3}). For (2) consider the exact sequence
\begin{equation*}
\Gamma_C M \hookrightarrow M \sur M / \Gamma_C M.
\end{equation*}
Because $M \in \nabflag{I}{J}$ the left (resp. right) module has a filtration with subquotients isomorphic to a direct sum of shifts of $\RS{I}{p}{J}$ with $p \in C$ (resp. $p \notin C$). Applying the exact functor $- \otimes_{R^K} R^J$ we obtain an exact sequence
\begin{equation*}
\Gamma_C M\otimes_{R^K} R^J \hookrightarrow M\otimes_{R^K} R^J\sur M / \Gamma_C M\otimes_{R^K} R^J.
\end{equation*}
By exactness, the left (resp. right) modules have a filtration with subquotients a direct sum of shifts of $\RS{I}{p}{J} \otimes_{R^K} R^J $ with $p \in C$ (resp. $p \notin C$). By Corollary \ref{cor:indfilt2},  $\RS{I}{p}{J} \otimes_{R^K} R^J $ has a filtration with subquotients isomorphic to (a shift of) $\RS{I}{q}{J}$ with $q \in \quo^{-1}(p)$. Moreover the support of any non-zero element in $\RS{I}{q}{J}$ is precisely $\Gr{I}{q}{J}$ (Lemma \ref{lem:standsupp}). Thus the above exact sequence is equal to
\begin{equation*}
\Gamma_{\quo^{-1}(C)} (M\otimes_{R^K} R^J) \hookrightarrow M\otimes_{R^K} R^J\sur M / \Gamma_{\quo^{-1}(C)} (M\otimes_{R^K} R^J)
\end{equation*} 
which implies the proposition.
\end{proof}

We can now prove the Theorem \ref{thm:translation}.

\begin{proof}[Proof of Theorem \ref{thm:translation}]
It is easy to see that $M \cdot \out{J}{K} \in \RR{I}{K}$ using Lemma \ref{lem:support} and the fact that $R^J$ is finite over $R^K$ in the case that $J \supset K$. We split the proof into two cases.

\emph{Case 1: Translating out of the wall ($J \supset K$):} We first prove part (1) of the theorem. Let
\begin{equation*}
\quo : \W{I}{K} \to \W{I}{J}
\end{equation*}
be the quotient map. Because $\quo$ is a surjective morphism of posets we may choose an enumeration $p_1, p_2, \dots$ of the elements of $\W{I}{K}$ refining the Bruhat order such that, after deleting repetitions, $\quo(p_1)$, $\quo(p_2)$, $\dots$ is a listing of the elements of $\W{I}{J}$ refining the Bruhat order. Fix $q \in \W{I}{J}$ and $p = p_m \in \quo^{-1}(q)$ and define
\begin{equation*}
C(n) = \{ p_1, p_2, \dots, p_n \}.
\end{equation*}
By the hin-und-her lemma (\ref{lem:hinundher}) it is enough to show that
\begin{equation*}
\Gamma_{C(m)} (M \otimes_{R^J} R^K) / \Gamma_{C(m-1)} (M \otimes_{R^J} R^K) 
\end{equation*}
is isomorphic to a direct sum of shifts of $\RS{I}{p}{K}$.

The set $F = \quo(C(m))$ is downwardly closed and contains $q$ as a maximal element. As $M \in \nabflag{I}{J}$ there exists an exact sequence
\begin{equation*}
\Gamma_{F \setminus \{ q \} } M \hookrightarrow \Gamma_F M \sur P \cdot \RS{I}{q}{J}
\end{equation*}
for some $P \in \N[v,v^{-1}]$. Applying $- \otimes_{R^J} R^K$ and using Proposition \ref{prop:transsupp} we conclude an exact sequence
\begin{equation*}
\Gamma_{\quo^{-1}(F \setminus \{ q \}) } (M \otimes_{R^J} R^K) \hookrightarrow \Gamma_{\quo^{-1} (F)} (M \otimes_{R^J} R^K) \sur P \cdot \RS{I}{q}{J}\otimes_{R^J} R^K
\end{equation*}
As $\Gamma_{\quo^{-1}(F \setminus \{ q \}) } (M \otimes_{R^J} R^K) $ is contained in both $\Gamma_{C(m)} (M \otimes_{R^J} R^K)$ and $\Gamma_{C(m-1)} (M \otimes_{R^J} R^K)$ by the third isomorphism theorem we will be finished if we can show that
\begin{equation*}
\Gamma_{C(m)} (\RS{I}{q}{J} \otimes_{R^J} R^K) / 
\Gamma_{C(m-1)} (\RS{I}{q}{J} \otimes_{R^J} R^K) 
\end{equation*}
is isomorphic to a direct sum of shifts of $\RS{I}{p}{J}$. But this is precisely the statement of Corollary \ref{cor:indfilt2}. Hence $M \cdot \out{J}{K} \in \nabflag{I}{K}$.

We now prove (2). The  commutativity of the right hand diagram is clear. As $- \cdot \out{J}{K}$ is exact and  every element in $\nabflag{I}{J}$ is an extension of the nabla modules we only have to check the commutativity of the left hand diagram for a nabla module. That is, we have to verify that
\begin{equation*}
  \ch_{\nabla}(\nab{I}{q}{J}) *_J \mst{J}{}{K} = \ch_{\nabla}( \nab{I}{q}{J} \cdot \out{J}{K}).
\end{equation*}
By Proposition \ref{prop:multform} the left hand side is equal to
\begin{equation*}
  \mst{I}{q}{J} *_J \mst{J}{}{K} =  \sum_{p \in W_I \! \setminus q /W_J} v^{\ell(q_+) - \ell(p_+)}\mst{I}{p}{K}.
\end{equation*}
For the right hand side note that:
\begin{align*}
\Gamma_{\le p} (\nab{I}{q}{J} & \otimes_{R^J} R^K) / \Gamma_{< p} (\nab{I}{q}{J} \otimes_{R^J} R^K) \cong \\
& \cong \Gamma_{\le p} (\RS{I}{q}{J} \otimes_{R^J} R^K) / \Gamma_{< p} (\RS{I}{q}{J} \otimes_{R^J} R^K)[\ell(q_+)] \\
& \cong \RS{I}{p}{K}[2\ell(p_+) - \ell(q_+)] & \text{(Corollary \ref{cor:indfilt2})}\\
& \cong v^{\ell(p_+) - \ell(q_+)}  \cdot \nab{I}{p}{K} \end{align*}
Therefore, by definition of $\ch_{\nabla}$,
\begin{equation*}
  \ch_{\nabla}( \nab{I}{q}{J} \cdot \out{J}{K}) = \sum_{p \in W_I \! \setminus q /W_J} v^{\ell(q_+) - \ell(p_+)}\mst{I}{p}{K}.
\end{equation*}
This completes the proof in case $J \supset K$.

\emph{Case 2: Translating onto the wall ($J \subset K$):}
Denote (as usual) by $\quo$ the quotient map
\begin{equation*}
\quo : \W{I}{J} \to \W{I}{K}.
\end{equation*}
Let $C \subset \W{I}{K}$ be downwardly closed and choose $q \in C$ maximal. Consider the exact sequence
\begin{equation*}
\Gamma_{\quo^{-1}(C\setminus \{ q \} )} M \hookrightarrow
\Gamma_{\quo^{-1}(C)} M \sur
\Gamma_{\quo^{-1}(C)} M /
\Gamma_{\quo^{-1}(C\setminus \{ q \} )} M.
\end{equation*}
As $M \in \nabflag{I}{J}$ the right-hand module has a filtration with subquotients isomorphic to direct sums of shifts $\RS{I}{p}{J}$ with $p \in \quo^{-1}(q)$. In Proposition \ref{prop:split} in the next subsection we will see that any such module splits as a direct sum of shifts of $\RS{I}{q}{K}$ upon restriction to $R^K$. This implies that $M_{R^K} \in \nabflag{I}{K}$ because, by Proposition \ref{prop:transsupp}, the restriction to $\bMod{R^I}{R^K}$ of the above exact sequence is identical to
\begin{equation*}
\Gamma_{C\setminus \{ q \}} (M_{R^K}) \hookrightarrow 
\Gamma_{C} (M_{R^K}) \sur 
\Gamma_{C} (M_{R^K}) /  
\Gamma_{C\setminus \{ q \}} (M_{R^K}).
\end{equation*}

We now turn our attention to (2). As above, it is enough to check the commutativity of the left-hand diagram for a nabla module. Let $p \in \W{I}{J}$ and $q = \quo(p)$. We need to check that
\begin{equation*}
\ch_{\nabla}(\nab{I}{p}{J}) *_J \mst{J}{}{K} = v^{\ell(q_-) - \ell(p_-)}
\frac{\pi(I,q,K)}{\pi(I,p,J)} \mst{I}{q}{K} = \ch_{\nabla} (\nab{I}{p}{J} \cdot \on{J}{K})
\end{equation*}
where the first equality follows from Proposition \ref{prop:multform}. By definition of $\ch_{\nabla}$ this follows from the isomorphism
\begin{equation*}
\nab{I}{p}{J}\cdot \on{J}{K} \cong v^{\ell(p_-) - \ell(q_-)}
\frac{\pi(I,q,K)}{\pi(I,p,J)} \cdot \nab{I}{q}{J}
\end{equation*}
which we prove in Lemma \ref{lem:nablarestrict} below.\end{proof}

\begin{lemma}\label{lem:nablarestrict}
Let $J \subset K$, $p \in \W{I}{J}$ and $q = W_IpW_K$. We have an isomorphism
  \begin{displaymath} 
    \nab{I}{p}{J} \cdot \on{J}{K} \cong v^{\ell(p_-) - \ell(q_-)}\frac{\Poinc(I,q,K)}{\Poinc(I,p,J)} \cdot \nab{I}{q}{K}.
  \end{displaymath}
\end{lemma}

\begin{proof} By Lemma \ref{cor:res} we have
\begin{align*}
(\nab{I}{p}{J}) \cdot \out{J}{K} &
\cong (\RS{I}{p}{J})_{R^K} [\ell(p_+) + \ell(\wo{K}) - \ell(\wo{J})] \\
& \cong \frac{\tPoinc(I,q,K)}{\tPoinc(I,p,J)} \cdot \RS{I}{q}{J}[\ell(p_+) + \ell(\wo{K}) - \ell(\wo{J})] \\
& \cong v^a \frac{\Poinc(I,q,K)}{\Poinc(I,p,J)} \cdot \nab{I}{q}{J}
\end{align*}
where
\begin{align*}
a & = \ell(\wo{I,p,J}) - \ell(\wo{I,q,K}) + \ell(p_+) - \ell(q_+) + \ell(\wo{K}) - \ell(\wo{J}) \\
& = (\ell(p_+) - \ell(\wo{I}) - \ell(\wo{J}) + \ell(\wo{I,p,J}) ) - \\
& \qquad (\ell(q_+) - \ell(\wo{I}) - \ell(\wo{K}) + \ell(\wo{I,q,K}) ) \\
& = \ell(p_-) - \ell(q_-)
\end{align*}
by \eqref{eq:poinc1}.\end{proof}

\subsection{Vanishing and splitting} \label{subsec:vansplit}

This is a technical section in which we prove two vanishing statements
which were postponed in the last section.

Let us begin with some generalities. Let $A$ be a ring. An extension
between two $A$-modules
\begin{equation*}
 \label{eq:extexample}
 M \to E \to N
\end{equation*}
gives an element of $\Ext^1_A(N,M)$ by considering the long exact
sequence associated to $\Hom(-,M)$ and looking at the image of $\id_M$
in $ \Ext^1(N,M)$; the sequence splits if and only if this class is
zero.

Now let $A^{\prime} \to A$ be a homomorphism of rings. If $M$ and $N$
are $A$-modules one has maps
\begin{displaymath}
 r_m: \Ext^m_A(N,M) \to \Ext^m_{A^{\prime}}(N,M).
\end{displaymath}
We will need the following facts:
\begin{enumerate}
\item An extension between $M$ and $N$ splits upon restriction to
$A^{\prime}$ if and only its class lies in the kernel of the map
 \begin{displaymath}
  r_1 : \Ext^1_A(N,M) \to \Ext^1_{A^{\prime}}(N,M).
 \end{displaymath}
\item A short exact sequence $M^{\prime} \hookrightarrow M \sur
M^{\prime\prime}$
yields a commutative diagram of long exact sequences:
 \begin{equation} \label{eq:long1}
  \xymatrix@R=0.3cm@C=0.4cm{
 \ar[r] & \Ext_A^{1}(M^{\prime\prime},N) \ar[r] \ar[d] &
\Ext^1_A(M, N) \ar[r] \ar[d] & \Ext^1_A(M^{\prime}, N) \ar[r]
\ar[d] & \\
\ar[r] & \Ext_{A^{\prime}}^{1}(M^{\prime\prime},N) \ar[r] &
\Ext^1_{A^{\prime}}(M, N) \ar[r] &
\Ext^1_{A^{\prime}}(M^{\prime}, N) \ar[r] & }
 \end{equation}
\item Similarly, if $N^{\prime} \hookrightarrow N \sur
N^{\prime\prime}$ is a short
exact sequence, we obtain a commutative diagram of long exact
sequences:
 \begin{equation}
  \label{eq:long2}
  \xymatrix@R=0.3cm@C=0.4cm{
  \ar[r] & \Ext_A^{1}(M,N^{\prime}) \ar[r] \ar[d] &
\Ext^1_A(M, N) \ar[r] \ar[d] & \Ext^1_A(M, N^{\prime\prime}) \ar[r] \ar[d] &
\\
\ar[r] & \Ext_{A^{\prime}}^{1}(M,N^{\prime}) \ar[r] &
\Ext^1_{A^{\prime}}(M, N) \ar[r] & \Ext^1_{A^{\prime}}(M, N^{\prime\prime}) \ar[r] & }
 \end{equation}
\end{enumerate}
These facts become transparent when interpreted in the derived category (see e.g. \cite{Wie}).

Given a vector
space $W$, denote by $\SO(W)$ its graded ring of regular functions.

\begin{lemma} (Lemma 5.8 in \cite{SoBimodules})
 \label{lem:ext}
Let $W$ be a finite dimensional vector space and $U, V \subset W$ two
linear subspaces. Then
\begin{displaymath}
 \Ext^1_{\SO(W)}( \SO(U), \SO(V))
\end{displaymath}
is only non-trivial if $V \cap U$ is $V$ or a hyperplane in $V$. In
the later case it is generated by the class of any short exact
sequence of the form
\begin{displaymath}
 \SO(V)[-2] \stackrel{\alpha \cdot }{\hookrightarrow} \SO(V \cup U) \sur \SO(U)
\end{displaymath}
with $\alpha \in W^*$ a linear form satisfying $\alpha|_U = 0$ and
$\alpha|_V \ne 0$.
\end{lemma}

We now turn to our situation, with the goal of analysing extensions between standard modules. Notationally it proves more convenient to work with left modules, which we may do using the equivalences $\bMod{A_1}{A_2} \cong \lMod{A_1\otimes A_2}$ as all our rings are assumed commutative. We will do this for the rest of ths subsection without further comment.

Using the identification of $R_x$ with $\SO(\Gr{}{x}{})$ and Lemma \ref{lem:ext} we see that $\Ext^1_{R\otimes R}(R_x, R_y)$ is non-zero only when $\Gr{}{x}{}$ and $\Gr{}{y}{}$ intersect in codimension 1. As
\begin{equation*}
 \Gr{}{x}{} \cap \Gr{}{y}{} \cong V^{x^{-1}y}
\end{equation*}
and the representation of $W$ on $V$ is reflection faithful, this
occurs only when $y = xt$ for some reflection $t \in T$. We conclude
that there are no extensions between $R_x$ and $R_y$ unless $x \ne yt$
for some reflection $t \in W$.

Now let $p, p^{\prime} \in \W{I}{J}$ and suppose we have
an extension of the form
\begin{equation*}
 \RS{I}{p}{J} \hookrightarrow E \sur \RS{I}{p^{\prime}}{J}.
\end{equation*}
we may extend scalars to obtain an exact sequence
\begin{displaymath}
 R \otimes_{R^I} \RS{I}{p}{J} \otimes_{R^J} R \hookrightarrow
\tilde{E} \sur R \otimes_{R^I}\RS{I}{p^{\prime}}{J}\otimes_{R^J} R.
\end{displaymath}
If we again restrict to ${R^I}\otimes{R^J}$ we obtain a number of copies
of our original extension. By Theorem \ref{thm:ind} we have an
isomorphism
\begin{displaymath}
 R \otimes_{R^I} \RS{I}{p}{J} \otimes_{R^J} R \cong R(p).
\end{displaymath}
Therefore our  extension takes the form
\begin{displaymath}
 R(p) \hookrightarrow \tilde{E} \sur R(p^{\prime}).
\end{displaymath}
\excise{
Also note that, by Proposition \ref{cor:indfilt} for all $z \in p$ we
have an isomorphism
\begin{equation*}
 \Gamma_{\le z} R(p) / \Gamma_{< z} R(p) \cong R_z[\nu_z]
\end{equation*}
for some $\nu_x \in \Z$.}

\begin{lemma} \label{lem:vanish}
Suppose that $p, p^{\prime} \in \W{I}{J}$ are not
comparable in the Bruhat order. Then
\begin{equation*}
\Ext^1_{R^I\otimes R^J}(\RS{I}{p}{J}, \RS{I}{p^{\prime}}{J}) = 0.
\end{equation*}
\end{lemma}

\begin{proof} By the above discussion it is enough to show that there are no extensions between $R(p)$ and $R(p^{\prime})$. As $p$ and $p^{\prime}$ are incomparable, there are no pairs $x \in p$ and $x^{\prime} \in p^{\prime}$ with $x^{\prime} = xt$ for some $t \in T$. Thus (again by the above discussion), $\Ext^1_{R\otimes R}(R_x, R_{x^{\prime}})$ for all $x \in p$, $x^{\prime} \in p^{\prime}$. By Corollary \ref{cor:indfilt}, $R(p)$ (resp. $R(p^{\prime})$) has a filtration with successive subquotients $R_x$ for $x \in p$ (resp. $x \in p^{\prime}$). By induction and the long exact sequence of $\Ext$ it follows first that $\Ext^1_{R \otimes R}(R(p), R_{x^{\prime}}) = 0$ for all $x^{\prime} \in p^{\prime}$, and then that $\Ext^1_{R\otimes R}(R(p), R(p^\prime)) = 0$.
\end{proof}

Our goal for the rest of this section is to prove Proposition \ref{prop:split} below.
We start with two preparatory lemmas.

\begin{lemma}\label{lem:simplesplit} If $x \in W$ and $t \in T$ then the map
\begin{equation*}
r_1 : \Ext^1_{R \otimes R} (R_x, R_{xt}) \to \Ext^1_{R
\otimes R^t} (R_x, R_{xt})
\end{equation*}
induced by the inclusion $R \otimes R^t
\hookrightarrow {R \otimes R}$ is zero.\end{lemma}

\begin{proof} Given $c \in R\otimes R$ of degree 2, vanishing on
$\Gr{}{x}{}$ but not on $\Gr{}{xt}{}$ we obtain an extension
\begin{equation} \label{eq:monkeyext}
R_{xt}[-2] \stackrel{\cdot c}{\hookrightarrow} R_{x,xt} \sur R_x.
\end{equation}
By Lemma \ref{lem:ext}, it is enough to show that
(\ref{eq:monkeyext}) splits upon restriction to $R \otimes R^{t}$.
Consider the map $R_{x,xt} \to R_{xt}[-2]$ sending $f$ to the image of
$f\partial_t$, where $\partial_t$ is the (right) Demazure operator
introduced in Section \ref{subsec:demRX}. This is a morphism of $R\otimes
R^t$-modules. As $c$ vanishes on $\Gr{}{x}{}$ but not on
$\Gr{}{xt}{}$, $c\partial_t$ is non-zero, hence is a non-zero scalar for degree reasons. Thus a suitable scalar
multiple of this map provides a splitting of (\ref{eq:monkeyext}) over
$R \otimes R^t$.
\end{proof}

\begin{lemma} \label{lem:standardsplit}
Let $I, J \subset K$ be finitary subsets of $S$ and $p, p^{\prime} \in \W{I}{J}$ be such that $p \ne p^{\prime}$ but $W_IpW_K = W_I
p^{\prime} W_K$. Then every extension between $\RS{I}{p}{J}$ and
$\RS{I}{p^{\prime}}{J}$ splits upon restriction restriction to $R^I
\otimes R^K$. \end{lemma}

\begin{proof} Note that by the above discussion it is enough to show
that every extension between $R(p)$ and $R(p^{\prime})$ splits upon
restriction to $R^I \otimes R^K$. First note that if $x \in
p^{\prime}$ and $y \in p$ with $x = yt$ for some reflection $t \in W$,
then either $t \in W_K$ or $x = t^{\prime}y$ for some $t^{\prime} \in
W_I$ by Proposition \ref{prop:doublebruhat}. The second possibility is
impossible however, as $p \ne p^{\prime}$. We conclude, using the
previous lemma, that if $x \in p^{\prime}$ and $y \in p$ then either
$\Ext_{R \otimes R} (R_x, R_y) = 0$ or the map $\Ext_{R \otimes R}
(R_x, R_y) \to \Ext_{R^I \otimes R^K}(R_x, R_y)$ is zero.

We now proceed similarly to as in the proof of Lemma \ref{lem:vanish}.
Inducting over a filtration on $R(p)$ and using (\ref{eq:long2}) we conclude that the map
\begin{equation*}
\Ext_{R \otimes R}^1(R_x, R(p)) \to \Ext_{R^I \otimes R^K}^1(R_x, R(p))
\end{equation*}
induced by the inclusion $R^I \otimes R^K \hookrightarrow R \otimes R$ is zero for all $x \in p^{\prime}$. Inducting again using
(\ref{eq:long1}) we see that the map $\Ext_{R \otimes
R}^1(R(p^{\prime}), R(p)) \to \Ext_{R^I \otimes R^K}^1(R(p^{\prime}),
R(p))$ is zero, which establishes the lemma.
\end{proof}

\begin{proposition} \label{prop:split}
Let $I,J\subset K$  be finitary subsets of $S$ and let $q \in \W{I}{K}$. Let $B \in
\nabflag{I}{J}$ and suppose that $\supp B \subset \Gr{I}{C}{J}$ for
some $C \subset W_I \! \setminus \! q / W_J$. Then the restriction $B_{R^K} \in
\bMod{R^I}{R^K}$ is isomorphic to a direct sum of shifts of standard modules
$\RS{I}{q}{K}$. \end{proposition}

\begin{proof}Choose $p \in C$ maximal in the Bruhat order. As $B \in \nabflag{I}{J}$ we have an exact sequence
\begin{equation} \label{eq:Bseq}
\Gamma_{C \setminus \{ p \} }B \hookrightarrow B \sur P \cdot \RS{I}{p}{J}
\end{equation}
for some $P \in \N[v,v^{-1}]$. As $\Gamma_{C \setminus \{ p \} }B \in \nabflag{I}{J}$ we may induct over a suitable filtration of $\Gamma_{C \setminus \{ p \} }B $ and conclude, with the help of Lemma \ref{lem:standardsplit}, that (\ref{eq:Bseq}) splits upon restriction to $R^I \otimes R^K$.

Now let us choose a listing $p_1, p_2, \dots  p_n$ of the elements of $C$ refining the Bruhat order and let $C(m) = \{ p_1, p_2, \dots , p_m \}$ denote the first $m$ elements as usual.
Using downward induction and the above argument it follows that, in $\bMod{R^I}{R^K}$, we have an isomorphism
\begin{equation*}
B_{R^K} \cong \bigoplus (\Gamma_{C(m)} B / \Gamma_{C(m-1)}B)_{R^K}.
\end{equation*}
The proposition then follows as $(\RS{I}{p}{J})_{R^K}$ is isomorphic to
a direct sum of shifts of $\RS{I}{q}{K}$ where $q = pW_K$ by Corollary
\ref{cor:res}.
\end{proof}

\subsection{Delta flags and duality} \label{subsec:delta}

In this section we define a category of objects with delta flags,
$\delflag{I}{J}$, which is ``dual'' to $\nabflag{I}{J}$. Just as in
the case of objects with nabla flags the translation functors
preserve $\delflag{I}{J}$ and their effect on a ``delta character''
\begin{equation*}
\ch_{\Delta} : \delflag{I}{J} \to \He{I}{J}
\end{equation*}
can be described in terms of the Hecke category.

Of course it would be possible to repeat the same arguments as those
used for objects with nabla flags. However, one may define a
duality
\begin{equation*}
D : \nabflag{I}{J} \stackrel{\sim}{\to} {\delflag{I}{J}}^{opp}
\end{equation*}
commuting with the translation functors. This allows us to use what we
already know about objects with nabla flags to follow similar
statements for objects with delta flags.

For the rest of this section fix a pair $I, J \subset S$ of finitary subsets.
Recall that we call a subset $U \subset \W{I}{J}$ upwardly closed if
\begin{equation*}
U = \{ p \in \W{I}{J} \; | \; p \ge q \text{ for some } q \in C \}.
\end{equation*}

\begin{definition}
The category of \emph{objects with $\Delta$-flags}, denoted
$\delflag{I}{J}$ is the full subcategory of $\RR{I}{J}$ whose objects
are modules $M \in \RR{I}{J}$ such that, for all upwardly closed subsets $U
\subset \W{I}{J}$ and minimal elements $p \in U$, the
subquotient
\begin{equation*}
 \Gamma_{U}M / \Gamma_{U \setminus \{ p \} }M
\end{equation*}
is isomorphic to a direct sum of shifts of $\RS{I}{p}{J}$.
\end{definition}

Just as for objects with nabla flags there is a ``hin-und-her''
lemma, whose proof is similar to that for objects with
nabla flags (and works because the support of $M \in \RR{I}{J}$ is always contained in $\Gr{I}{C}{J}$ for some finite subset $C \subset \W{I}{J}$).

\begin{lemma}[``Hin-und-her lemma for delta flags''] Let $p_1, p_2,
\dots$ be an enumeration of the elements of $\W{I}{J}$ refining the Bruhat order and let $\check{C}(m) = \{ p_{m+1},
p_{m+2}, \dots \}$. Then $M \in \RR{I}{J}$ is in $\nabflag{I}{J}$
if and only if, for all $m$, the subquotient
 \begin{equation*}
  \Gamma_{\check{C}(m-1)}M / \Gamma_{\check{C}(m) }M
 \end{equation*}
is isomorphic to a direct sum of shifts of $\RS{I}{p_m}{J}$.

Moreover, if $M \in \RR{I}{J}$ and $p = p_{m}$ then the natural map
\begin{equation*}
 \Gamma_{ \ge p}  M / \Gamma_{ > p} M \to \Gamma_{\check{C}(m-1)}M /
\Gamma_{\check{C}(m) }M
\end{equation*}
is an isomorphism.
\end{lemma}

For each $p \in \W{I}{J}$ we renormalise $\RS{I}{p}{J}$ and define
\label{lab:del}
\begin{equation*}
\del{I}{p}{J} = \RS{I}{p}{J}[-\ell(p_-)].
\end{equation*}
If $\id  \in p$ we sometimes omit $p$ and write $\del{I}{}{J}$ for $\del{I}{p}{J}$. If $M \in \delflag{I}{J}$ then we may find polynomials $h_p(M) \in
\N[v,v^{-1}]$ such that, for all $p  \in  \W{I}{J}$, we
have an isomorphism
\begin{equation*}
 \Gamma_{\ge p}  M /  \Gamma_{> p}  M \cong h_p(M) \cdot \del{I}{p}{J}.
\end{equation*}

We define the \emph{delta character} to be the map
\begin{align*}
 \delchar : \delflag{I}{J} & \to \He{I}{J} \\
 M & \mapsto \sum_{p \in \W{I}{J}}
v^{\ell(p_-)-\ell(p_+)}h_p(M) \mst{I}{p}{J}.
\end{align*}

The analogue of Theorem \ref{thm:translation} in this context is the following:

\begin{theorem} \label{thm:translation2}
Let $K \subset S$ with either $J \subset K$ or $J \supset K$.
\begin{enumerate}
\item If $M \in \delflag{I}{J}$ then $B \cdot \on{J}{K} \in \delflag{I}{K}$.
\item The following diagrams commute:
\begin{displaymath}
\xymatrix@C=1.5cm{
\delflag{I}{J} \ar[d]^{\delchar} \ar[r]^{ - \cdot \on{J}{K}}&
\delflag{I}{K} \ar[d]^{\delchar} &
\delflag{I}{J} \ar[d]^{\delchar} \ar[r]^{[1]}& \delflag{I}{J}
\ar[d]^{\delchar} \\
\He{I}{J} \ar[r]^{*_J \mst{J}{}{K} } & \He{I}{K} &
\He{I}{J} \ar[r]^{v\cdot } & \He{I}{J}
}
\end{displaymath}
\end{enumerate}
\end{theorem}

\label{lab:duality}
We define a \emph{duality} functor
\begin{align*}
 D :  \bMod{R^I}{R^J}  &\to \bMod{R^I}{R^J} \\
M & \mapsto \Hom_{R^I}  (M, R^I[2\ell(\wo{J})])
\end{align*}
where we make $DM$ into a bimodule using the bimodule structure on
$M$. That  is, if $f \in DM$, then
\begin{equation*}
 (r_1fr_2)(m) = f(r_1mr_2) \qquad \text{for all $m \in M$.}
\end{equation*}
We do not include reference to $I$ and $J$ in the notation for $D$,
and hope this will not lead to confusion. The following proposition
shows that the translation functors commute with duality.

\begin{proposition} \label{prop:transcommute}
 Let $K \subset S$ be finitary with either $J \subset K$ or $J \supset
K$, and let $M \in \bMod{R^I}{R^J}$. In $\bMod{R^I}{R^K}$ one has
 \begin{equation*}
  D(M \cdot \out{J}{K}) \cong (DM) \cdot \out{J}{K}.
 \end{equation*}
\end{proposition}

\begin{proof}
If $J \subset K$ then the isomorphism $D(M \cdot \on{J}{K}) \cong (DM)
\cdot \on{J}{K}$ is a tautology. So assume that $J \supset K$. We will
use standard isomorphisms discussed in Section \ref{sec:not} and switch between left and right modules as
appropriate (note that we have already done this once in the
definition of $D$). In $\bMod{R^I}{R^K}$ we have
\begin{align*}
 D(M \cdot \out{J}{K}) & = \Hom_{R^I}(M \otimes_{R^J} R^K, R^I[2\ell(\wo{K})]) \\
&  \cong \Hom_{R^J}(R^K, \Hom_{R^I}(M, R^I[2\ell(\wo{K})]) &
\text{(\ref{eq:adj3})}\\
&  \cong \Hom_{R^I}(M, R^I[2\ell(\wo{K})]) \otimes_{R^J} \Hom_{R^J}
(R^K, R^J) & \text{(\ref{eq:adj4})}\\
& \cong \Hom_{R^I}(M, R^I[2\ell(\wo{J})]) \otimes_{R^J} R^K
& \text{(\ref{cor:reldual})}
\\
& =  (DM) \cdot \out{J}{K} & \qedhere
\end{align*}
\end{proof}

Theorem \ref{thm:translation2} now follows from Theorem \ref{thm:translation} and the following proposition, which also explains the name ``duality''.

\begin{proposition} \label{prop:D}
The restriction of $D$ to $\nabflag{I}{J}$ defines an equivalence of $\nabflag{I}{J}$ with ${\delflag{I}{J}}^{opp}$ and we have a commutative diagram:
\begin{equation*}
\xymatrix@C=0.3cm{
\nabflag{I}{J} \ar[rr]^{D} \ar[dr]_{\nabchar} & & {\delflag{I}{J}}^{opp}
\ar[dl]^{\delchar} \\
& \He{I}{J}}
\end{equation*}
\end{proposition}

Before we begin the proof we state a lemma, describing the effect of $D$ on a nabla module.

\begin{lemma} \label{lem:Dnab}
If $p \in \W{I}{J}$ we have
\begin{equation*}
D(\nab{I}{p}{J}) \cong \del{I}{p}{J}[\ell(p_+) - \ell(p_-)].
\end{equation*}\end{lemma}

\begin{proof} Let $K = I \cap p_-Jp_-^{-1}$. In $\lMod{R^K}$ we have isomorphisms
\begin{align*}
\Hom_{R^I}(R^K, R^I[2\ell(\wo{J})]) &
\cong R^K[2(\ell(\wo{I}) + \ell(\wo{J}) - \ell(\wo{K}))] &
(\text{Cor. \ref{cor:reldual}}) \\
& \cong R^K[2(\ell(p_+) - \ell(p_-))] & (\ref{eq:poinc1}).
\end{align*}
As a left module, $\RS{I}{p}{J}$ is equal to $R^K$ where $R^I$ acts via the inclusion $R^I
\hookrightarrow R^K$. Hence
\begin{equation*}
  D(\RS{I}{p}{J}) \cong \RS{I}{p}{J}[2(\ell(p_+) - \ell(p_-))]
\end{equation*}
and we have
\begin{align*}
D(\nab{I}{p}{J}) \cong D( \RS{I}{p}{J} [\ell(p_+)]) \cong
 \RS{I}{p}{J}[\ell(p_+) - 2\ell(p_-)] \cong
\del{I}{p}{J}[\ell(p_+) - \ell(p_-)]
\end{align*}
as claimed.
\end{proof}

\begin{proof}[Proof of Proposition \ref{prop:D}]
Let $M \in \nabflag{I}{J}$. We have to  show that $DM \in \delflag{I}{J}$, and that $\nabchar(M) = \delchar(DM)$. Choose an enumeration $p_1, p_2, \dots$ of the
elements of $\W{I}{J}$ refining the Bruhat order and let
$C(m) = \{ p_1, \dots, p_m \}$ and $\check{C}(m) = \{ p_{m+1}, p_{m+2}, \dots\}$. 
As $M \in \nabflag{I}{J}$ we can find 
polynomials $g_m \in \N[v,v^{-1}]$ such that, for all $m$, we have an
exact sequence
\begin{equation*}
\Gamma_{C(m-1)} M \hookrightarrow \Gamma_{C(m)} \sur g_m \cdot \nab{I}{p_m}{J}.
\end{equation*}
Consider the ``cofiltration'':
\begin{equation} \label{eq:cofilt}
\cdots \sur M / \Gamma_{C(m-1)} M \sur M / \Gamma_{C(m)} M \sur \cdots
\end{equation}
By the third isomorphism theorem we have an exact sequence
\begin{equation*}
g_m \cdot \nab{I}{p_m}{J} \hookrightarrow M / \Gamma_{C(m-1)} M \sur M
/ \Gamma_{C(m)}M.
\end{equation*}
We know that
$\nab{I}{p}{J}$ is graded free as an $R^I$-module for all $p$. We
conclude, using induction and the above exact sequence that the same
is true of every module in (\ref{eq:cofilt}). In particular, $D$ is
exact when applied to (\ref{eq:cofilt}) and we obtain a filtration of
$DM$
\begin{equation} \label{eq:DMfilt}
\cdots \hookleftarrow D(M / \Gamma_{C(m-1)} M) \hookleftarrow D(M /
\Gamma_{C(m)} M) \hookleftarrow \cdots
\end{equation}
with successive subquotients isomorphic to
\begin{equation} \label{eq:subquot}
D(g_m \cdot \nab{I}{p_m}{J}) \cong \overline{g_m} \cdot D(\nab{I}{p_m}{J})
\cong v^{\ell(p_+) - \ell(p_-)}\overline{g_m} \cdot \del{I}{p_m}{J}
\end{equation}
(for the second isomorphism we use Lemma \ref{lem:Dnab} above). It follows that the
filtration (\ref{eq:DMfilt}) is identical to
\begin{equation} \label{eq:DMfilt2}
\cdots \hookleftarrow \Gamma_{\check{C}(m-1)} DM \hookleftarrow
\Gamma_{\check{C}(m)} DM \hookleftarrow \cdots.
\end{equation}
Thus, by the ``hin-und-her'' lemma we conclude that $M \in
\delflag{I}{J}$. Using (\ref{eq:subquot}) and the ``hin-und-her''
lemma again we see that
\begin{equation*}
\nabchar(M) = \sum \overline{g_m} \mst{I}{p_m}{J} = \delchar(DM).
\end{equation*}
Lastly, the restriction of $D$ to $\nabflag{I}{J}$ gives an
equivalence with ${\delflag{I}{J}}^{opp}$ because the objects in both categories are free as left $R^I$-modules.
\end{proof}

\section{Singular Soergel bimodules and their classification}
\label{sec:classification}

In this section we complete the categorication of the Hecke category in terms of Soergel bimodules. After the preliminary work completed in the previous sections, the only remaining difficulty is the classification of the indecomposable objects in $\BB{I}{J}$. The key to the classification is provided by Theorem \ref{thm:hom} which explicitly describes the graded dimension of $\Hom(M, N)$ for certain combinations of Soergel bimodules and modules with nabla and delta flags.

In Section \ref{subsec:defs} we define the categories of singular Soergel bimodules, as well as a certain smaller category of bimodules (the ``Bott-Samelson bimodules''), for which a description of homomorphisms is straightforward (Theorem \ref{thm:BShom}). In order to extend this description to all special bimodules we need to consider various localisations of Soergel bimodules, which occupies Section \ref{subsec:local}. In Section \ref{subsec:hom} we then prove the Theorem \ref{thm:hom} and the classification follows easily. In the last section we investigate the characters of indecomposable Soergel bimodules more closely, recall Soergel's conjecture and show that it implies a formula the characters of all indecomposable special bimodules in $\BB{I}{J}$ in terms of Kazhdan-Lusztig polynomials.

\subsection{Singular Bott-Samelson and Soergel bimodules}
\label{subsec:defs}

We finally come to the definition of Soergel bimodules.

\begin{definition} We define the categories of \emph{Bott-Samelson bimodules}, denoted $\preBB{I}{J}$, to be the smallest collection of full additive subcategories of $\bMod{R^I}{R^J}$ for all finitary subsets $I, J \subset S$ satisfying:
\begin{enumerate}
\item $\preBB{I}{I}$ contains $\RS{I}{}{I}$ for all finitary subsets $I \subset S$;
\item If $B \in \preBB{I}{J}$ then so is $B[\nu]$ for all $\nu \in \Z$;
\item If $B \in \preBB{I}{J}$ and $K \subset S$ is finitary, satisfying $J \subset K$ or $J \supset K$, then $B \cdot \out{J}{K} \in \preBB{I}{K}$;
\item If $B \in \preBB{I}{J}$ then all objects isomorphic to $B$ are in $\preBB{I}{J}$.
\end{enumerate}
We define the categories of \emph{singular Soergel bimodules},  denoted $\BB{I}{J}$, to be the smallest collection of additive subcategories of $\bMod{R^I}{R^J}$ for all finitary $I, J \subset S$ satisfying:
\begin{enumerate}
\item $\BB{I}{J}$ contains all objects of $\preBB{I}{J}$;
\item $\BB{I}{J}$ is closed under taking direct summands.
\end{enumerate}
We write $\preBB{}{}$ and $\BB{}{}$ instead of $\preBB{\emptyset}{\emptyset}$ and $\BB{\emptyset}{\emptyset}$.
\end{definition}

The definition of the category of singular Soergel bimodules is more technical than that used in the introduction. However, from condition 3) it is clear that $\preBB{I}{J}$ contains all tensor products
\begin{equation*}
 R^{I_1} \otimes_{R^{J_1}} R^{I_2} \otimes_{R^{J_2}} \cdots
\otimes_{R^{J_{n-1}}} R^{I_n}
\end{equation*}
where $I = I_1 \subset J_1 \supset I_2 \subset \dots \subset
J_{n-1} \supset I_n = J$ are all finitary subsets of $S$. It follows that the definition of $\BB{I}{J}$ given above and in the introduction are the same.

By Theorems \ref{thm:translation} and \ref{thm:translation2} it follows by induction that any  object $M \in \preBB{I}{J}$ lies in $\nabflag{I}{J}$ and $\delflag{I}{J}$. As the categories $\nabflag{I}{J}$ and $\delflag{I}{J}$ are closed under taking direct summands, the same is true of $\BB{I}{J}$.

\excise{
In particular, given any object $B \in \BB{I}{J}$ we can consider $\nabchar(B)$ and $\delchar(B)$ in $\He{I}{J}$.
It it straightforward (again using Theorems \ref{thm:translation} and \ref{thm:translation2}) to see that $\nabchar(B) = \overline{\delchar(B)}$ for all objects $B \in \preBB{I}{J}$. Once we have the classification we will be able to show that this is true for all objects in $\BB{I}{J}$.}

\subsection{Homomorphisms between Bott-Samelson bimodules}
\label{subsec:homBS}

In this section we use the fact that translation onto and out of the wall are adjoint (up to a shift) to establish a formula for all homomorphisms between Bott-Samelson bimodules.

We start by proving the adjunction.

\begin{lemma} \label{lem:homadj}
Let $I, J, K \subset S$ be finitary with either $J \subset K$ or $J \supset K$. Let $M \in \bMod{R^I}{R^J}$ and $N \in \bMod{R^I}{R^K}$. We have an isomorphism in $\lMod{R^I}$:
  \begin{equation*}
    \Hom(M \cdot \out{J}{K}, M ) \cong \Hom(M, N \cdot \out{K}{J})[\ell(\wo{K}) - \ell(\wo{J})].
  \end{equation*} 
  \excise{
  \begin{equation*}
    \Hom_{R^I\!-\!R^K}(M \cdot \out{J}{K}, M ) \cong \Hom_{R^I\!-\!R^J}(M, N \cdot \out{K}{J})[\ell(\wo{K}) - \ell(\wo{J})].
  \end{equation*} }
\end{lemma}

\begin{proof} If $J \supset K$ we have isomorphisms of $R^I$-modules:
  \begin{align*}
    \Hom_{R^I-R^K}(M  \cdot \out{J}{K}, N) 
& \cong  \Hom_{R^I-R^J}( M  , \Hom_{R^K}(R^K, N)) \quad (\ref{eq:adj1}) \\
& \cong \Hom_{R^I-R^J}(M, N_{R^J}) \\
& \cong \Hom_{R^I-R^J}(M, N \cdot \on{J}{K})[\ell(\wo{K}) - \ell(\wo{J})]
  \end{align*}

If $J \subset K$ then,  setting $\nu = \ell(\wo{K}) - \ell(\wo{J})$ we have isomorphisms of $R^I$-modules:
\begin{align*}
\Hom _{R^I-R^K} (& M \cdot \on{J}{K}, N) \cong \\
& \cong \Hom_{R^I-R^K}(M \otimes_{R^J} R^J, N) [-\nu] \\
& \cong \Hom_{R^I-R^J}(M, \Hom_{R^K}(R^J,  N))[-\nu]
& (\ref{eq:adj1}) \\
& \cong \Hom_{R^I-R^J}(M, N \otimes \Hom_{R^K}(R^J[\nu], R^K))
& (\ref{eq:adj2}) \\
& \cong \Hom_{R^I-R^J}(M, N\otimes_{R^K}R^J)[\nu]
& (\text{Cor. \ref{cor:reldual}}) \\
& \cong \Hom_{R^I-R^J}(M, N \cdot \out{K}{J})[\ell(\wo{K}) - \ell(\wo{J})] & \qedhere
\end{align*}\end{proof}

We can now establish the first version of the homomorphism formula.

\begin{theorem}\label{thm:BShom}
If $M \in \preBB{I}{J}$, $N \in \nabflag{I}{J}$ or $M \in \delflag{I}{J}$, $N \in \preBB{I}{J}$ then $\Hom(M,N)$ is graded free as an $R^I$-module and we have an isomorphism
  \begin{equation*}
    \Hom(M, N)[-\ell(\wo{J})] \cong \overline{ \langle \delchar(M), \nabchar(N) \rangle } \cdot R^I
  \end{equation*}
of graded $R^I$-modules.
\end{theorem}

\begin{proof} Let us first assume that $M \in \preBB{I}{J}$ and $N \in \nabflag{I}{J}$. Using Lemma \ref{lem:homadj} we see that, as $R^I$-modules
  \begin{equation*}
    \Hom(M \cdot  \out{J}{K}, N) [-\ell(\wo{K})] \cong \Hom(M, N \cdot \on{K}{J})[-\ell(\wo{J})].
  \end{equation*}
By (\ref{eq:formident}) and Theorems \ref{thm:translation} and \ref{thm:translation2} we  have
\begin{align*}
  \langle \ch_{\Delta}(M \cdot  & \out{J}{K}), \nabchar(N)  \rangle
 = \langle \ch_{\Delta}(M )*_J \mst{J}{}{K}, \nabchar(N)  \rangle =  \\
& = \langle \ch_{\Delta}(M ), \nabchar(N) *_K \mst{K}{}{J} \rangle 
 = \langle \ch_{\Delta}(M ), \nabchar(N \cdot  \on{K}{J}) \rangle
\end{align*}
We conclude that the formula is true for $(M \cdot \out{J}{K}, N)$ if and only if it is true for $(M, N \cdot \out{K}{J})$. It is also clear that it is true for $(M, N)$ if and only it if it true for any shift of $M$ or $N$.  Thus, without loss of generality, we may assume that $M = \RS{I}{}{I} = \del{I}{}{I}$.

By Lemma \ref{lem:biform2} we know
\begin{equation*}
\langle \delchar(\del{I}{}{I}), \nabchar(N)\rangle
= \langle v^{-\ell(\wo{I})} \mst{I}{}{I}, \nabchar(N) \rangle = \text{coefficient of $\mst{I}{}{I}$ in $\nabchar{N}$}.
\end{equation*}
Thus, by definition of $\nabchar$, we have
\begin{equation*}
  \Gamma_{W_I} N \cong \overline{ \langle \delchar(\del{I}{}{I}), \nabchar(N) \rangle} \cdot \nab{I}{}{I}
\end{equation*}
It follows that
\begin{equation*}
  \Hom(\del{I}{}{I}, N)[-\ell(\wo{I})] = \Gamma_{W_I}(N)[-\ell(\wo{I})]
= \overline{\langle \delchar(\del{I}{}{I}), \nabchar(N) \rangle} \cdot \RS{I}{}{I}
\end{equation*}
which settles the case when $M \in \preBB{I}{J}$ and $N \in \nabflag{I}{J}$.

If $M \in \delflag{I}{J}$ and $N \in \preBB{I}{J}$ then identical arguments to those above allow us to assume that $N = \nab{I}{}{I}$. We have
\begin{equation*}
  \Gamma^{W_I} M = \langle \delchar{M}, \mst{I}{}{I} \rangle \cdot \del{I}{}{I}
\end{equation*}
and hence
\begin{align*}
\Hom(M, \nab{I}{}{I})[-\ell(\wo{I})] & = \Hom(M,\RS{I}{}{I})  \\
& = \Hom(\Gamma^{W_I}M, \RS{I}{}{I}) \\
& \cong \overline{\langle \delchar{M}, \mst{I}{}{I} \rangle} \cdot \Hom(\RS{I}{}{I}, \RS{I}{}{I}) \\
& \cong \overline{\langle \delchar{M}, \nabchar(\nab{I}{}{I}) \rangle} \cdot R^I & \qedhere
\end{align*}
\end{proof}

\subsection{Some local results}
\label{subsec:local}
We would like to generalise the homomophism formula of the previous section to all objects in $\BB{I}{J}$. The crucial point is determining $\Hom(M, \nab{I}{p}{J})$ and $\Hom(\del{I}{p}{J}, N)$ for $M, N \in \BB{I}{J}$. For this we consider various localisations of special bimodules, which is the purpose of this section.

Given any reflection $t \in W$ let $\Rloc{t}$ denote the local ring of
$V^t \subset V$. In other words, in $\Rloc{t}$ we invert all functions
$f \in R$ which do not vanish identically on $V^t$.

The ring $\Rloc{t}$ is no longer graded and we will denote by
$\bimod{\Rloc{t}}{R}$ the category of $(\Rloc{t}, R)$-bimodules. The
lack of a grading on $\Rloc{t}$ means that we do not know if objects
in $\bimod{\Rloc{t}}{R}$ satisfy Krull-Schmidt, which explains some
strange wording below.

If $M, N \in \bMod{R}{R}$ are free as left $R$-modules, with $M$
finitely generated we have an isomorphism
\begin{equation*}
 \Hom_{\Rloc{t}-R} (\Rloc{t} \otimes_R M, \Rloc{t} \otimes_R N) \cong
 \Rloc{t} \otimes_R \Hom_{R-R}(M, N).
\end{equation*}
It follows that, with the same assumptions on $M$ and $N$,
\begin{equation*}
\Ext^1_{\Rloc{t}-R} (\Rloc{t} \otimes_R M, \Rloc{t} \otimes_R N)
 \cong
 \Rloc{t} \otimes_R \Ext^1_{R-R}(M, N).
\end{equation*}
Lemma \ref{lem:ext} tells us that that $\Ext^1_{R-R}(R_x, R_y)$ is non-zero
if and only if $y = rx$ for some reflection $r \in T$, in which case
it is supported on $\Gr{}{x}{} \cap \Gr{}{rx}{}$. We conclude that
\begin{equation} \label{eq:localext}
\Ext^1_{\Rloc{t}-R} ( \Rloc{t} \otimes_R R_x, \Rloc{t} \otimes_R R_y) = 0
\text{ unless $y = tx$}.
\end{equation}
(Alternatively, one  may explicitly split the extension of scalars of the generator of $\Ext^1(R_x, R_{rx})$ to $\bimod{\Rloc{t}}{R}$ using a Demazure operator, if $r \ne t$.)

Suppose that $M\in \bMod{R}{R}$ has a filtration with successive
subquotients isomorphic to a direct sum of shifts of $R_x$, and
that no (shift of) $R_x$ occurs in two different subquotients. By inducting over
the filtration of $M$ and using (\ref{eq:localext}), we see that
$\Rloc{t} \otimes_R M$ has a decomposition in which each summand is
either isomorphic to $\Rloc{t} \otimes_R R_x$ or is an extension
between $\Rloc{t} \otimes_R R_{x}$ and $\Rloc{t} \otimes_R R_{tx}$.

The next two results makes this decomposition more precise for special
classes of modules.

\begin{lemma} \label{lem:pdecomp}
Let $I, J \subset S$ be finitary and $p \in \W{I}{J}$ be a
double coset. In $\bimod{\Rloc{t}}{R}$ we have an isomorphism
 \begin{equation*}
   \Rloc{t} \otimes_R R(p) \cong \left \{ \begin{array}{ll}
\bigoplus_{x \in p } \Rloc{t} \otimes_R R_x & \text{if $tp \ne p$} \\
\bigoplus_{x \in p; x < tx} \Rloc{t} \otimes_R R_{x, tx} & \text{if $tp = p$}.
\end{array} \right .
 \end{equation*}
\end{lemma}

\begin{proof} Note that, by Proposition \ref{prop:doublebruhat},
either $tp = p$ or $tp \cap p = \emptyset$.
The lemma then follows by applying $\Rloc{t} \otimes_R -$ to the exact sequence in Proposition \ref{prop:exactsequence}. \end{proof}

\begin{proposition} \label{prop:specialloc} 
If $B \in \BB{I}{J}$ then $\Rloc{t} \otimes_{R^I} B
\otimes_{R^J} R \in \bimod{\Rloc{t}}{R}$ is isomorphic to a direct
summand in a direct sum of modules of the form $\Rloc{t} \otimes_R
R_x$ and $\Rloc{t}\otimes_R R_{x, tx}$ with $x < tx$. \end{proposition}

\begin{proof} If the statement is true for $B$, then it is true for any direct summand of $B$, and hence we may assume that $B \in \preBB{I}{J}$.
If $B = \RS{I}{}{I}$ then $R \otimes_{R^I} \RS{I}{}{I}
\otimes_{R^I} R \cong R(W_I)$ (Theorem \ref{thm:ind}) and the
necessary decomposition is provided by Lemma \ref{lem:pdecomp}. By the
inductive definition of $\preBB{I}{J}$ it is enough to show that, if the
lemma is true for $B \in \BB{I}{J}$, then it is also true for $B \cdot
\out{J}{K} \in \BB{I}{K}$ with $J \subset K$ or $J \supset K$. The
case $J \supset K$ is trivial, and so we are left with the case $J
\subset K$.

The module $B \otimes_{R^K} R$ is a direct summand in $B \otimes_{R^J}
R \otimes_{R^K}R$ and, by assumption, $\Rloc{t} \otimes_R B
\otimes_{R^J} R$ is a direct summand in a direct sum of the modules
$\Rloc{t} \otimes_R R_x$ and $\Rloc{t} \otimes_R R_{x,tx}$ with $x < tx$. Hence it
is enough to show that the statement of the lemma is true for
$\Rloc{t} \otimes_R R_x \otimes_{R^K} R$ and $\Rloc{t} \otimes_R R_{x,
tx} \otimes_{R^K} R$.

In the first case $R_x \otimes_{R^K} R \cong R(xW_K)$ (Theorem
\ref{thm:ind} again) and the decomposition follows again from Lemma
\ref{lem:pdecomp} together with the fact that $tx > x$.

In the second case there are two possibilities. If $tx = xt^{\prime}$
for a reflection $t^{\prime} \in W_K$ then $R_{x, tx}$ splits upon
restriction to $R^K$ (Lemma \ref{lem:simplesplit}) and we may apply
Lemma \ref{lem:pdecomp} again.

If $tx \ne xt^{\prime}$ for any reflection $t^{\prime} \in W_K$ then
the sets $xW_K$ and $txW_K$ are disjoint. By applying $- \otimes_{R^K}
R$ to the exact sequence $R_{x}[-2] \hookrightarrow R_{x,tx} \sur
R_{tx}$ and using the identification $R_x \otimes_{R^K} R \cong
R(W_K)$ we see that $R_{x,tx} \otimes_{R^K} R$ has a filtration with
subquotients (a shift of) $R_{w}$ with $w \in xW_K$ or $txW_K$. It
follows that we have an isomorphism
\begin{equation*}
\Rloc{t} \otimes_R R_{x,tx} \otimes_{R^K} R \cong \bigoplus_{y \in
W_K} E_{xy, txy}
\end{equation*}
where $E_{xy, txy}$ is a (possibly trivial) extension of $\Rloc{t} \otimes_R R_{xy}$ and
$\Rloc{t} \otimes_R R_{txy}$.

To identify $E_{xy, txy}$ we tensor the surjection $R(W_K) \sur
R_y$ with the exact sequence $R_x[-2] \hookrightarrow R_{x, tx} \sur
R_{tx}$ to obtain a diagram
\begin{equation*}
 \xymatrix{
R_x \otimes_{R^K} R[-2] \ar@{^{(}->}[r] \ar@{->>}[d] & R_{x, tx}
\otimes_{R^K} R \ar@{->>}[r]\ar@{->>}[d] & R_{tx} \otimes_{R^K} R
\ar@{->>}[d]  \\
R_{xy}[-2] \ar@{^{(}->}[r] & R_{xy, txy} \ar@{->>}[r] & R_{txy}.}
\end{equation*}
After tensoring with $\Rloc{t}$ the left and right surjections split
by Lemma \ref{lem:pdecomp}. It follows that $E_{xy, txy}$ is
isomorphic to $\Rloc{t} \otimes_R R_{xy, txy}$ for all $y \in W_K$ and
the lemma follows. \end{proof}

We now come to the goal of this section, which is to relate $\Hom(\del{I}{p}{J}, B)$ and $\Hom(B, \nab{I}{p}{J})$ for a singular Soergel bimodule $B \in \BB{I}{J}$ to the nabla and delta filtrations on $B$. This provides the essential (and trickiest) step in generalising the homomorphism formula for Bott-Samelson bimodules to all Soergel bimodules.

The arguments used to establish this relation are complicated and so we first sketch the basic idea. Let us consider a nabla filtration on a Bott-Samelson bimodule $B$. By Theorem \ref{thm:BShom} we know the rank of $\Hom(\del{I}{p}{J}, B)$ in terms of $\Gamma_{p}^{\le}B$ and a simple calculation confirms that $\Hom(\del{I}{p}{J}, B)$ and $\Gamma_p^{\le} B[-\ell(p_-)]$ have the same graded rank as left $R^I$-modules.

Given a morphism $\alpha : \del{I}{p}{J} \to B$ one may consider the image of a non-zero element of lowest degree in $\Gamma_p^{\le}B$ and one obtains in this way an injection
\begin{equation*} 
  \Hom(\del{I}{p}{J}, B) \to \Gamma_p^{\le}B[\ell(p_-)].
\end{equation*}
One might hope that this maps into a submodule isomorphic to $\Gamma_p^{\le}B[-\ell(p_-)]$, which would explain the above equality of ranks.

In order to show that this is the case we choose a decomposition
\begin{equation*}
\Gamma_p^{\le}B \cong P \cdot \RS{I}{p}{J}
\end{equation*}
and recall that $\RS{I}{p}{J}$ has the structure of a graded algebra compatible with the bimodule structure. In particular, elements in $\RS{I}{p}{J}$ define endomorphisms of $\Gamma_p^{\le} B$ (which in general do not come from acting by an element in $R^I \otimes R^J$). Given an element $m \in \RS{I}{p}{J}$, we will abuse notation and denote by $m\Gamma_p^{\le} B$ the image of this endomorphism.

We define an element $m_p \in \RS{I}{p}{J}$ of degree $2\ell(p_-)$ and argue (using localisation) that the above injection lands in
\begin{equation*}
  m_p \Gamma_p^{\le} B[\ell(p_-)] \cong \Gamma_p^{\le} B[-\ell(p_-)].
\end{equation*}
Thus the two modules $\Gamma_p^{\le} B[-\ell(p_-)]$ and $\Hom(\del{I}{p}{J}, B)$ are isomorphic.

\begin{remark} If $W$ is a finite one may make the arguments in this section simpler by considering certain elements (similar to our $\phi_x \in R(p)$) constructed using Demazure operators. This is discussed in \cite{SoBimodules}, Bemerkung 6.7.
\end{remark}

We begin by defining the special elements $m_p \in \RS{I}{p}{J}$. Recall that, by definition, the modules $\RS{I}{p}{J}$ are the invariants in $R$ under $W_K$, where $K = I \cap p_-Jp_-^{-1}$.

\begin{lemma} The element
  \begin{equation*}
    m_p = \prod_{t \in T \atop tp_- < p_-} h_t \in R.
  \end{equation*}
lies in $\RS{I}{p}{J}$.
\end{lemma}

\begin{proof} Because $xh_s = h_{xsx^{-1}}$ if $x \in W$ (\ref{eq:hident}) it is enough to show that if $s \in I \cap p_-Jp_-^{-1}$ and $t \in T$ with $tp_- < p_-$, then $(sts)p_- < p_-$. Choose $r \in J$ such that $sp_- = p_-r$. We have either $(sts)sp_- = stp_- \le sp_-$ or $stp_-  \ge sp_-$. However the latter is impossible as $tp_- \notin p$. Similarly, either $stp_-r \le sp_-r = p_-$ or $stp_-r \ge sp_-r$ and the latter is again impossible. It follows that $(sts)p_- \le p_-$ as claimed. \end{proof}

We now come to the main goal of this section.

\begin{theorem}
\label{thm:specialstalk} Let $I, J \subset S$ be finitary, $B \in \BB{I}{J}$ and $p  \in \W{I}{J}$. We have isomorphisms
  \begin{enumerate}
  \item   $\Hom(\RS{I}{p}{J}, B) \cong \Hom(\RS{I}{p}{J}, \Gamma_p^{\le} B)[-2\ell(p_-)]$,
\item $\Hom(B, \RS{I}{p}{J}) \cong \Hom( \Gamma^{\ge}_pB, \RS{I}{p}{J})[-2\ell(p_-)]$.\end{enumerate}
\end{theorem}

The proof depends on a lemma which we establish by considering various localisations of $B$. Given a subset $A  \subset W$ we extend the notion to sections supported in $\Gr{}{A}{}$ to modules $M \in \bimod{\Rloc{t}}{R}$ as follows. Writing $I_A$ for the ideal of functions vanishing on $\Gr{}{A}{}$, we define $\Gamma_{A} M$ to be the submodule of elements annihilated by $\langle I_A \rangle$, the ideal generated by $I_A$ in $\Rloc{t} \otimes R$.

\begin{lemma} \label{lem:comp}
For any pair of morphisms
  \begin{equation*}
    M \to B \to \RS{I}{p}{J}
  \end{equation*}
with $M \in \delflag{I}{J}$ such that $\Gamma_{\ge p} M = M$, the composition lands in $m_p \RS{I}{p}{J}$. \end{lemma}

\begin{proof} As in Lemma \ref{lem:Rpinv} let us regard $\RS{I}{p}{J}$ as the subalgebra of $W_I \times W_J$-invariants in $R(p)$. Using Theorem \ref{thm:ind} we obtain, for all $t \in T$, a commutative diagram (where the vertical inclusions are inclusions of abelian groups):
\begin{equation*}
\xymatrix@C=0.4cm@R=0cm{
 m \in M \ar[r] & B \ar[r] & \RS{I}{p}{J} \\
\cap & \cap & \cap \\
R \otimes_{R^I}  M \otimes_{R^J} R \ar[r] & R \otimes_{R^I} B \otimes_{R^J} R \ar[r] & R(p) \ni (f_x) \\
\cap & \cap & \cap \\
 \Rloc{t} \otimes_{R^I}  M \otimes_{R^J} R\ar[r] & \Rloc{t} \otimes_{R^I} B \otimes_{R^J} R \ar[r] & \Rloc{t} \otimes_R R(p)}
\end{equation*}
Denote by $f = (f_x)$ the image of $m \in M$ in $R(p)$  as shown. By $W_I \times W_J$-invariance, it is enough to show that $f_{p_-}$ is divisible by $m_p$.

To this end, let $t \in T$ satisfy $tp_- < p_-$. Considering elements supported on $\Gr{}{p_-}{}$ and $\Gr{}{tp_-}{}$ and using Lemma \ref{lem:pdecomp} and Proposition \ref{prop:specialloc} we see that the bottom row admits a morphism to a composition of the form
\begin{equation*}
  \Rloc{t} \otimes_R R_{p_-} \to \bigoplus \Rloc{t} \otimes_R R_{tp_-, p_-} \to \Rloc{t} \otimes_R R_{p_-}.
\end{equation*}
The composition of any two such maps must land in $h_t \Rloc{t} \otimes_R R_{p_-}$. It follows that
\begin{equation*}
  f_{p_-} \in R \cap \bigcap_{t \in T \atop tp_- < p_-} h_t \Rloc{t} \otimes_R R = m_pR
\end{equation*}
and the lemma follows.
\end{proof}

\begin{proof}[Proof of Theorem \ref{thm:specialstalk}]
First note that if the theorem is true for a module $B$, then it is true to any direct summand of $B$. Thus we may assume without loss of generality that $B \in \preBB{I}{J}$.

We begin with 1). Let $\alpha : \RS{I}{p}{J} \to B$ be a morphism. As $\supp \RS{I}{p}{J} = \Gr{I}{p}{J}$ the image of $\alpha$ is  contained in $\Gamma_{\le p} B$ and, by composing with the quotient map we obtain a map $\RS{I}{p}{J} \to \Gamma^{\le}_p B$. This yields a morphism
\begin{equation*}
  \Phi: \Hom(\RS{I}{p}{J}, B) \to \Hom(\RS{I}{p}{J}, \Gamma_p^{\le} B).
\end{equation*}
As $B$ has a nabla flag, any element of $B$ has support consisting of a union of $\Gr{I}{q}{J}$ for $q \in \W{I}{J}$ by Lemma \ref{lem:standsupp}. It follows that $\Phi$ is injective.

Let us now fix an isomorphism
\begin{equation*}
  \Gamma_p^{\le} B \cong P \cdot \RS{I}{p}{J}.
\end{equation*}
By Lemma \ref{lem:comp} above, given any $\alpha \in \Hom(\RS{I}{p}{J}, B)$ the image of $\Phi(\alpha)$ is contained in $P \cdot m_p \RS{I}{p}{J} \cong \Gamma_p^{\le} B[-2\ell(p_-)]$. Thus we obtain an injection
\begin{equation} \label{eq:quotiso}
  \Hom(\RS{I}{p}{J}, B) \to \Hom(\RS{I}{p}{J}, \Gamma_p^{\le}B)[-2\ell(p_-)].
\end{equation}
We compare ranks in order to show that this is an isomorphism.

Let us write $g \in \N[v,v^{-1}]$ for the coefficient of $\mst{I}{p}{J}$ in $\nabchar(N)$ written in the standard basis. By Theorem \ref{thm:BShom}, we have, as left $R^I$-modules,
\begin{align*}
  \Hom(\RS{I}{p}{J}, B)[\ell(p_-) - \ell(\wo{J})]
& \cong \Hom(\del{I}{p}{J}, B)[-\ell(\wo{J})] \\
& \cong \overline{ \langle v^{\ell(p_-) - \ell(p_+)} \mst{I}{p}{J}, \nabchar(B) \rangle }  \cdot R^I\\
& \cong \overline{g} \frac{\pi(p)}{\pi(J)} \cdot R^I.
\end{align*}
One the other hand,
\begin{align*}
  \Hom(\RS{I}{p}{J}, \Gamma_p^{\le}B)&[-\ell(p_-) - \ell(\wo{J})]  \cong \\
& \cong \overline{g} \cdot \nab{I}{p}{J}[-\ell(p_-) - \ell(\wo{J}) ] 
& \text{(Cor. \ref{cor:standardhoms})}\\
& = \overline{g} \cdot \RS{I}{p}{J}[\ell(p_+)-\ell(p_-) - \ell(\wo{J}) ] \\
& = \overline{g} \cdot \RS{I}{p}{J}[\ell(\wo{I})- \ell(\wo{I,p,J}) ] 
& \text{(\ref{eq:poinc1})}\\
& = \overline{g} \frac{\pi(I)}{\pi(I,  p, J)} \cdot R^I
& \text{(Cor. \ref{cor:relinv})}\\
& = \overline{g} \frac{\pi(p)}{\pi(J)} \cdot R^I.
& \text{(\ref{eq:poinc3})}\\
\end{align*}
Thus (\ref{eq:quotiso}) is an isomorphism and 1) follows.

We now turn to 2) which, of course, is similar. Let $\alpha : B \to \RS{I}{p}{J}$ be a morphism. For support reasons, $\alpha$ annihilates $\Gamma_{> p} B$ and hence factorises to yield a map
$\Gamma_p^{\ge}B \to \RS{I}{p}{J}$. We obtain in this way an injection
\begin{align*}
  \Phi: \Hom(B, \RS{I}{p}{J}) \to \Hom(\Gamma_p^{\ge}B, \RS{I}{p}{J}).
\end{align*}
Let us fix an isomorphism
\begin{equation*}
  \Gamma_p^{\ge} B \cong P \cdot \RS{I}{p}{J}
\end{equation*}
for some $P \in \N[v,v^{-1}]$.  By the above lemma if $\alpha \in \Hom(B, \RS{I}{p}{J})$ then the image of $\Phi(\alpha)$ is contained in $P \cdot m_p \RS{I}{p}{J}$ and thus we obtain an injection
\begin{equation*}
\Hom(B, \RS{I}{p}{J}) \to \Hom(\Gamma_p^{\ge}B, \RS{I}{p}{J})[-2\ell(p_-)].
\end{equation*}
Again we compare ranks.
Choose $h \in \N[v,v^{-1}]$ such that $\Gamma_p^{\ge} B \cong h \cdot \del{I}{p}{J}$. By Theorem \ref{thm:BShom} we have isomorphisms of left $R^I$-modules:
\begin{align*}
\Hom(B, \RS{I}{p}{J})[\ell(p_+)  - \ell(\wo{J})] & \cong 
\Hom(B, \nab{I}{p}{J})[-\ell(\wo{J})] \\
& \cong \overline{h} \frac{\pi(p) }{ \pi(J)} \cdot R^I.
\end{align*}
On the other hand
\begin{align*}
  \Hom(\Gamma_p^{\ge}B, & \RS{I}{p}{J})[-2\ell(p_-) + \ell(p_+) -\ell(\wo{J})] \cong \\
& \cong \Hom( h \cdot \del{I}{p}{J}, \RS{I}{p}{J})[-2\ell(p_-) + \ell(p_+) -\ell(\wo{J})] \\
& \cong \overline{h} \cdot \RS{I}{p}{J}[\ell(p_+)- \ell(p_-)  - \ell(\wo{J})]
\qquad \text{(Cor. \ref{cor:standardhoms})}\\
& \cong \overline{h} \frac{\pi(p)}{\pi(J)} \cdot R^I
\end{align*}
which completes the proof of 2).
\end{proof}

\subsection{The general homomorphism formula and classification}
\label{subsec:hom}
We can now prove the natural generalisation of Theorem \ref{thm:BShom} to all Soergel bimodules. For the duration of this section fix $I, J \subset S$ finitary.

\begin{theorem} \label{thm:hom}
If $M \in \BB{I}{J}$, $N \in \nabflag{I}{J}$ or $M \in \delflag{I}{J}$, $N \in \BB{I}{J}$ then $\Hom(M,N)$ is graded free as an $R^I$-module and we have an isomorphism
  \begin{equation*}
    \Hom(M, N)[-\ell(\wo{J})] \cong \overline{ \langle \delchar(M), \nabchar(N) \rangle } \cdot R^I
  \end{equation*}
of graded $R^I$-modules.
\end{theorem}

\begin{proof} We handle first the case $M \in \delflag{I}{J}$ and $N \in \BB{I}{J}$. We will prove the theorem via induction on the length of a delta flag of $M$. The base case where $M \cong \del{I}{p}{J}$ for some $p \in \W{I}{J}$ follows by essentially the same calculations as those in the proof of Theorem \ref{thm:specialstalk}. Namely, if we write $g$ for the coefficient of $\mst{I}{p}{J}$ in $\nabchar(N)$, we have
  \begin{align*}
    \Hom(\del{I}{p}{J},N) & \cong \Gamma_p^{\le} N[-\ell(p_-)] \\
 & \cong \overline{g} \cdot \RS{I}{p}{J}[\ell(p_+) -\ell(p_-)] 
 \qquad \text{(Theorem \ref{thm:specialstalk})}\\
 & \cong  \overline{g} \frac{\Poinc(I)}{\Poinc(I,p,J)} \cdot R^I[\ell(\wo{J})] \\
 & \cong \overline{g} \frac{\Poinc(p)}{\Poinc(J)} \cdot R^I[\ell(\wo{J})] \\
 & \cong \overline{\langle \delchar(\del{I}{p}{J}), \nabchar(N) \rangle } \cdot R^I[\ell(\wo{J}].
  \end{align*}

For the general case we may choose $p \in \W{I}{J}$ minimal with $\Gamma^pM \ne 0$ and obtain an exact sequence
\begin{equation} \label{eq:mseq}
  \Gamma_{\ne p}M \hookrightarrow M \sur \Gamma^p M.
\end{equation}
By the minimality of $p$, both $\Gamma_{\ne p}M$ and $\Gamma^pM$ are in $\delflag{I}{J}$ and
\begin{equation*}
  \delchar{M}= \delchar(\Gamma_{\ne p}M) + \delchar(\Gamma^pM).
\end{equation*}
As $N \in \BB{I}{J}$ there exists some $\widetilde{N}\in \preBB{I}{J}$ in which $N$ occurs as a direct summand. The homomorphism formula for Bott-Samelson modules (\ref{thm:BShom}) tells us  that $\Hom(-, \widetilde{N})$ is exact when applied to (\ref{eq:mseq}). Hence the same is true for $\Hom(-, N)$ and we conclude by induction that we have isomorphisms of graded $R^I$-modules:
\begin{align*}
  \Hom(M,N) & \cong \Hom(\Gamma_{\ne p}M,N) \oplus \Hom(\Gamma^pM, N) \\
& \cong \langle \overline{\delchar(M), \nabchar(N)} \cdot R^I[\ell(\wo{J})].
\end{align*}

The case when $M \in \BB{I}{J}$ and $N \in \nabflag{I}{J}$ is handled similarly. If $N$ is isomorphic to $\nab{I}{p}{J}$ for some $p \in \W{I}{J}$, then similar calculations to those in Theorem \ref{thm:specialstalk} verify the theorem in this case. For general $N$ we choose $p$ minimal with $\Gamma_p N \ne 0$ and obtain an exact sequence
\begin{equation*}
  \Gamma_{p} N \hookrightarrow N \sur N/\Gamma_p N.
\end{equation*}
Applying $\Hom(M, -)$ this stays exact for the same reasons as above, and the isomophism in the theorem follows by induction.
\end{proof}

We now come to the classification.
 
\begin{theorem} \label{thm:classification}
For every $p \in \W{I}{J}$ there is, up to isomorphism, a unique indecomposable module $\Bi{I}{p}{J} \in \BB{I}{J}$ satisfying
\begin{enumerate}
\item $\supp \Bi{I}{p}{J} \subset \Gr{I}{\le p }{J}$;
\item $\Gamma^p (\Bi{I}{p}{J}) \cong \nab{I}{p}{J}$.
\end{enumerate}
The bimodule $\Bi{I}{p}{J}$ is self-dual and any indecomposable object in $\BB{I}{J}$ is isomorphic to $\Bi{I}{p}{J}[\nu]$ for some $p \in \W{I}{J}$ and $\nu\in \Z$
\end{theorem}

In keeping with our notational convention, if $I = J = \emptyset$ we will write $B_w$ instead of $\Bi{I}{w}{J}$.

\begin{proof} Choose $p \in \W{I}{J}$. By Proposition
  \ref{prop:transseqprod} we can find a sequence
  $(J_i)_{0 \le i \le n}$ of finitary subsets of $S$ such that $I =
  J_0$, for
all $0 \le i < n$ either $J_i \subset J_{i+1}$ or $J_{i} \supset
J_{i+1}$ and
\begin{equation*} H := \mst{J_0}{}{J_1}*_{J_1} \mst{J_1}{}{J_2}*_{J_2} \cdots *_{J_{n-1}} \mst{J_{n-1}}{}{J_n} = \mst{I}{p}{J} + \sum_{q < p} \lambda_{q} \mst{I}{q}{J}.
\end{equation*}
Consider the module
  \begin{equation*}
    \widetilde{B}= \nab{I}{}{I} \cdot \out{J_0}{J_1} \cdot \out{J_1}{J_2} \cdot \dots \cdot \out{J_{n-1}}{J_{n}} \in \BB{I}{J}.
  \end{equation*}
By Theorem \ref{thm:translation} and Proposition \ref{prop:transseqprod} we have
$\nabchar{\widetilde{B}} = H$.
Hence $\widetilde{B}$ satisfies conditions 1) and 2) in the theorem. Let $\Bi{I}{p}{J}$ te the unique indecomposable summand of $\widetilde{B}$ with non-zero support on $\Gr{I}{p}{J}$. Clearly $\Bi{I}{p}{J}$ also satisfies conditions 1) and 2).

Note that $\widetilde{B}$ is self-dual (because $\nab{I}{}{I}$ is and the translation functors commute with duality by Proposition \ref{prop:transcommute}). As $\Bi{I}{p}{J}$ is the only direct summand of $\widetilde{B}$ with support containing $\Gr{I}{p}{J}$, it follows that $\Bi{I}{p}{J}$ is also self-dual.

Let $M$ and $N$ be objects in $\BB{I}{J}$ and assume that $p$ is maximal for both modules with $\Gamma^pM \ne 0$ and $\Gamma^pN \ne 0$. Using Theorem \ref{thm:hom} we see that $\Hom(M, -)$ is exact when applied to the sequence
\begin{equation*}
  \Gamma_{\ne p} N \hookrightarrow N \sur \Gamma^p N.
\end{equation*}
In other words we have a surjection
\begin{equation*}
  \Hom(M, N) \sur \Hom(M, \Gamma^pN) = \Hom(\Gamma^pM, \Gamma^pN).
\end{equation*}
By symmetry, we also have a surjection
\begin{equation*}
  \Hom(N,M) \sur \Hom(\Gamma^pN, \Gamma^pM).
\end{equation*}
These surjections tell us that we can lift homomorphisms between $\Gamma^pM$ and $\Gamma^pN$ to $M$ and $N$.

Now assume that $M$ and $N$ are indecomposable. After shifting $M$ and $N$ if necessary we may find $\alpha : \Gamma^pM \to \Gamma^pN$ and $\beta: \Gamma^pN \to \Gamma^pM$ of degree zero, such that $\beta \circ \alpha$ is the identity on a fixed direct summand $\nab{I}{p}{J}$ in $\Gamma^pM$ and zero elsewhere. By the above arguments we may find lifts $\tilde{\alpha} : M \to N$ and $\tilde{\beta} : N \to M$ of $\alpha$ and $\beta$ of degree zero. As $M$ is indecomposable and $\tilde{b} \circ \tilde{\alpha}$ is not nilpotent it must be an isomorphism. Thus $\Gamma^p M\cong \nab{I}{p}{J}$ and $M$ is isomorphic to a direct summand of $N$. However $N$ is indecomposable by assumption and thus $M$ and $N$ are isomorphic.

We conclude that, for any fixed $p \in \W{I}{J}$, there is at most one isomorphism class (up to shifts) of indecomposable bimodules $B \in \BB{I}{J}$ such that $p$ is maximal with $\Gamma^p B \ne 0$.  The theorem then follows as we already know that $\Bi{I}{p}{J}$ satisfies these conditions. \end{proof}

The classification allows us to prove that indecomposable Soergel bimodules stay indecomposable when translated out of the wall:

\begin{proposition} \label{prop:outofthewall}
Let $K \subset I$ and $L \subset J$ be finitary subsets of $S$ and
\begin{equation*}
\quo : \W{K}{L} \to \W{I}{J}
\end{equation*}
be the quotient map. Choose $p \in \W{I}{J}$ and let $q$ be the unique maximal element in $\quo^{-1}(p)$.
\begin{enumerate}
\item In $\bMod{R^K}{R^L}$ we have an isomorphism
\begin{equation*}
R^K \otimes_{R^I} \Bi{I}{p}{J} \otimes_{R^J} R^L \cong \Bi{K}{q}{L}.
\end{equation*}
\item In $\bMod{R^I}{R^J}$ we have an isomorphism
\begin{equation*}
_{R^I}(\Bi{K}{q}{L})_{R^J} \cong \frac{\tPoinc(I)\tPoinc(J)}{\tPoinc(K)\tPoinc(L)} \cdot \Bi{I}{p}{J}.
\end{equation*}
\end{enumerate}
\end{proposition}

\begin{proof} For the course of the proof let use define
\begin{equation*}
P = \frac{\tPoinc(I)\tPoinc(J)}{\tPoinc(K)\tPoinc(L)}.
\end{equation*}
The composition of inducing to $\bMod{R^K}{R^L}$ and restricting to $\bMod{R^I}{R^J}$ always produces a factor of $P$. To get started, note that $\Gamma^p(\Bi{I}{p}{J}) \cong \nab{I}{p}{J}$ and hence (using Proposition \ref{prop:transsupp})
\begin{align*}
\Gamma_{\quo^{-1}(\{ \le p\})} (R^K \otimes_{R^I} \Bi{I}{p}{J} \otimes_{R^J} R^L) /
\Gamma_{\quo^{-1}(\{< p\})} & (R^K \otimes_{R^I} \Bi{I}{p}{J} \otimes_{R^J} R^L) \cong \\
& \cong R^K \otimes_{R^I} \nab{I}{p}{J} \otimes_{R^J} R^L
\end{align*}
The latter is isomorphic to a shift of $R(p)^{W_K \times W_L}$ by Theorem \ref{thm:ind} and hence is indecomposable. By the classification, we may write
\begin{equation} \label{eq:indiso}
R^K \otimes_{R^I} \Bi{I}{p}{J} \otimes_{R^J} R^L \cong \Bi{K}{q}{L} \oplus M
\end{equation}
for some $M \in \BB{K}{L}$ whose support is contained in $\Gr{K}{\quo^{-1}(\{<q\})}{L}$. It follows that 
\begin{equation*}
\Gamma_{\quo^{-1}(\{\le p\})} (\Bi{K}{q}{L}) /
\Gamma_{\quo^{-1}(\{< p\})} (\Bi{K}{q}{L}) \cong R^K \otimes_{R^I} \nab{I}{p}{J} \otimes_{R^J} R^L
\end{equation*}
This tells us (again by Proposition \ref{prop:transsupp}) that
\begin{align*}
\Gamma_{\le p} ({}_{R^I}(\Bi{K}{q}{L})_{R^J}) /
\Gamma_{< p} ({}_{R^I}(\Bi{K}{q}{L})_{R^J}) & \cong {}_{R^I}(R^K \otimes_{R^I} \nab{I}{p}{J} \otimes_{R^J} R^L)_{R^J} \\
& \cong P \cdot \nab{I}{p}{J}
\end{align*}
Therefore we may write
\begin{equation*}
{}_{R^I}(\Bi{K}{q}{L})_{R^J} \cong P \cdot \Bi{I}{p}{J} \oplus N
\end{equation*}
for some $N \in \BB{I}{J}$. Restricting (\ref{eq:indiso}) to $\bMod{R^I}{R^J}$ yields
\begin{align*}
P \cdot \Bi{I}{p}{J} 
\cong {}_{R^I}(\Bi{K}{q}{L})_{R^J} \oplus {}_{R^I} M_{R^J} \cong 
P \cdot \Bi{I}{p}{J} 
\oplus {}_{R^I} M_{R^J} \oplus N
\end{align*}
whence $M = N = 0$. Both claims then follow.
\end{proof}

\subsection{Characters and Soergel's conjecture} \label{sec:characters}

In this section we turn our attention to the characters of Soergel bimodules. We will see in the following theorem that the nabla character of a singular Soergel bimodule is determined by its delta character (and vice versa). Therefore we simplify notation and define
\begin{equation*}
  \ch(B) = \delchar(B)
\end{equation*}
for all Soergel bimodules $B$.

\begin{theorem} \label{thm:categorification}
Let $I, J$ and $K$ be finitary subsets of $S$.
\begin{enumerate}
\item For all $B \in \BB{I}{J}$ we have $\nabchar(B) = \overline{\delchar(B)}$.
\item We have a commutative diagram
\begin{equation*}
\xymatrix@C=2cm{ \BB{I}{J} \times \BB{J}{K} \ar[r]^{- \otimes_{R^J} - } \ar[d]^{\ch \times \ch} & \BB{I}{K} \ar[d]^{\ch} \\
\He{I}{J} \times \He{J}{K} \ar[r]^{-*_J-} & \He{I}{K}}.
\end{equation*}
\item The set $\{ \ch(\Bi{I}{p}{J}) \; | \; p \in \W{I}{J} \}$ builds a self-dual basis for $\He{I}{J}$.
\end{enumerate}
\end{theorem}

\begin{proof} We begin with 1). As $\nabchar(\RS{I}{}{I}) = \overline{\delchar(\RS{I}{}{I})}$ we may use Theorems \ref{thm:translation} and \ref{thm:translation2} to conclude that the statement is true for all Bott-Samelson bimodules. We now use induction over the Bruhat order on $\W{I}{J}$ to show that $\nabchar(\Bi{I}{p}{J})= \overline{\delchar(\Bi{I}{p}{J})}$ for all $p  \in \W{I}{J}$, which implies the claim. If $p$ contains the identity, then $\Bi{I}{p}{J}$ is Bott-Samelson and so the claim is true. For general $p \in \W{I}{J}$ we may (as in the proof of Theorem \ref{thm:classification}) find a Bott-Samelson module $N$ such that $N \cong \Bi{I}{p}{J} \oplus \widetilde{N}$ and the support of $\widetilde{N}$ is contained in $\Gr{I}{<p}{J}$. We have
\begin{align*}
\nabchar(\Bi{I}{p}{J}) + \nabchar(\widetilde{N}) = \nabchar(N) = \overline{\delchar(N)} = \overline{\delchar(\Bi{I}{p}{J})} + \overline{\delchar(\widetilde{N})}.
\end{align*}
By induction $\nabchar(\widetilde{N}) = \overline{\delchar(\widetilde{N})}$ and the claim follows.

Statement 2) follows by a very similar argument. It is clear from Theorem \ref{thm:translation2} that the statement is true for Bott-Samelson bimodules.
Let us fix $M \in \BB{I}{J}$. It is enough to show that $\ch(M \otimes_{R^J} \Bi{J}{p}{K}) = \ch(M) *_J \ch(\Bi{J}{p}{K})$ for all $p \in \W{J}{K}$. Again we induct over the Bruhat order on $\W{J}{K}$. If $p$ is minimal then $\Bi{J}{p}{K}$ is Bott-Samelson and the claim follows by Theorem \ref{thm:translation2}. If $p \in \W{J}{K}$ is arbitrary then we may find, as above, a Bott-Samelson bimodule $N \in \preBB{J}{K}$ which decomposes as $N \cong \Bi{J}{p}{K} \oplus \widetilde{N}$ with the support of $\widetilde{N}$ contained in $\Gr{I}{<p}{J}$. We have\begin{align*}
\ch(M\otimes_{R^J}& \Bi{J}{p}{K}) + \ch(M \otimes_{R^J} \widetilde{N}) = 
 \ch(M \otimes_{R^J} N) = \\ = & \ch(M)*_J \ch(N)  = \ch(M)*_J \ch(\Bi{J}{p}{K}) + \ch(M) *_J \ch(\widetilde{N}).
\end{align*}
By induction $\ch(M \otimes_{R^J} \widetilde{N}) = \ch(M) *_J \ch(\widetilde{N})$ and the claim follows.

We now turn to 3). By Theorem \ref{thm:classification}, we have
\begin{equation*}
  \ch(\Bi{I}{p}{J}) = \mst{I}{p}{J} + \sum_{q < p} \lambda_q \mst{I}{q}{J}
\end{equation*}
for some $\lambda_q \in \N[v,v^{-1}]$. It follows that the set $\{ \ch(\Bi{I}{p}{J}) \}$ gives a basis for $\He{I}{J}$. The self-duality of $\ch(\Bi{I}{p}{J})$ follows from the self-duality of $\Bi{I}{p}{J}$ and Proposition \ref{prop:D}:
\begin{equation*}
  \ch(\Bi{I}{p}{J})  = \delchar(D\Bi{I}{p}{J}) = \nabchar(\Bi{I}{p}{J})  = \overline{\ch(\Bi{I}{p}{J})}. \qedhere
\end{equation*}
\end{proof}

Given the theorem it is desirable to understand this basis  $\{ \ch(\Bi{I}{p}{J}) \}$ for $p \in \W{I}{J}$ more explicitly. We will finish by recalling Soergel's conjecture on the characters of the indecomposable bimodules in $\BB{}{}$ (recall that we write $\BB{}{}$ instead of $\BB{\emptyset}{\emptyset}$).

In \cite{SoBimodules} Soergel considers the full subcategory of $\bMod{R}{R}$ consisting of all objects isomorphic to direct sums, summands and shifts of objects of the form
\begin{equation} \label{eq:iteratedtensor}
R \otimes_{R^s} R \otimes_{R^t} \dots \otimes_{R^u} R
\end{equation}
where $s,t, \dots, u \in S$ are simple reflections. A priori, this category may not contain all objects of $\BB{}{}$. However using the same arguments as in the proof of Theorem \ref{thm:classification} one can show that one obtains all indecomposable objects in $\BB{}{}$ as direct summands of bimodules of the form (\ref{eq:iteratedtensor}) for reduced expressions $st\dots u$. Thus Soergel's category is precisely $\BB{}{}$.

\excise{
Let us temporarily denote this category by $\BB{}{}{}_S$.

Clearly all objects in $\BB{}{}{}_S$ lie in $\BB{}{}{}$. On the other hand, if $w \in W$ and $st\dots u$ is a reduced expression for $w$ then, by Theorem \ref{translation2}
\begin{equation*}
\ch (R \otimes_{R^s} R \otimes_{R^t} \dots \otimes_{R^u} R)
= \h{s}\h{t}\dots \h{u}
\end{equation*}
and hence
\begin{equation*}
R \otimes_{R^s} R \otimes_{R^t} \dots \otimes_{R^u} R \cong B_w \oplus \bigoplus_{x < w} \lambda_x \cdot B_x
\end{equation*}
for some polynomials $\lambda_x \in \N[v,v^{-1}]$. It follows that $\BB{}{}_S$ contains all indecomposable objects in $\BB{}{}$, and hence the two categories are the same.}

The following is Vermutung 1.13 in \cite{SoBimodules}.

\begin{conjecture}(Soergel) For all $w \in W$ we have $\ch(B_w) = \h{w}$. \end{conjecture}

If Soergel's conjecture is true then, by Proposition \ref{prop:outofthewall},
\begin{equation*}
\ch(R \otimes_{R^I} \Bi{I}{p}{J} \otimes_{R^J} R) = \ch(\Bi{}{p_+}{}) = \h{p_+}.
\end{equation*}
By Theorem \ref{thm:categorification}, $\ch(R \otimes_{R^I} \Bi{I}{p}{J} \otimes_{R^J} R)$ is equal to $\ch(\Bi{I}{p}{J})$ regarded as an element of $\He{}{}$. Hence
\begin{equation*}
\ch(\Bi{I}{p}{J}) = \mkl{I}{p}{J}.
\end{equation*}

\subsection{Acknowledgements}
This paper is a version of my PhD thesis at the University of Freiburg written under the supervision of Wolfgang Soergel. I would like to thank him warmly for his support and inspiration during my years in Freiburg. He suggested that a systematic study of singular Soergel bimodules would be worthwhile, and patiently explained many of the arguments in \cite{SoBimodules}. I am also greatly indebted to Matthew Dyer, who pointed out an error in an earlier version of this work. Peter Fiebig also offered numerous helpful suggestions, and it is at his suggestion that I used the modules $R(X)$ to control the extension of scalars of standard modules. I am indebted to Marco Mackaay for pointing out that what I had been calling the Hecke category was another reincarnation of the Schur algebra, which explains its renaming to Schur algebroid in this paper. Lastly I would like to thank Raph\"ael Rouquier, Olaf Schn\"urer, Catharina Stroppel and Ben Webster for many useful discussions.

\newpage

\section{Erratum for ``Singular Soergel bimodules''}

\maketitle

Theorem 1.2 in the paper \cite{SSBim} (Theorem 7.12(2) in the main body) is false as
stated. For example, consider $R^I$ as an $(R^I,
R^I)$-bimodule. This is a singular Soergel bimodule and
\[
R^I \otimes_{R^I} R^I \cong R^I.
\]
Thus, $R^I$ should be mapped to an idempotent in ${}^I \HC^I$ under
$\ch$. The only reasonable possibility is
\[
\ch(R^I) =
{}^IH^I.\] However, this is not the case:
\begin{gather*}
\ch_{\nabla}(R^I) = \ch( v^{-\ell(w_I)} \cdot R^I[\ell(w_I)]) =
\ch_{\nabla}( v^{-\ell(w_I)} \cdot {}^I\nabla ^I) = v^{\ell(w_I)}
\cdot {}^IH^I \\
\ch_{\Delta}(R^I) = \ch_{\Delta}({}^I\Delta^I) = v^{-\ell(w_I)}\cdot
{}^IH^I \\
\ch(R^I) = \ch_{\Delta}(R^I) = v^{-\ell(w_I)}\cdot
{}^IH^I 
\end{gather*}
(For the first line, we use the formulas at the beginning of \S3, and the definition of the
nabla module and character on page 4599. For the second line we use
the definition of the delta module on 4609 and delta character on
4610. The last line uses the definition of $\ch$ at the beginning of
\S 7.5.)

\subsection{The mistake}
The mistake is the second sentence of the second paragraph of the
proof of Theorem 7.12: ``It is clear from Theorem 6.14 
that the statement is true for Bott–Samelson bimodules.'' This is not
true, as we now explain.

Theorems 6.4
and 6.14 give us that $\ch(M
\cdot {}^J\nu^K) = \ch(M) *_J {}^JH^K$ for appropriate $M$. By
Definition 6.3 we have
\[
M
\cdot {}^J\nu^K = \begin{cases} M \otimes_{R^J} R^J[\ell(w_K) -
  \ell(w_J)] & \text{if $J \subset K$,}\\
  M \otimes_{R^J} R^K & \text{if $J \supset K$.} \end{cases}
  \]
However, $\ch( R^J[\ell(w_K)-\ell(w_J)])$ and $\ch(R^K)$ are equal to
${}^JH^K$ only up to a power of $v$. Thus, Theorem 7.12 is off by a
shift.

\subsection{The ``quick fix''} \label{sec:quick fix}
One can fix things in several ways, one way (which has minimal impact
on the rest of the paper) is as follows.
Let us write $\oldch$ for the
character map appearing in \cite[\S7.5]{SSBim}, and write $\newch$ for
\[
  \newch(M) = v^{\ell(w_I)} \oldch(M) \quad \text{for  $M \in {}^I \BC^J$.}
\]
This new definition corrects Theorem 7.12(2):

\begin{proposition} \label{prop:fix}
  For singular Soergel bimodules $B \in {}^I \BC^J$ and $B' \in  {}^J
  \BC^K$,
\[
  \newch(B \otimes_{R^J} B') = \newch(B) *_J
  \newch(B').
  \]
\end{proposition}

\begin{proof}
  Certainly Theorem 6.4 still holds with $\newch$ in place of
  $\oldch$:
  \[
    \newch(M \cdot {}^J \nu^K) = v^{\ell(w_I)} \oldch(M \cdot {}^J
    \nu^K) =v^{\ell(w_I)} \oldch(M) *_J {}^JH^K =  \newch(M) *_J {}^JH^K.
  \]
  However, $\newch$ also has the desired effect on bimodules corresponding to
  translation functors. When $J \supset K$ (and considering $(R^J,R^K)$-bimodules), we have
  \begin{gather*}
\newch( R^K) = v^{\ell(w_J)} \oldch({}^J\Delta^K) =
v^{\ell(w_J)-\ell(w_J)}\cdot {}^JH^K= {}^JH^K.
\end{gather*}
(We use that ${}^J\Delta^K = R^K$ (see bottom of page 4609) and the
definition of $\ch_\Delta$ on page 4610). When $J \subset K$,
  \begin{gather*}
\newch( R^J[\ell(w_K)-\ell(w_J)]) = v^{\ell(w_J)}  \oldch(
{}^J\Delta^K[\ell(w_K)-\ell(w_J)])  = \\ = v^{\ell(w_J)}v^{\ell(w_K)-\ell(w_J)}
 \oldch( {}^J\Delta^K) = v^{\ell(w_K)}
 \oldch({}^J\Delta^K)=v^{\ell(w_K)-\ell(w_K)} \cdot {}^JH^K = {}^JH^K 
\end{gather*}
Now the proof proceeds in the same was as \cite[Theorem 7.12]{SSBim}.
\end{proof}

Using Proposition \ref{prop:fix} to expand $\ch(B
\otimes_{R^J} B')$ and cancelling $v^{\ell(w_I)}$ from both sides leads to
the identity
\begin{equation} \label{eq:factor}
\oldch(B \otimes B') = v^{\ell(w_J)} \cdot \oldch(B) *_J \oldch(B')
  \end{equation}
Thus \cite[Prop. 7.12(2)]{SSBim} is wrong by a factor of
$v^{\ell(w_J)}$. Equation \eqref{eq:factor} also implies that we could
also fix $\oldch$ by defining
\begin{equation} \label{eq:newch'}
\ch'_{\textrm{new}}(M) := v^{\ell(w_J)} \cdot \oldch(M) \quad \text{for $M \in {}^I\BC^J$.}
  \end{equation}
  
\subsection{Indecomposables after the ``quick fix''} The change to
$\newch$ means that the characters of the (old 
normalizations of the) indecomposable Soergel bimodules are no longer
self-dual in the Hecke algebroid. One can fix this as follows: in the proof
of \cite[Theorem 7.10]{SSBim} we define ${}^IB_p^J$ to be the indecomposable
summand occurring inside translation functors applied to $R^I$, rather
than ${}^I \nabla^I$. (Thus the new indecomposable singular Soergel
bimodules agree with the old ones up to a shift by $\ell(w_I)$.)

\subsection{A more comprehensive fix} The above fixes the character
map, however a more comprehensive fix redefines the shifts on nabla and delta
modules.\footnote{Revisiting the paper after over a decade, it is not
  clear to the author where the apparently unmotivated shifts
  in the definition of the delta and nabla modules come from!} We
outline one possibility below.

In what follows we use wide tildes to define new objects. The unadorned
symbols refer to the objects as defined in \cite{SSBim}.

\emph{Step 1:} We define the duality as
\[
\widetilde{D} := \Hom_{R^I}(-, R^I) \quad \text{on $(R^I,R^J)$-bimodules.}
\]
(That is, we remove the shift present in the definition of $D$ on page
4610.)

\emph{Step 2:} We define
\[
{}^I \widetilde{\nabla}^J_p := {}^IR^J_p[\ell(p_+)-\ell(w_J)] =
{}^I\nabla^J_p[-\ell(w_J)]
\]
(This differs from the definition of ${}^I\nabla^J_p$ on page 4599 by
a shift by $-\ell(w_J)$.)

\emph{Step 3:} We define
\[
{}^I\widetilde{\Delta}^J_p := \widetilde{D}({}^I\widetilde{\nabla}^J_p).
  \]
(Thus we define the delta modules to be the duals (under
$\widetilde{D}$) of the nabla modules -- this is not the case in
\cite{SSBim}, see Lemma 6.17.) We claim that
\begin{equation} \label{eq:nablaform}
  \widetilde{D}({}^I\widetilde{\nabla}^J_p)={}^I\widetilde{\nabla}^J_p[-2\ell(p_-)]\end{equation}
Indeed, from the second displayed equation in the proof of Lemma 6.17
we deduce
\[
\widetilde{D}({}^IR^J_p)= {}^IR^J_p[2(\ell(p_+)-\ell(p_-)-\ell(w_J))]
\]
and hence
\[
\widetilde{D}({}^I\widetilde{\nabla}^J_p) = \widetilde{D}(
{}^IR^J_p[\ell(p_+)-\ell(w_J)]) = {}^IR^J_p[\ell(p_+)-2\ell(p_-)-\ell(w_J)] = {}^I\widetilde{\nabla}^J_p[-2\ell(p_-)].\]

\emph{Step 4:} We now define $\widetilde{\ch}_{\nabla}$ and
$\widetilde{\ch}_{\Delta}$ (almost) as before
\[
\widetilde{\ch}_{\nabla}(M) := \sum \overline{g_p(M)} {}^IH^J_p \quad
\text{and} \quad \widetilde{\ch}_{\Delta}(M) := \sum h_p(M) {}^IH^J_p 
\]
(compare pp. 4599 and 4610), however now $g_p(M)$ and $h_p(M)$ are
defined using $\widetilde{\nabla}$ and $\widetilde{\Delta}$
filtrations respectively. (Note that  the unmotivated shift in the
definition of $\ch_{\Delta}$ goes away in the definition of
$\widetilde{\ch}_{\Delta}$.)

Because the old and new normalizations of nabla modules are related by a
shift by $-\ell(w_J)$ it is immediate that we have
\begin{equation} \label{eq:relate}
\widetilde{\ch}_{\nabla}(M) = v^{-\ell(w_J)} \ch_\nabla(M).
  \end{equation}
Also, because nabla and delta modules are dual on the nose, the
analogue of \cite[Prop. 6.16]{SSBim} is immediate: we have
\begin{equation} \label{eq:dual}
\widetilde{\ch}_{\Delta}(\widetilde{D}M) = \widetilde{\ch}_{\nabla}(M).
\end{equation}

Thus we have
\begin{equation} \label{eq:relatedelta}
\widetilde{\ch}_\Delta(M) = v^{\ell(w_J)}\ch_{\Delta}(M)
\end{equation}
indeed
\begin{gather*}
\widetilde{\ch}_\Delta(M) \stackrel{\eqref{eq:dual}}{=} \widetilde{\ch}_{\nabla}(\widetilde{D}M) \stackrel{\eqref{eq:relate}}{=} 
v^{-\ell(w_J)} \ch_{\nabla}(DM[-2\ell(w_J)]) \stackrel{\text{(Prop. 6.16)}}{=} v^{\ell(w_J)}
\ch_{\Delta}(M).
  \end{gather*}

\emph{Step 5:} We set
\[
\widetilde{\ch}(M)= \widetilde{\ch}_\Delta(M)
\stackrel{\eqref{eq:relatedelta}}{=} v^{\ell(w_J)}\ch_{\Delta}(M)
\]
Thus the final result is $\ch'_{\textrm{new}}$ discussed at
the end of \S \ref{sec:quick fix}.

\subsection{Indecomposables after the ``more comprehensive fix''} As with
the ``quick fix'', it is not true that the images of the
indecomposable Soergel bimodules ${}^IB^J_p$ under $\widetilde{\ch}$
are self-dual in the Hecke algebroid. However, one may check that
there is a unique shift of each ${}^IB^J_p$ which makes it self-dual
under $\widetilde{D}$, and that the image under $\widetilde{\ch}$ of
this self-dual normalization is self-dual.

Indeed, the proof of \cite[Proposition 6.15]{SSBim} can be modified to
show that the tensor product of two self-dual singular Soergel
bimodules is self-dual (see also the proof of Theorem 7.10). Now, one
checks easily that the bimodules
\begin{gather*}
  {}^I\widetilde{\nabla}^J = {}^IR^J \quad   \text{if $I \subset J$}, \\
{}^I\widetilde{\nabla}^J = {}^IR^J[\ell(w_I)-\ell(w_J)] \quad
\text{if $I \supset J$}
\end{gather*}
are self-dual and have self-dual character. Thus the maximal summand of
any tensor product of
these bimodules for a reduced translation sequence provide the
required self-dual bimodule.

\subsection{Comments} Either of the above two fixes introduces a
left/right asymmetry which is 
perhaps unappealing.

For example, when $W$ is of rank $1$ with simple
reflection $s$ the normalizations for the ``quick fix'' result in the
following normalizations of some of the indecomposable Soergel bimodules
\[
{}^{ \{ s \}} B^{\emptyset}_s = R \quad \text{and} \quad {}^{\emptyset
} B^{\{ s \}}_s = R[1].
\]
The self-dual indecomposables for the ``more comprehensive fix'' are as follows
\[
{}^{ \{ s \}} B^{\emptyset}_s = R[1] \quad \text{and} \quad {}^{\emptyset
} B^{\{ s \}}_s = R.
\]
(Note the lack of left/right symmetry in both cases.)

It does not appear to possible to make things left/right symmetric
without making more drastic changes. One elegant possibility
(suggested by Brundan) is to introduce a shift in the tensor product\footnote{That is, $M
\otimes_J N := M \otimes_{R^J} N [\ell(w_J)]$. With this change, singular Soergel
bimodules would no longer be a full subcategory of bimodules.}. We
leave the consideration of these possibilities to the interested reader.

\subsection{Acknowledgements} I would like to thank Noriyuki Abe, Ben
Elias, Jon Brundan and Leonardo Patimo for pointing out the
error. Abe suggested (a variant of) the definition of $\newch$. The
``more comprehensive fix'' is due to Brundan. A similar fix is present
in forthcoming work of Bodish, Brundan and Elias \cite{BBE}.

\end{document}